\documentclass[12pt,pdftex,a4paper]{amsart}

\selectfont

\usepackage{pdfpages}
\usepackage{amssymb,color, euscript, enumerate}
\usepackage{amsthm}

\usepackage{graphicx}

\usepackage{amsmath}
\usepackage{amscd}
\usepackage{txfonts}
\usepackage{comment}
\usepackage{bm}
\usepackage{pb-diagram}

\usepackage{comment}
\usepackage{amscd}
\usepackage[all]{xy}
\usepackage{wrapfig}

\makeatletter
\@addtoreset{equation}{section}
\makeatother

\newcommand{\F}{F_{0,1,\infty}}

\newcommand{\CP}{\C P^1}
\newcommand{\CPm}{\C P^1 \setminus \{0,1,\infty\}}

\newcommand{\sign}{\mathrm{sign}}

\newcommand{\Z}{\mathbb{Z}}
\newcommand{\Q}{\mathbb{Q}}
\newcommand{\R}{\mathbb{R}}
\newcommand{\C}{\mathbb{C}}
\newcommand{\va}{\bm{1}}
\newcommand{\z}{{\bar{z}}}

\newcommand{\h}{{\bar{h}}}
\newcommand{\p}{{\bar{p}}}

\newcommand{\y}{{\bar{y}}}

\newcommand{\uz}{{\underline{z}}}
\newcommand{\uuz}{{z,\bar{z}}}
\newcommand{\uuzo}{{z_1,\bar{z}_1}}
\newcommand{\uuzt}{{z_2,\bar{z}_2}}
\newcommand{\uuzz}{{z_0,\bar{z}_0}}

\newcommand{\s}{{\bar{s}}}
\newcommand{\td}{{\bar{t}}}
\newcommand{\n}{{\bar{n}}}

\newcommand{\al}{\alpha}
\newcommand{\be}{\beta}

\newcommand{\om}{\omega}
\newcommand{\omb}{{\bar{\omega}}}

\newcommand{\si}{\sigma}

\newcommand{\dz}{\frac{d}{dz}}
\newcommand{\ddz}{\frac{d}{d\z}}

\newcommand{\lat}{{\mathrm{lat}}}
\newcommand{\hY}{\hat{Y}}

\newcommand{\Ld}{{\overline{L}}}
\newcommand{\D}{{\bar{D}}}

\newcommand{\Om}{{\Omega}}

\newcommand{\GCor}{{\mathrm{GCor}}}

\newcommand{\Aut}{\mathrm{Aut}\,}
\newcommand{\genus}{{{I\hspace{-.1em}I}}}
\newcommand{\gen}{{\text{genus}}}

\newcommand{\tw}{{{I\hspace{-.1em}I}_{1,1}}}
\newcommand{\twN}{{{I\hspace{-.1em}I}_{k,k}}}
\newcommand{\End}{\mathrm{End}\,}
\newcommand{\Hom}{\mathrm{Hom}\,}

\newcommand{\fora}{\text{ for any }}

\newcommand{\FHp}{\underline{\text{Full $\mathcal{H}$-VA}}}
\newcommand{\VHp}{\underline{\text{$\mathcal{H}$-VA}}}
\newcommand{\AH}{\underline{\text{even AH pair}}}
\newcommand{\gAH}{\underline{\text{good AH pair}}}
\newcommand{\LP}{\underline{\text{Lattice pair}}}

\newcommand{\OHp}{\underline{\text{G-full VAp}}}
\newcommand{\GVA}{\underline{\text{G-VA}}}
\newcommand{\GFA}{\underline{\text{G-full VA}}}

\newtheorem{thm}{Theorem}[section]

\newtheorem{lem}[thm]{Lemma}
\newtheorem{prop}[thm]{Proposition}
\newtheorem{cor}[thm]{Corollary}
\newtheorem{rem}[thm]{Remark}

\begin{document}

\title[current-current deformation]{current-current deformation}
%
%
\author{YUTO MORIWAKI}


\vspace*{10mm}
\begin{center}
{\LARGE \bf Two-dimensional conformal field theory, full vertex algebra and
current-current deformation
} \par \bigskip

\renewcommand*{\thefootnote}{\fnsymbol{footnote}}
{\normalsize
Yuto Moriwaki \footnote{email: \texttt{moriwaki.yuto (at) gmail.com}}
}
\par \bigskip
{\footnotesize Kavli Institute for the Physics and Mathematics of the Universe, Chiba, Japan}

\par \bigskip
\end{center}


\noindent
\begin{center}
{\large Abstract
}
\end{center}
 \par \bigskip
The main purpose of this paper is a mathematical construction of non-perturbative deformations of two-dimensional conformal field theories.

We introduce the notion of a full vertex algebra which formulates a compact two-dimensional conformal field theory. Then, we construct a deformation family of full vertex algebras
which serves as a current-current deformation of conformal field theory in physics.
The parameter space of the deformations is expressed as a quotient of an orthogonal Grassmannian.
This parameter space is a part of the CFT moduli space expected in physics. We demonstrate that nontrivial quotients of Grassmannians are obtained from holomorphic vertex operator algebras of central charge $24$.


\pagenumbering{roman}

\vspace*{8mm}

\pagenumbering{arabic}

\clearpage

\begin{center}
{\large \bf Introduction
}
\end{center}
 \par \bigskip

In theoretical physics, quantum field theory is the conceptual framework that describes a wide range of objects from the world of elementary particles to the scale of the universe, and its mathematical basis is one of the most important problems in modern mathematics \cite{We,PS,Ha,W}.
%
%
In quantum field theory, deformations of theories are important, since in the case of free field theories, their deformations give phenomenological predictions about the real world.
A deformation is defined by adding 
 a new term to the original Lagrangian $\mathcal{L}(\mathcal{O}_i,\partial_\mu \mathcal{O}_i) \mapsto \mathcal{L}(\mathcal{O}_i,\partial_\mu \mathcal{O}_i)+g \mathcal{O}_k$.
Here $\mathcal{O}_k$ is an additional field, and $g \in \R$ is called a coupling constant (cf., \cite{IZ,Sr}).
A deformed {\it correlation function}, a physical quantity,
can be obtained by perturbation theory, i.e., expanded as a power series in $g$ by using the path-integral.
In most cases, the deformation obtained in this way remains  only an approximation. Therefore, it is unclear whether the deformed theory rigorously satisfies the axiom of quantum field theory. This is one of the difficulties in constructing a new quantum field theory mathematically.

Quantum field theory in higher dimensions is difficult to construct. However, {\it conformal field theory} (quantum field theory with conformal symmetry) in two-dimension has many mathematically rigorous and non-trivial examples \cite{FMS}. It is noteworthy that
two-dimensional conformal field theory
is an interesting object in itself
since it plays a significant role in statistical mechanics \cite{He}, condensed matter physics \cite{Kitae}, and string theory \cite{P} in physics and it is deeply related to 
elliptic genus \cite{Ta}, modular forms \cite{Zh}, infinite-dimensional Lie algebras and sporadic finite simple groups \cite{FLM,B2} in mathematics.

The purposes of this paper are
\begin{enumerate}
\item
to introduce the notion of a {\it full vertex algebra}
which is a mathematical formulation of two-dimensional conformal field theory;
\item
to construct a deformation of a full vertex algebra,
which serves as a deformation of conformal field theory;
\item
to apply the deformation to the classification theory of vertex algebras.
\end{enumerate}
%
%

There are various definitions of quantum field theory, such as axiomatic quantum field theory and conformal net \cite{OS,Ha,Wa}. These theories are based on a functional analysis, and the space of states is assumed to be a Hilbert space. However, a full vertex algebra does not consider such an assumption.
In fact, two-dimensional conformal field theories without unitarity, such as the Yang-Lee model, are examples of full vertex algebras.
Various examples of full vertex algebras are constructed in \cite{M4,M5}.



\noindent
\begin{center}
{0.1. \bf 
Conformal field theory in physics and mathematics
}
\end{center}

First, we briefly recall a formulation of quantum field theory from physics, which is not necessarily two-dimensional.
One aim of quantum field theory is to calculate {\it $n$-point correlation functions}, that is, the vacuum expectation value of an interaction of $n$ particles.
An interaction of $n$ particles decomposes into 
subsequent interactions of three particles.
Thus, an $n$-point correlation function can be expressed in terms of three-point correlation functions, together with a choice of decompositions.
Quantum field theory requires that the resulting $n$-point correlation functions are \emph{independent} of the choice of decompositions.
This principle is known as the {\it consistency of quantum field theory}.
Although it is known to be difficult to construct mathematically rigorous quantum field theories, surprisingly many examples, especially conformal field theories, have been constructed in two-dimension in physics literature (see \cite{FMS}).

In (not necessarily two-dimensional) conformal field theories, it is believed in physics that the whole consistency of $n$-point correlation functions follows from 
the \emph{bootstrap equations (or hypothesis)},
which are distinguished consistencies of four-point correlation functions \cite{FGG,P2}.
This hypothesis was used successfully by Belavin, Polyakov, and Zamolodchikov in \cite{BPZ}, where the modern study of 
two-dimensional conformal field theories was initiated.

Hereafter, we consider two-dimensional conformal field theory.  A field of two-dimensional conformal field theory is an operator-valued real analytic function.
A conformal field theory in which any field is holomorphic is called a {\it chiral conformal field theory.}
It is noteworthy that the algebra of a chiral conformal field theory satisfies a purely algebraic set of axioms, which was introduced by Borcherds \cite{B} (see also \cite{G}).
It is called a {\it vertex algebra} or a {\it vertex operator algebra} \cite{FLM}
and has been studied intensively by many authors, e.g., \cite{LL,FHL,FB}.
In contrast, a formulation of the algebra of a non-chiral conformal field theory needs analytic properties and seems impossible to describe in a purely algebraic way.

Moore and Seiberg constructed a non-chiral conformal field theory as an extension of a chiral and an anti-chiral (anti-holomorphic) vertex operator algebras by their modules \cite{MS1,MS2}.
The bootstrap equations, in this case, are translated as a monodromy invariant property of the four-point correlation functions. In the physics literature, this property was reformulated later by Fuchs, Runkel, and Schweigert in \cite{FRS}, which says that the algebra describing the conformal field theory is a Frobenius algebra object in the braided tensor category constructed from chiral and anti-chiral vertex operator algebras.

A mathematical approach in this direction is due to Huang and Kong \cite{HK} based on the representation theory of a {\it regular vertex operator algebra} developed by Huang and Lepowsky in a series of papers \cite{HL1,HL2,HL3,H1,H2} (see also \cite{M6} for a different method using conformal blocks).
A regular vertex operator algebra is a class of vertex operator algebras with a semisimple module category (all the representations are completely reducible).
One of the noticeable results is obtained by Huang, which states that the representation category of a regular vertex operator algebra (of strong CFT type) inherits a modular tensor category structure \cite{H3,H4}.

Based on this theory, Huang and Kong \cite{HK} introduced the notion of a {\it full field algebra}, which is a mathematical axiomatization of the algebras describing non-chiral two-dimensional conformal field theory. They also constructed conformal field theories, called \emph{diagonal theories} in physics, as finite module extensions of the tensor products of regular vertex operator algebras.
Their theory basically assumes that the conformal field theory is a finite extension of a tensor product of chiral and anti-chiral regular vertex operator algebras.
Such a conformal field theory is called a {\it rational conformal field theory}, and it is known that the energy spectrum of the theory is a subset of the rational numbers. 
However, when considering a deformation of a theory, the energies must change continuously;
thus it is necessary to consider irrational conformal field theories. 

\noindent
\begin{center}
{0.2. \bf 
Full vertex algebra -- a formulation of compact conformal field theory
}
\end{center}

In this paper, we introduce \footnote{This is the author's doctoral dissertation and is a reorganization of \cite{M3}, on which this paper is based, and a part of \cite{M2}, from which the introduction of a full vertex algebra was taken.} the notion of a full vertex algebra (and a full vertex operator algebra) which formulates compact two-dimensional conformal field theory on $ \CP$.
While the definition of a full field algebra by \cite{HK} is based on a part of the consistency of $n$-point correlation functions for all $n \geq 1$, the definition of a full vertex algebra is based on  ``the bootstrap equations'', which are expected to be sufficient to derive the whole consistency of the theory.

We note that in recent years, the bootstrap hypothesis has
become increasingly important in the study of conformal field theory, including higher-dimensional cases.
An infinite number of inequalities can be obtained from the bootstrap equation for a unitary conformal field theory, which is a constraint on the existence of the theory.
By numerically evaluating the constraint conditions, the critical exponents (physical quantities) of the three-dimensional critical Ising model are calculated with high accuracy (cf., \cite{RRTV,EPPRSV}).
In \cite{M2}, we prove that the axiom of a full vertex algebra is equivalent to the bootstrap equation
under reasonable assumptions.

%

A crucial point of our definition is
introducing a class of real analytic functions on $\CPm$
with certain possible singularities at $\{0,1,\infty\}$, which we call {\it conformal singularities}.
Roughly speaking, a function with a conformal singularity at $0$ has the following expansion around $z=0$,
\begin{align}
\sum_{r \in \R}\sum_{n,m \geq 0} a_{n,m}^r z^n \z^m |z|^r, \label{eq_CS}
\end{align}
where $|z|=z\z$, the square of the absolute value, and $a_{n,m}^r \in \C$. This series is assumed to be absolutely convergent in an annulus $0<|z|<R$ (for the precise definition, see Section \ref{sec_singularity}).
A typical example of such a function on $\CP$
is $|z|^r$ ($r\in \R$), which has the conformal singularities at $\{0,\infty\}$.
Another example is 
\begin{align}
f_{\mathrm{Ising}}(z)=\frac{1}{2}(|1-\sqrt{1-z}|^{1/2}+|1+\sqrt{1-z}|^{1/2}), \label{Ising}
\end{align}
which appears as a four point function of the two-dimensional critical Ising model 
(for construction of the full vertex algebra corresponding to this conformal field theory, see \cite{M4}).
The expansion of $f_{\mathrm{Ising}}(z)$ at $z=0$ is
\begin{align*}
1+|z|^{1/2}/4-z/8-\z/8+|z|^{1/2}(z+\z)/32+ z\z/64-5z^2/128-5\z^2/128+\dots.
\end{align*}

By using the notion of a conformal singularity,
we introduce a space of real analytic functions
on $Y_2=\{(z_1,z_2)\in \C^2\;|\; z_1\neq z_2, z_1\neq 0,z_2 \neq 0 \}$
which has possible conformal singularities along $z_1=0,z_2=0,z_1=z_2$
and denote it by $\GCor_2$ (see Section \ref{sec_gen}).

Let us describe the precise definition of a full vertex algebra.
For a vector space $V$, let $V[[z,\z,|z|^\R]]$ be a space of formal power series spanned by $$\sum_{r \in \R} \sum_{n,m \geq 0} v_{n,m}^r z^n\z^m |z|^r,$$
where $v_{n,m}^r \in V$
and $V((z,\z,|z|^\R))$ a subspace of $V[[z,\z,|z|^\R]]$ consisting of formal power series which are bounded below and discrete (see Section \ref{sec_formal}).
A full vertex algebra is an $\R^2$-graded vector space $F=\bigoplus_{h,\h\in \R} F_{h,\h}$ with a distinguished vector $\va \in F_{0,0}$ and a linear map
$$Y(-,\uuz):F\rightarrow \End F[[z,\z,|z|^\R]],\; a\mapsto Y(a,\uuz)=
\sum_{r,s\in \R}a(r,s)z^{-r-1}\z^{-s-1}$$
satisfying the following axioms:
\begin{enumerate}
\item[FV1)]
For any $a,b\in F$, $Y(a,\uuz)b \in F((z,\z,|z|^\R))$;
\item[FV2)]
$F_{h,\h}=0$ unless $h-\h \in \Z$;
\item[FV3)]
For any $a \in F$, $Y(a,\uuz)\va \in F[[z,\z]]$ and $\lim_{z \to 0}Y(a,\uuz)\va = a(-1,-1)\va=a$;
\item[FV4)]
$Y(\va ,\uuz)=\mathrm{id}_F$;
\item[FV5)]
For any $a,b,c \in F$ and $u \in F^\vee=\bigoplus_{h,\h\in \R} F_{h,\h}^*$,
there exists $\mu(z_1,z_2) \in \GCor_2$ such that
\begin{align}
u(Y(a,\uuzo)Y(b,\uuzt)c)&=\mu(z_1,z_2)|_{|z_1|>|z_2|},\nonumber \\
u(Y(Y(a,\uuzz)b,\uuzt)c)&=\mu(z_0+z_2,z_2)|_{|z_2|>|z_0|}, \label{eq_Borcherds} \\
u(Y(b,\uuzt)Y(a,\uuzo)c)&=\mu(z_1,z_2)|_{|z_2|>|z_1|},\nonumber
\end{align}
where $F_{h,\h}^*$ is the dual of $F_{h,\h}$
and $\mu(z_1,z_2)|_{|z_1|>|z_2|}$ is the expansion of $\mu(z_1,z_2)$ in
$\{|z_1|>|z_2|\}$;
\item[FV6)]
$F_{h,\h}(r,s)F_{h',\h'}\subset F_{h+h'-r-1,\h+\h'-s-1}$ for
any $h,h',\h,\h',r,s\in\R$.
\end{enumerate}

Let us explain the physical background of this definition.
All the states of a conformal field theory
form a vector space, which is $F$ in our definition.
The global conformal symmetry 
$\mathrm{SO}(3,1)$ acts on $F$.
The $\R^2$-grading on $F$ is induced from this action, and the assumptions (FV3), (FV4), and (FV6) are natural requirements which conformal field theory satisfies.
For a vector $v \in F_{h,\h}$, the value $h+\h$ and $h-\h$
are physically the energy and the spin of a state $v$.
A state $v$ changes as $\exp(i\theta (h-\h))$ under the rotation group $\mathrm{SO}(2) \subset \mathrm{SO}(3,1)$,
which requires the assumption (FV2) (if the theory does not contain fermions). Although (FV1) and (FV5) are not satisfied by general conformal field theories, they are satisfied by a wide class of conformal field theories, called {\it compact conformal field theories}.

In this paper, a compact conformal field theory is a conformal field theory whose state space $F$ satisfies the following conditions:
\begin{enumerate}
\item[C1)]
There exists $N \in \R$ such that $F_{h,\h}=0$ for any $h\leq N$ or $\h \leq N$;
\item[C2)]
For any $H\in \R$, $\sum_{\substack{h,\h\in \R\\h+\h \leq H}} \dim F_{h,\h}$ is finite.
\end{enumerate}
We also call a full vertex algebra $F$ {\it compact} if it satisfies (C1) and (C2) (Note that this definition of compactness is a bit different from the definition used in physics).

For non-compact conformal field theory, the correlation functions are no longer power series of the form \eqref{eq_CS} but an integral over $\R^2$. Thus, in this paper, we restrict ourselves to compact conformal field theory to avoid difficulties in analysis.
There are important non-compact conformal field theories, e.g., the Liouville field theory and non-compact WZW conformal field theory \cite{DO,ZZ}. We hope to come back to this point in the future.
%
%

Now, we explain the physical meaning of (FV1) and (FV5) from the compactness.
(FV1) is a mathematical consequence of (C), (FV2), and (FV6)  (see Proposition \ref{lem_grading_add}),
thus is satisfied for any compact conformal field theory.
Furthermore, (FV1) and the bootstrap equation imply (FV5).
Therefore, the notion of a compact full vertex algebra gives a mathematical formulation of two-dimensional compact conformal field theory on $\CP$. In particular, any correlation function of compact conformal field theory always has an expansion of the form \eqref{eq_CS}, which is our motivation for defining a conformal singularity.

As discussed in Section 0.1, rational conformal field theory contains many important conformal field theories, but it is too restrictive to consider deformations. Compact conformal field theory is a wider class of conformal field theory which includes rational conformal field theory,
e.g., the WZW-model for a compact semisimple Lie group and the Virasoro minimal models.
Furthermore, by the definition of the compactness, at least its small deformation seems to be compact.
In particular, we prove under some mild assumption compactness is preserved by the current-current deformation constructed in this paper (for the precise statement, see Proposition \ref{positive_compact}). We expect that (unitary) compact conformal field theory is stable under all exactly marginal deformations.

%

Finally, as expected in physics, a chiral conformal field theory (vertex algebra) naturally appears as a subalgebra of a full vertex algebra.
In fact, the chiral subspace $\ker \frac{d}{d\z}\bigl|_F$ of a full vertex algebra $F$ forms a vertex algebra,
and $F$ is a module on $\ker \frac{d}{d\z}\bigl|_F$ (see Proposition \ref{vertex_algebra}).
Hence the full vertex algebra $F$ can be seen as an extension of the tensor product of chiral and anti-chiral vertex algebras $\ker \frac{d}{d\z} \bigl|_F \otimes \ker \frac{d}{dz} \bigl|_F$ (Proposition \ref{ker_hom}).
This is an assumption in the study of Huang and Kong.
We note that the notions of a full field algebra \cite{HK} and a full vertex algebra are equivalent
\footnote{Since a full vertex algebra satisfies {\it the associativity} and {\it the skew-symmetry} (Proposition \ref{translation}), if $\ker \frac{d}{d\z} \bigl|_F$ and $\ker \frac{d}{dz} \bigl|_F$ are regular VOAs, then $F$ is a full field algebra by \cite[Theorem 2.11]{HK}. Conversely, given a full field algebra $F$ that is an extension of regular vertex operator algebras, we see from \cite[Theorem 2.11]{HK} and \cite[Proposition 4.3]{M2} that $F$ gives a full vertex algebra.}
 if the algebra is an extension of a tensor product of chiral and anti-chiral regular vertex operator algebras.

%
%
%

\noindent
\begin{center}
{0.3. \bf 
Current-current deformation in physics and its formulation
}
\end{center}

Now, we briefly review a deformation of two-dimensional conformal field theory in physics.
The deformation of two-dimensional conformal field theory $F=\bigoplus_{h,\h\in \R}F_{h,\h}$ generated by a general field $O_k \in F_{h,\h}$ does not always preserve the conformal symmetry.
This general deformation has been studied by many physicists, e.g., \cite{Z,EY} to understand the structure of quantum field theories.
Meanwhile, a deformation of a two-dimensional conformal field theory which preserves the conformal symmetry is known to be generated by a special field $O_k \in F_{1,1}$, called an (exactly) {\it marginal field} \cite{DVV}.

Chaudhuri and Schwartz considered the deformation of a conformal field theory generated by
a field in $F_{1,0}\otimes F_{0,1} \subset F_{1,1}$ (a sum of products of chiral currents and anti-chiral currents).
They showed that the field is exactly marginal if and only if
the chiral currents, as well as the anti-chiral currents, belong to {\it commutative current algebras} \cite{CS}. The deformation generated by this $(1,1)$-field is called a {\it current-current deformation} in the physics literature (cf., \cite{FR}).
Those studies depend on the path integral method,
which is not mathematically rigorous.
This paper aims to mathematically formulate and construct the current-current deformation of two-dimensional conformal field theory.

In terms of a full vertex algebra, the commutative current algebra which generates a current-current deformation 
corresponds to a subalgebra of a full vertex algebra which is isomorphic to the tensor product of
chiral and anti-chiral Heisenberg vertex algebras.

It is convenient to introduce the notion of a {\it full $\mathcal{H}$-vertex algebra}.
Let $H_l$ and $H_r$ be real vector spaces equipped with non-degenerate bilinear forms $(-,-)_l:H_l\times H_l\rightarrow \R$
and $(-,-)_r:H_r\times H_r\rightarrow \R$
and $M_{H_l}(0)$ and $M_{H_r}(0)$ be the affine Heisenberg vertex algebras associated with $(H_l,(-,-)_l)$ and $(H_r,(-,-)_r)$, respectively.
Set $H=H_l \oplus H_r$ and let $p,\p \in \End H$ be the projections of $H$ onto $H_l$ and $H_r$,
$(H,(-,-)_p)=(H_l\oplus H_r,(-,-)_l\oplus (-,-)_r)$
the orthogonal sum of vector spaces
and
$$
M_{H,p}=M_{H_l}(0)\otimes \overline{M_{H_r}(0)}
$$
the tensor product of the vertex algebra $M_{H_l}(0)$
and the anti-chiral vertex algebra $\overline{M_{H_r}(0)}$.
A {\it full $\mathcal{H}$-vertex algebra}
is a full vertex algebra $F$
together with a full vertex algebra homomorphism
$M_{H,p} \rightarrow F$.
Since $F$ is an $M_{H,p}$-module,
$F$ is a module of the affine Heisenberg Lie algebra $\hat{H}$
associated with $(H,(-,-)_l\oplus (-,-)_r)$.
For $\al \in H$,
set
$$
\Om_{F,H}^\al=\{v\in F\;|\;
h(n)v=0, h(0)v=(h,\al)_p v \fora h \in H \text{ and }n\geq 1\}
$$ and $\Om_{F,H}=\bigoplus_{\al \in H} \Om^\al$.
The lowest weight space $\Om_{F,H}$ is called a {\it vacuum space}
in \cite[Section 1.7]{FLM}.
We assume that the vacuum space $\Om_{F,H}$ generates $F$ as an $\hat{H}$-module,
that is,
\begin{align}
F \cong \bigoplus_{\al \in H} M_{H,p}(\al) \otimes \Om_{F,H}^\al.
\label{eq_decomposition}
\end{align}
Then, as suggested by F\"{o}rste and Roggenkamp in \cite{FR},
$\Om_{F,H}$ inherits an algebra structure by modifying the full vertex algebra structure on $F$.
More precisely, we introduce the notion of a {\it generalized full vertex algebra},
which is in fact a mathematical formulation of the above ``structure of the lowest weight space''.
Then, we show that $\Om_{F,H}$ is a generalized full vertex algebra (Theorem \ref{vacuum_space}).
Before stating the main results,
we briefly explain the definition of a generalized full vertex algebra,
which plays a crucial role in this paper.

\noindent
\begin{center}
{0.4. \bf Generalized full vertex algebras.
}
\end{center}

The notion of a generalized full vertex algebra
is a ``full'' analogy of the notion of a (chiral) generalized vertex algebra
introduced by Dong and Lepowsky \cite{DL},
in order to study the affine vertex algebras and the parafermion vertex algebras
\cite{DL}.

We first recall their results. Let $\mathfrak{g}$ be a simple Lie algebra
and $L_{\mathfrak{g},k}$ the simple affine vertex algebra at level $k$.
Then, $L_{\mathfrak{g},k}$ has a Heisenberg vertex subalgebra generated by a Cartan subalgebra of the Lie algebra,
$H_{\mathfrak{g}}\subset \mathfrak{g}$.
Thus, $(L_{\mathfrak{g},k},H_{\mathfrak{g}})$ is a chiral full $\mathcal{H}$-vertex algebra, which we call an
{\it $\mathcal{H}$-vertex algebra}.
Dong and Lepowsky showed that if $k \in \Z_{\geq 0}$, called an integrable level, the vacuum space $\Om_{L_{\mathfrak{g},k}, H_{\mathfrak{g}}}$ inherits a generalized vertex algebra structure \cite[Theorem 14.16]{DL}.
They also constructed a generalized vertex algebra from a pair of a real finite-dimensional vector space $H$ equipped with a non-degenerate symmetric bilinear form and an abelian subgroup $L \subset H$. They call it a {\it generalized lattice vertex algebra.}

We remark that 
our proof of the existence of a generalized full vertex algebra
structure on $\Om_{F,H}$ (Theorem \ref{vacuum_space}) seems different from \cite[Theorem 14.16]{DL}.
Since any $\Z$-graded vertex algebra is a full vertex algebra (Proposition \ref{graded_vertex}),
Theorem \ref{vacuum_space} generalizes their results to any vertex algebras, in particular, to the affine vertex algebras at any level $k \in \R$.
In fact, we prove that the category of generalized vertex algebras
and the category of $\mathcal{H}$-vertex algebras are equivalent (Proposition \ref{vacuum_vertex}).

A generalized full vertex algebra is,
roughly, an $H$-graded vector space $\Om=\bigoplus_{\al \in H}\Om^\al$
equipped with a linear map
$$
\hY(-,\uuz):\Om \rightarrow \End\Om[[z^\R,\z^\R]],\;
a \mapsto \hY(a,\uuz)=\sum_{r,s \in \R}a(r,s)z^{-r-1}\z^{-s-1},
$$
where $H$ is a finite-dimensional vector space
equipped with a non-degenerate symmetric bilinear form.
The critical point is that we allow the 
correlation function for $\al_i\in H$ and $a_i \in \Om^{\al_i}$
to have a $U(1)$-monodromy of the form
$\exp(2\pi (\al_i,\al_j))$ under the interchange of states $a_i$ and $a_j$
 (for the precise definition, see Section \ref{sec_generalized}).
If the monodromy is trivial,
then a generalized full vertex algebra is a full vertex algebra (Lemma \ref{even_omega}).

Thus, a fundamental question is
whether it is possible to cancel the monodromy
for a given generalized full vertex algebra.
The answer is yes.
Let $\Om$ be a generalized full vertex algebra  graded by $H$
and $P(H)$ the set of projections $p \in \End H$ such that
the subspaces $\ker p$ and $\ker (1-p)$ is orthogonal.
Then, we can construct a full vertex algebra for each $p \in P(H)$
by canceling the monodromy (Theorem \ref{construction}).
In fact, we have a family of full $\mathcal{H}$-vertex algebras
parametrized by $P(H)$.
Each element of $P(H)$ determines the charge of the decomposition \eqref{eq_decomposition}.

\noindent
\begin{center}
{0.5. \bf Main results
}
\end{center}
Before stating the main result, we explain how the $U(1)$-monodromies on the vacuum space appear.
Let $(F,H,p)$ be a full $\mathcal{H}$-vertex algebra and $\al_1,\al_2 \in H$.
Then, the conformal block (or the correlation function) of the affine Heisenberg full vertex algebra $M_{H,p}$ labeled by $\al_1,\al_2$
is of the form
$$
(z_1-z_2)^{(p\al_1,p\al_2)_l}(\z_1-\z_2)^{(\p\al_1,\p\al_2)_r}
=|z_1-z_2|^{(\p\al_1,\p\al_2)_r}(z_1-z_2)^{(p\al_1,p\al_2)_l-(\p\al_1,\p\al_2)_r},
$$
where $|z_1-z_2|$ is the square of the absolute value $(z_1-z_2)(\z_1-\z_2)$.
The above
$|z_1-z_2|^r$ is a single-valued function for any $r\in\R$.
Thus, the monodromy of the conformal block is controlled by
the bilinear form $(-,-)_\lat$ on $H$ defined by
$(\al_1,\al_2)_\lat=
(p\al_1,p\al_2)_l-(\p\al_1,\p\al_2)_r.$
We denote the space $(H,(-,-)_\lat)$ by
$H_l\oplus -H_r$.
Then, the first main result of this paper is
that the assignment $(F,H,p) \mapsto (\Om_{F,H},H_l\oplus -H_r, p)$
gives an equivalence between 
the category of full $\mathcal{H}$-vertex algebras
and the category of generalized full vertex algebras with the charge structure $p$ (Theorem \ref{equivalence}).

The real orthogonal group $O(H_l\oplus-H_r;\R)$ acts on 
the set of all the possible charge structures
$P(H_l\oplus -H_r)$ and the orbit of the original projection $p$
forms the orthogonal Grassmannian 
$O(H_l\oplus -H_r;\R)/O(H_l;\R) \times O(-H_r;\R)$,
which is a connected component of $P(H_l\oplus -H_r)$.
Thus, by using the inverse functor,
we have a family of full $\mathcal{H}$-vertex algebras parametrized by 
the Grassmannian.

We note that for $h_l \in H_l$ and $h_r \in H_r$ with $(h_l,h_l)\neq 0$ and
$(h_r,h_r)\neq 0$,
we have a one-parameter subgroup $\{\si(g)\}_{g\in \R} \subset 
O(H_l\oplus -H_r)$ (see Section \ref{sec_physics}).
The family of full $\mathcal{H}$-vertex algebras associated with $\{\si(g)p \si(g)^{-1} \}_{g \in \R} \subset P(H_l\oplus -H_r)$
is, in fact, the current-current deformation of a full $\mathcal{H}$-vertex algebra $(F,H,p)$ associated with the exactly marginal field $Y(h_l(-1,-1)h_r,\uuz)=h_l(z)h_r(\z)$.
Thus, the above family gives a mathematical formulation of the non-perturbative current-current deformation associated with the commutative
current algebras $H_l$ and $H_r$.

Finally, we give the double coset description of the parameter space.
The automorphism group of
the generalized full vertex algebra $\Om_{F,H}$
naturally acts on the grading $H_l\oplus -H_r$.
Let $D_{F,H}$ be the image of the automorphism group 
in $O(H_l\oplus -H_r)$.
Then, the isomorphism classes of the current-current deformation
of a full $\mathcal{H}$-vertex algebra $(F,H,p)$ are parametrized by the double coset (Theorem \ref{moduli})
\begin{align}
D_{F,H} \backslash O(H_l\oplus -H_r) / O(H_l) \times O(-H_r),
\label{eq_double}
\end{align}
which is conjectured in \cite{FR}.
Thus, $D_{F,H}$ is a mathematical formulation of the duality group,
which implies the T-duality of string theory (see below).

For example, let $F_{\mathrm{SU}(2)}$ be a full vertex algebra
corresponding to the $\mathrm{SU}(2)$ WZW model at level one.
Then, $F_{\mathrm{SU}(2)}$ is a full $\mathcal{H}$-vertex algebra by one-dimensional Cartan subalgebras of $\mathrm{SU}(2)$. 
Since $O(1,1) / O(1) \times O(1) \cong \R_{>0}$, the current-current deformation of $F_{\mathrm{SU}(2)}$ is
parametrized by $R \in \R_{>0}$. Let denote $C_R$ the full $\mathcal{H}$-vertex algebra corresponding to $R \in \R_{>0}$.
The algebra structure of $C_R$ will be studied in detail in Section \ref{sec_toroidal}.
As mentioned in Section 0.2, the chiral and the anti-chiral parts of $C_R$ are vertex operator algebras.
If the square $R^2$ is an irrational number, then both the chiral and the anti-chiral parts
are Heisenberg vertex operator algebras, and $C_R$ defines an irrational conformal field theory.
If $R^2= p / q $ with $p,q \in \Z_{>0}$ are coprime integers,
then both the chiral and the anti-chiral parts are isomorphic to the lattice vertex operator algebra $V_{\sqrt{2pq}\Z}$ and $C_R$ is a finite extension of $V_{\sqrt{2pq}\Z}\otimes \overline{V_{\sqrt{2pq}\Z}}$,
where $\sqrt{2pq}\Z$ is the rank one lattice generated by $\al$ with $(\al,\al)=2pq$
and $\overline{V_{\sqrt{2pq}\Z}}$ is an anti-chiral vertex operator algebra (see Proposition \ref{conjugate}).
For example, the full vertex algebras $C_R$ with $R=\sqrt{6}$ or $R=\sqrt{3/2}$ have
 the same underlying lattice vertex algebra $V_{\sqrt{12}\Z}$.
However, $C_{\sqrt{6}}$ and $C_{\sqrt{3/2}}$ are non-isomorphic.
The decomposition of $C_{\sqrt{6}}$ and $C_{\sqrt{3/2}}$
into irreducible $V_{\sqrt{12}\Z}\otimes \overline{V_{\sqrt{12}\Z}}$-modules are
\begin{align*}
C_{\sqrt{6}} &= \bigoplus_{i \in \Z /12\Z} 
V_{\sqrt{12}\Z+\frac{i}{\sqrt{12}}} \otimes
 \overline{V_{\sqrt{12}\Z+\frac{i}{\sqrt{12}}}}\\
C_{\sqrt{3/2}} &= \bigoplus_{i \in \Z /12\Z} V_{\sqrt{12}\Z+\frac{i}{\sqrt{12}}} \otimes \overline{V_{\sqrt{12}\Z+\frac{7i}{\sqrt{12}}}}.
\end{align*}
Thus, while $C_{\sqrt{6}}$ is a diagonal sum of irreducible modules of $V_{\sqrt{12}\Z}$,
$C_{\sqrt{3/2}}$ is twisted by $7 \in (\Z/\Z_{12})^\times$.
The general twist $n_{p,q} \in (\Z/2pq\Z)^\times$ for $R^2= p/q$ is given in Proposition \ref{twist},
which corresponds to an automorphism of the modular tensor category $\mathrm{Rep}V_{\sqrt{2pq}\Z}$.
In this way, the rational conformal field theory $C_R$ with $R^2\in \mathbb{Q}$ is controlled by a number-theoretic discrete structure,
and the irrational conformal field theory connects them continuously.

\begin{wrapfigure}{c}{0.3\textwidth}
\centering
\includegraphics[scale=0.3,width=0.3 \textwidth]{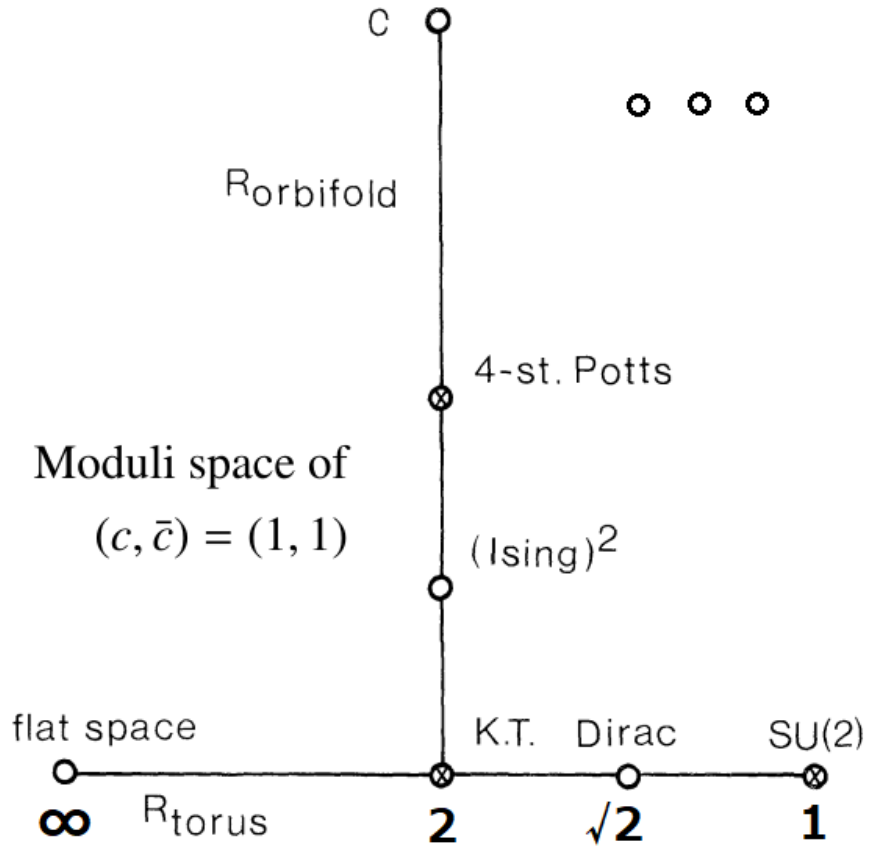}
\caption{\\ \tiny{CFT moduli space of $(c,\bar{c})=(1,1)$}}
\end{wrapfigure}

It is noteworthy that $C_R$ and $C_{R'}$ are isomorphic if and only if $R=R'$ or $R=\frac{1}{R'}$,
which just corresponds to the action of the duality group $ D_{F_{\mathrm{SU}(2)}} \cong D_4$ (the dihedral group)
on $O(1,1) / O(1) \times O(1) \cong \R_{>0}$.
The double coset $D_4 \backslash \mathrm{O}(1,1;\R) / \mathrm{O}(1;\R) \times \mathrm{O}(1;\R)$ is a half line $[1,\infty)$. This corresponds to the horizontal line in the moduli space of conformal field theories of central charge $ (c, \bar{c}) = (1,1)$ expected in physics (see Fig.$1$, \cite{Gi,DVV,DVV2}).
The line also corresponds to a family of conformal field theories resulting from a compactification of string theory
whose target space is the cycle $S_R^1 = \R /R \R$ with a radius $R \in \R_{>0}$, and the group $D_ {F}$ generalizes the T-duality $R \leftrightarrow R ^ {-1}$ in string theory.
We note that there is a conjectured central charge $ (c, \bar{c}) = (1,1)$ conformal field theory which does not belong to Fig 1 \cite{RW}, however, if we restrict ourselves to compact conformal field theory, then such models seem to be excluded.

In Section \ref{sec_adjoint2}, we construct a family of full $\mathcal{H}$-vertex operator algebras which corresponds to
the toroidal compactification of string theory with $N$-dimensional target space, called a Narain moduli space \cite{N,NSW},
parameterized by the following double coset:
\begin{align}
\mathrm{O}(N,N;\Z) \backslash \mathrm{O}(N,N;\R) / \mathrm{O}(N;\R) \times \mathrm{O}(N;\R).
\label{eq_string}
\end{align}
Thus, the double coset description \eqref{eq_double} gives global information about the moduli space of conformal field theories, which is important in the study of string theory. We remark that recently the moduli space of conformal field theories
is also of interest in the context of three-dimensional gravity, where a random ensemble of conformal field theory seems to be important and Maloney and Witten considered an integral over the Narain moduli space \eqref{eq_string} \cite{MW}.
We hope our results will motivate further studies of the CFT moduli spaces.


We also remark that although this paper treats only free theory as an example of a full vertex algebra, the results of this paper can be applied to any compact conformal field theory (a full vertex algebra).
The current-current deformation of non-free full vertex algebras is discussed in \cite{M4}.

\noindent
\begin{center}
{0.6. \bf Application to classification of vertex operator algebras.
}
\end{center}

One of the key problems in vertex operator algebras (VOAs) is the classification of vertex operator algebras with good representation theory.
A holomorphic VOA is a simple vertex operator algebra whose any module is completely reducible and any irreducible module is isomorphic to the adjoint module. The problem of classifying holomorphic VOAs has a long history
(``holomorphic'' here has nothing to do with real analytic/holomorphic in the context of full/chiral).

If a VOA is holomorphic, the central charge is a
non-negative integer multiple of $8$, and if the central charge is $0,8,16$, the classification is complete \cite{Zh,DM}.
There are expected to be $71$ VOAs of central charge $24$ \cite{Sc}, whose existence and uniqueness have been proved 
if VOAs have non-zero Lie algebras in the degree $1$ subspaces (These are results of many researchers, see for example the introductions of \cite{Ho,MSc,CLM,ELMS,BLS}).

Originally, these $71$ holomorphic VOAs were constructed one by one, but H\"{o}hn proposed a more systematic construction by classifying them into $12$ types \cite{Ho} (see also \cite{La}).
Inspired by the papers of H\"{o}hn and Scheithauer \cite{HS,HS2},
we introduced an equivalence class for general $\mathcal{H}$-vertex algebras (not necessarily holomorphic VOAs), which we call a {\it genus of $\mathcal{H}$-vertex algebras} \cite{M1}.
All $\mathcal{H}$-vertex algebras in the same genus can be systematically constructed.
In the case of holomorphic VOAs of central charge $24$,
the $12$ types given by H\"{o}hn and the genera of $\mathcal{H}$-vertex algebras are the same (see Proposition \ref{prop_genera_hol}).
This means that from the $12$ conformal vertex algebras of central charge $26$ (a class of vertex algebras broader than VOAs), we can construct all $71$ holomorphic VOAs of central charges $24$ \cite[Theorem 4.5]{M1}.
Although the notion of a genus of $\mathcal{H}$-vertex algebras was introduced independently of current-current deformations, they are closely related to each other. Indeed, we have shown that from 12 families of the current-current deformations of full VOAs of central charge $(25,1)$, we can obtain all the $71$ chiral holomorphic VOAs of central charge $24$ (Proposition \ref{prop_genera_hol}). Thus, deformations of full VOAs are useful for classifying  and constructing chiral VOAs. This is explained in more detail below.

A chiral vertex algebra cannot be deformed, but by tensoring a full vertex algebra, we can consider a deformation that mixes the chiral part  and the anti-chiral part. We can then consider the classification of vertex algebras up to deformations.
In this paper, we consider this problem in the case of tensoring a full vertex algebra $F_{\mathrm{SU}(2)}$.
$\mathcal{H}$-vertex algebras $V,W$ are said to be equivalent if
there exits a current-current deformation between the full $\mathcal{H}$-vertex algebras $V\otimes F_{\mathrm{SU}(2)}$ and $W\otimes F_{\mathrm{SU}(2)}$.
We have shown that this equivalence relation is the same as the genus of $\mathcal{H}$-vertex algebras introduced in \cite{M1} (Theorem \ref{genus}), and thus, the H\"ohn's genus (Proposition \ref{prop_genera_hol}).

If $V$ is a VOA of central charge $c$, then $V\otimes F_{\mathrm{SU}(2)}$ is a full VOA of central charge $(c+1,1)$.
Thus, if two holomorphic VOAs of central charge $24$ are in the same genus,
the corresponding full VOAs are in the same connected component of the moduli space of conformal field theory of central charge $(25,1)$ (the conjectured CFT moduli space of central charge $(1,1)$ in Fig 1 has four connected components).


A mathematical definition of the CFT moduli space does not yet exist, what we have just described suggests that defining and constructing the CFT moduli space is important for constructing and classifying chiral/full vertex algebras.
The parameter space of current-current deformations (a subspace of the CFT moduli space) has the form of a double coset as in \eqref{eq_double}  and \eqref{eq_string}.
Using holomorphic VOA of central charge $24$, non-trivial duality groups are obtained and thus interesting quotients of the Grassmannians. We will illustrate this point at the end.

Let $V_B$ (resp. $V_G$) be a holomorphic VOA of central charge $24$
and of type B (resp. of type G) in H\"{o}hn's notion \cite{Ho}.
By a result of \cite{BLS}, the duality groups of $V_{B}\otimes F_{\mathrm{SU}(2)}$ and $V_G \otimes F_{\mathrm{SU}(2)}$
 are obtained and the parameter spaces of the current-current deformations are of the form (Theorem \ref{prop_new_grass}):
\begin{align}
\Aut \genus_{17,1}(2_{\genus}^{+10}) \backslash O(17,1;\R) /O(17;\R) \times O(1;\R),
\label{eq_new_lattice}
\end{align}
for $V_{B}\otimes F_{\mathrm{SU}(2)}$ and
\begin{align}
\Aut \genus_{9,1}(2_\genus^{+6}3^{-6}) \backslash O(9,1;\R) /O(9;\R) \times O(1;\R)
\label{eq_new_lattice2}
\end{align}
for $V_{G}\otimes F_{\mathrm{SU}(2)}$.
Here, $\genus_{17,1}(2_{\genus}^{+10})$ and $\genus_{9,1}(2_\genus^{+6}3^{-6})$ are some even lattices of signature $(17,1)$ and $(9,1)$, respectively.
Note that this result does not depend on the choice of $V_B$, since all holomorphic VOAs of type B are included in \eqref{eq_new_lattice}.
(The same is true for type G.)

Comparing with \eqref{eq_string}, it is interesting to note that here nontrivial lattices appear instead of unimodular lattices.
The moduli spaces \eqref{eq_new_lattice} and \eqref{eq_new_lattice2} are new in physics, and it is a future task to investigate the CFT moduli spaces from the studies of the automorphism groups of holomorphic VOAs (for example \cite{BLS}).

\noindent
\begin{center}
{\bf Outline.
}
\end{center}

The definition of a full vertex algebra is an analogue of the definition of a vertex algebra, and therefore we review the definition of a vertex algebra in a way that makes the analogy clearer in Section 1.
In Section 2, we examine a space of formal power series of two variables needed to define a full vertex algebra.
In Section 3, we introduce the notion of a full vertex algebra and study its properties.
In Section 4, we introduce the notion of a generalized full vertex algebra, construct a standard example and tensor product
and prove Theorem \ref{construction} by canceling the monodromies.
The notion of a full $\mathcal{H}$-vertex algebra is introduced in Section 5.
We show that the vacuum space inherits a generalized full vertex algebra structure
(Theorem \ref{vacuum_space})
and the equivalence of the categories (Theorem \ref{equivalence}).
We also construct some adjoint functors, which will be used later.
Combining the above results,
the current-current deformation of a full $\mathcal{H}$-vertex algebra is defined, and the double coset description
of the parameter space is proved (Theorem \ref{moduli}) in Section 6.
As an application, we study the relation between the current-current deformation of $\mathcal{H}$-vertex algebras
and the genus of vertex algebras and compute the duality group for non-trivial example in Section 7.

\tableofcontents

\section{From chiral vertex algebra to full vertex algebra}

Scale invariance (conformal invariance) is a key concept in vertex algebras and full vertex algebras. In this section, we reconsider the definition of a vertex algebra with a focus on scale invariance. This will naturally lead to the notion of a full vertex algebra.

In Section \ref{sec_chiral_vertex}, we derive the {\it vertex operator} $Y(-,z):V\rightarrow \End V[[z^\pm]]$ of a vertex algebra starting from scale-invariant operator-valued holomorphic functions (quantum fields). 
Turning the operator-valued holomorphic functions into operator-valued real-analytic functions immediately leads to the notion of a {\it full vertex operator} 
 $Y(-,z,\z):F\rightarrow \End F[[z,\bar{z},|z|^\R]]$.
We also generalize the {\it field condition} $Y(a,z)b \in V((z))$ of a chiral vertex operator to a full vertex operator from this perspective.

In Section \ref{sec_def_vertex}, we revisit the definition of a vertex algebra. In \cite{LL,FB}, a vertex algebra is defined by using a system of correlation functions (see also \cite{FLM}).
We consider the definition of a vertex algebra 
with correlation functions in a slightly different way by focusing on scale invariance and analytic properties  (see Remark \ref{rem_diff_LL} for the differences).
Such a derivation motivates the definition of a full vertex algebra in the next and subsequent sections.

\subsection{The spaces of chiral / full formal power series}
In this section, we introduce spaces of formal power series which is necessary to define a chiral / full vertex algebras. 
We assume that the base field is $\C$ unless otherwise stated. 
Let $z$ be a formal variable. 
For a vector space $W$,
we denote by $W[[z^\pm]]$
the set of formal sums 
$\sum_{n \in \Z} a_{n}z^n$ with $a_n \in W$.
We set the following subspaces of $W[[z^\pm]]$:
\begin{align}
\begin{split}
W[z] &=\Bigl\{\sum_{n \geq 0} a_{n}z^n \mid a_n \in W, a_n=0 \text{ for }n \text{ sufficiently large} \Bigr\},\\
W[z^\pm]&=\Bigl\{\sum_{n \in \Z} a_{n}z^n \mid a_n \in W, a_n=0 \text{ for all but finitely many }n \Bigr\},\\
W((z))&=\Bigl\{\sum_{n \in \Z} a_{n}z^n \mid a_n \in W, a_n=0 \text{ for }n \text{ sufficiently large} \Bigr\},\\
W[[z]] &=\Bigl\{\sum_{n \geq 0} a_{n}z^n \mid a_n \in W\Bigr\}.
\end{split} 
\end{align}
These formal power series are important in defining a vertex algebra \cite{FLM,LL,FHL}.

Next, we consider a 2-variable version of these.
Let $z$ and $\z$ be independent formal variables.
We will use the notation $\underline{z}$ for the pair $(z,\z)$ and $|z|$ for $z\z$.
For a vector space $W$, we denote by $W[[z,\z,|z|^\R]]$
the set of formal sums 
$\sum_{s,\s \in \R} a_{s,\s}z^s\z^\s$ with $a_{s,\s} \in W$ and $a_{s,\s}=0$ unless $s-\s \in \Z$.
\begin{rem}
If $s-\s \in \Z$, then $z^s\z^\s= z^{s-\s} (z\z)^{\s}$ is a single-valued function around $z=0$.
\end{rem}
We will consider the following subspaces of $W[[z,\z,|z|^\R]]$:
\begin{align*}
W[[z,\z]]&= \Bigl\{ \sum_{s,\s \in \Z_{\geq 0}} a_{s,\s} z^s\z^\s \mid a_{s,\s} \in W \Bigr\}, \\
W[z^\pm,\z^\pm]&=\Bigl\{\sum_{s,\s \in \Z}a_{s,\s}z^s \z^\s \;|\;
a_{s,\s} \in W,
\text{ all but finitely many } a_{s,\s}=0 \Bigr\}, \\
W[|z|^\R] &=\Bigl \{\sum_{r \in \R}a_{r}z^r \z^r \;|\;
a_{r} \in W,
\text{ all but finitely many } a_{r}=0 \Bigr\}.
\end{align*}

The most important is the generalization of the space of Laurent series $W((z))$ to two variables.
We denote by
$W((z,\z,|z|^\R))$ the subspace of $W[[z,\z,|z|^\R]]$
consisting of series
$\sum_{s,\s \in \R} a_{s,\s}z^{s} \bar{z}^{\s} \in W[[z,\z,|z|^\R]]$  such that:
\begin{enumerate}
\item
There exists $N\in \R$ such that $a_{s,\s}=0$ for any $s\leq N$ or $\s \leq N$.
\item
For any $R \in \R$,
$\#\{(s,\s)\in \R^2\;|\; a_{s,\s}\neq 0 \text{ and }s+\s\leq R \}
$ is finite.
\end{enumerate}

Let $\frac{d}{dz}$ and $\frac{d}{d\z}$ be the formal differential operators
acting on $W[[z,\z,|z|^\R]]$ by
\begin{align*}
\frac{d}{dz}\sum_{s,\s \in \R} a_{s,\s}z^s \bar{z}^\s
&= \sum_{s,\s \in \R} s a_{s,\s}z^{s-1} \bar{z}^\s \\
\frac{d}{d\z}\sum_{s,\s \in \R} a_{s,\s}z^s \bar{z}^\s
&= \sum_{s,\s \in \R} \s a_{s,\s}z^{s} \bar{z}^{\s-1}.
\end{align*}
Since $\frac{d}{dz}|z|^s=s|z|^s z^{-1}$, the differential operators
$\frac{d}{dz}$ and $\frac{d}{d\z}$ acts on all the above vector spaces.
\begin{lem} \label{hol_formal}
If $f(\uz) \in W((z,\z,|z|^\R))$ satisfies $\frac{d}{d\z} f(\uz)=0$,
then $f(\uz) \in W((z))$.
\end{lem}


\subsection{Chiral / full vertex operators and expansions}\mbox{}\label{sec_chiral_vertex}


In this section, we introduce a {\it chiral vertex operator} (the data defining a vertex algebra) as an operator-valued holomorphic function with scale invariance. We then generalize this definition to a full vertex operator.
The chiral vertex operator is the crucial data defining a vertex algebra. The reader is referred to \cite{LL,FLM,FB,FHL} as textbooks for vertex algebras.

Let $V=\bigoplus_{n \in \Z}V_n$ be a $\Z$-graded vector space and $L(0):V\rightarrow V$ a linear map defined by $L(0)|_{V_n}=n\, \mathrm{id}_{V_n}$ for $n\in \Z$.
In this section, we assume that
\begin{enumerate}
\item[CO1)]
For any $N \in \Z$, $\sum_{n \leq N} \dim V_n$ is finite.
\end{enumerate}
Set 
\begin{align}
\overline{V}&= \Pi_{n\in \Z}V_n,\\
V^\vee &= \bigoplus_{n\in \Z}V_n^*,\label{eq_V_dual}
\end{align}
where $V_n^*$ is the dual vector space.
We note that the canonical pairing $\langle -, -\rangle: V^\vee \otimes \overline{V} \rightarrow \C$ is well-defined.

We consider a map 
\begin{align*}
O(-,z):V \times \C^\times \rightarrow \mathrm{Hom}_\C(V,\overline{V}),\; (a,z)\mapsto O(a,z) 
\end{align*}
which satisfies the following property:
\begin{enumerate}
\item[CO2)]
$O(-,z)$ is linear in the first variable;
\item[CO3)]
For any $u\in V^\vee$ and $a,b \in V$, 
\begin{align*}
\langle u,O(a,-)b\rangle: \C^\times \rightarrow \C,\; z\mapsto \langle u,O(a,z)b\rangle
\end{align*}
is a holomorphic function.
\item[CO4)]
For any $a\in V$,
\begin{align*}
z\frac{d}{dz}O(a,z)=[L(0),O(a,z)] - O(L(0)a,z).
\end{align*}
\end{enumerate}

Let $h_0,h_1,h_2 \in \Z$ and $u \in V_{h_0}^*$, $a_i \in V_{h_i}$.
Then, by (CO4),
\begin{align}
\begin{split}
z\frac{d}{dz}\langle u,O(a_1,z)a_2 \rangle&=
\langle u,[L(0),O(a_1,z)]a_2\rangle - \langle u, O(L(0)a_1,z)a_2\rangle\\
&=(h_0-h_1-h_2) \langle u, O(a_1,z)a_2\rangle.\label{eq_cov_above}
\end{split}
\end{align}
Since $\langle u, O(a_1,z)a_2\rangle$ is a holomorphic function,
there exists a constant $C_{a_1,a_2}^u \in \C$ such that
\begin{align*}
\langle u, O(a_1,z)a_2\rangle = C_{a_1,a_2}^u z^{h_0-h_1-h_2}
\end{align*}
as holomorphic functions on $\C^\times$.
Since $C_{a_1,a_2}^-:V_{h_0}^* \rightarrow \C, u\mapsto C_{a_1,a_2}^u$ is linear in $u$,
there exists a bilinear map
$C_{-,-}^{h_0}:V_{h_1}\otimes V_{h_2} \rightarrow V_{h_0},\; (a_1,a_2)\mapsto C_{a_1,a_2}^{h_0}$
such that
$\langle u, C_{a_1,a_2}^{h_0} \rangle=C_{a_1,a_2}^u$ for any $u\in V_{h_0}^*$.
Then, 
\begin{align}
O(a_1,z)a_2 = (C_{a_1,a_2}^{n} z^{n-h_1-h_2})_{n\in\Z} \in \Pi_{n\in \Z}V_{n}.\label{eq_exp_fir}
\end{align}
Since $V_n=0$ for sufficiently small $n$, we may rewrite \eqref{eq_exp_fir} as
$O(a_1,z)a_2 \in V((z))$.

For $a\in V$, we define a sequence of linear maps $\{a(n) \in \End V\}_{n\in \Z}$ to satisfy 
$O(a,z)b = \sum_{n\in \Z} a(n)b\, z^{-n-1}$ for any $b\in V$,
and let $Y(a,z)=\sum_{n \in \Z}a(n)z^{-n-1} \in \End V[[z^\pm]]$ be its generating function.
Then, we have:
\begin{lem}\label{lem_chiral_cond}
The linear map 
\begin{align*}
Y(-,z):V\rightarrow \End V[[z^\pm]], \; a\mapsto Y(a,z)=\sum_{n\in \Z}a(n)z^{-n-1}
\end{align*}
satisfies
$z\frac{d}{dz}Y(a,z) =[L(0),Y(a,z)]-Y(L(0)a,z)$
and
$Y(a,z)b \in V((z))$
for any $a,b\in V$.
\end{lem}
We call a linear map $Y(-,z):V\rightarrow \End V[[z^\pm]]$
 which satisfies the conditions in Lemma \ref{lem_chiral_cond} a {\it chiral vertex operator} (see Section \ref{sec_def_vertex} for a more precise definition, and see also \cite{LL,FLM,FB}).
By replacing the $\Z$-graded vector space $V$ with $\R^2$-graded vector space, we introduce a {\it full vertex operator}.

Let $F=\bigoplus_{h,\h \in \R}F_{h,\h}$ be a $\R^2$-graded vector space and $L(0),\Ld(0):F\rightarrow F$ a linear map defined by $L(0)|_{F_{h,\h}}=h\, \mathrm{id}_{F_{h,\h}}$ and $\Ld(0)|_{F_{h,\h}}=\h\, \mathrm{id}_{F_{h,\h}}$ for $h,\h\in \R$.
In this section, we assume that
\begin{enumerate}
\item[FO1)]
$F_{h,\h}=0$ unless $h-\h \in \Z$;
\item[FO2)]
There exists $N\in \R$ such that $F_{h,\h}=0$ for any $h\leq N$ or $\h \leq N$;
\item[FO3)]
For any $H \in \R$, $\sum_{\substack{h,h\in\R\\ h+\h \leq H}} \dim F_{h,\h}$ is finite.
\end{enumerate}
Set 
\begin{align}
\overline{F}&= \Pi_{h,\h\in \R}F_{h,\h},\\
F^\vee &= \bigoplus_{h,\h \in \R}F_{h,\h}^*. \label{eq_F_dual}
\end{align}

\begin{rem}
\label{rem_spin_energy}
The vector space of all the states in a full 2d CFT is a module of 
the Lie algebra $\mathrm{so}(1,3)_\C \cong \mathrm{sl}_2\C\oplus \mathrm{sl}_2\C=\bigoplus_{n=-1,0,1}\C L(n) \oplus \bigoplus_{n=-1,0,1}\C\Ld(n)$.
The $\R^2$-grading is defined by the action of $L(0)$ and $\Ld(0)$.
From a physical background, eigenvalues of $L(0)+\Ld(0)$ (resp. $L(0)-\Ld(0)$) is the energy (resp. the spin) of the states.
(FO3) means that the vector space of states below a certain energy is finite dimensional, and (FO1) means that $\exp(2\pi i\theta(L(0)-\Ld(0)))$ defines a representation of $\mathrm{SO}(2)$.
\end{rem}

We consider a map 
\begin{align*}
O(-,x,y): F \times \R^2\setminus \{(0,0)\} \rightarrow \mathrm{Hom}_\C(F,\overline{F}),\; (a,x,y)\mapsto O(a,x,y) 
\end{align*}
which satisfies the following property:
\begin{enumerate}
\item[FO4)]
$O(-,x,y)$ is linear in the first variable;
\item[FO5)]
For any $u\in F^\vee$ and $a,b \in F$, 
\begin{align*}
\langle u,O(a,-)b\rangle: \R^2\setminus \{(0,0)\} \rightarrow \C,\; (x,y)\mapsto \langle u,O(a,x,y)b\rangle
\end{align*}
is continuously differentiable in $x$ and $y$;
\item[FO6)]
For any $a\in F$,
\begin{align*}
z\frac{d}{dz}O(a,x,y)&=[L(0),O(a,x,y)] - O(L(0)a,x,y),\\
\z\frac{d}{d\z}O(a,x,y)&=[\Ld(0),O(a,x,y)] - O(\Ld(0)a,x,y),
\end{align*}
where $\frac{d}{dz}=\frac{1}{2}(\frac{d}{dx} - i\frac{d}{dy})$ and $\frac{d}{d\z}=\frac{1}{2}(\frac{d}{dx} + i\frac{d}{dy})$.
\end{enumerate}

Let $h_0,\h_0,h_1,\h_1,h_2,\h_2 \in \R$ with $h_i-\h_i \in \Z$ ($i=0,1,2$) and $u \in F_{h_0,\h_0}^*$, $a_i \in F_{h_i,\h_i}$.
Then, similar to \eqref{eq_cov_above}, we have
\begin{align*}
z\frac{d}{dz}\langle u,O(a_1,x,y)a_2 \rangle&=(h_0-h_1-h_2) \langle u, O(a_1,x,y)a_2\rangle,\\
\z\frac{d}{d\z}\langle u,O(a_1,x,y)a_2 \rangle&=(\h_0-\h_1-\h_2) \langle u, O(a_1,x,y)a_2\rangle.
\end{align*}
We note that 
since $h_i-\h_i \in \Z$ by (FO1),
\begin{align*}
(x+iy)^{h_0-h_1-h_2}(x-iy)^{\h_0-\h_1-\h_2} = (x^2+y^2)^{\h_0-\h_1-\h_2} (x+iy)^{(h_0-\h_0)-(h_1-\h_1)-(h_2-\h_2)}
\end{align*}
is a single-valued real analytic function on $\R^2 \setminus \{(0,0)\}$.
Hence,
there exists a constant $C_{a_1,a_2}^u \in \C$ such that
\begin{align*}
\langle u, O(a_1,z)a_2\rangle = C_{a_1,a_2}^u z^{h_0-h_1-h_2}\z^{\h_0-h_1-h_2}
\end{align*}
with $z=x+iy$ and $\z = x-iy$ as differentiable functions on $\R^2\setminus \{0,0\}$.

Since $C_{a_1,a_2}^-:F_{h_0,\h_0}^* \rightarrow \C, u\mapsto C_{a_1,a_2}^u$ is linear in $u$,
there exists a bilinear map
$C_{-,-}^{h_0,\h_0}:F_{h_1,\h_1}\otimes F_{h_2,\h_2} \rightarrow F_{h_0,\h_0},\; (a_1,a_2)\mapsto C_{a_1,a_2}^{h_0,\h_0}$
such that
$\langle u, C_{a_1,a_2}^{h_0,\h_0} \rangle=C_{a_1,a_2}^u$ for any $u\in F_{h_0,\h_0}^*$.
Then, 
\begin{align}
O(a_1,x,y)a_2 = (C_{a_1,a_2}^{h,\h} z^{h-h_1-h_2}\z^{\h-\h_1-\h_2})_{h,\h\in\R} \in \Pi_{h,\h\in \R}F_{h,\h}.\label{eq_exp_fir2}
\end{align}
Thus, there exists a unique linear map 
\begin{align*}
Y(-,\uz):F\rightarrow \End F[[z,\z,|z|^\R]],\;a\mapsto Y(a,\uz)=\sum_{r,s\in \R^2}a(r,s)z^{-r-1}\z^{-s-1}
\end{align*}
such that $\langle u,Y(a,\uz)b\rangle = \langle u, O(a,x,y)b\rangle$ for any $u\in F^\vee$ and $a,b\in F$. Moreover, by (FO1), (FO2) and (FO3), it satisfies $Y(a_1,\uz)a_2 \in F((z,\z,|z|^\R))$.
Thus, we have:
\begin{prop}\label{lem_grading_add}
The linear map $Y(-,\uz)$ satisfies
$z\frac{d}{dz}Y(a,\uz) =[L(0),Y(a,\uz)]-Y(L(0)a,\uz)$, $\z\frac{d}{d\z}Y(a,\uz) =[\Ld(0),Y(a,\uz)]-Y(\Ld(0)a,\uz)$ and $Y(a,\uz)b \in F((z,\z,|z|^\R))$.
\end{prop}
We call a linear map $Y(-,\uz):F\rightarrow \End F[[z,\z,|z|^\R]]$
which satisfies the conditions in Proposition \ref{lem_grading_add} a {\it full vertex operator}.
Note that (FO1), (FO2) and (FO3) lead to the property $Y(a,\uz)b\in F((z,\z,|z|^\R))$.


\subsection{Compositions of chiral vertex operator and $\Z$-graded vertex algebra}\mbox{}\label{sec_def_vertex}

Let $V$ be a $\Z$-graded vector space and $Y(-,z):V\rightarrow \End V[[z^\pm]]$ be a linear map such that:
\begin{enumerate}
\item[CV1)]
$Y(a,z)b \in V((z))$ for any $a,b\in V$;
\item[CV2)]
$z\frac{d}{dz}Y(a,z) = [L(0),Y(a,z)]-Y(L(0)a,z)$ for any $a\in V$.
\end{enumerate}
We call $Y(-,z)$ a chiral vertex operator. 
In this section, we investigate compositions of chiral vertex operators based on \cite[Section 1.2]{FB} and \cite[Section 4.2 and 4.3]{LL}.
The difference between us and \cite{FB,LL} is that we focus on the scale invariance ($L(0)$-covariance) and analytic properties of the chiral vertex operator.

\begin{lem}\label{chiral_generalized_formal}
Let $n_i \in \Z$, $a_i \in V_{n_i}$, ($i=1,2,3$), and $u \in V_{n_0}^*$.
Then, $u(Y(a_1,z_1)Y(a_{2},z_{2})a_3) \in z_2^{n_0-n_1-n_2-n_3}
\C((z_2/z_1))$ and $u(Y(Y(a_1,z_0)a_{2},z_{2})a_3) \in z_2^{n_0-n_1-n_2-n_3}
\C((z_0/z_2))$.
\end{lem}
\begin{proof}
Set
$$\sum_{k_1,k_2 \in \Z}c_{k_1,k_2}z_1^{k_1}z_{2}^{k_{2}}=u(Y(a_1,z_1)Y(a_2,z_2)a_3).$$
Then, $$c_{k_1,k_{2}}=
u(a_1(-k_1-1)a_2(-k_2-1)a_3).$$
Since
 $a_1(-k_1-1)a_2(-k_2-1)a_3\in 
V_{n_1+n_2+n_3+k_1+k_2}$.
Hence,
$c_{k_1,k_{2}}=0$
unless
$n_0=n_1+n_2+n_3+k_1+k_2$.
Thus, we have 
\begin{align*}
u(Y(a_1,z_1)Y(a_2,z_2)a_3)
&=\sum_{k \in \Z}c_{n-k,k}z_1^{n-k}z_2^{k}\\
&=z_1^n \sum_{k \in \Z}c_{n-k,k} \left(\frac{z_2}{z_1}\right)^{k},
\end{align*}
where $n=n_0-n_1-n_2-n_3$. The assertion follows from $Y(a_2,z_2)a_3 \in V((z_2))$.
\end{proof}

The composition of chiral vertex operators $u(Y(a_1,z_1)Y(a_{2},z_{2})a_3)$ (resp. $u(Y(Y(a_1,z_0)a_{2},z_{2})a_3)$) is said to be convergent in $|z_1|>|z_2|$ (resp. $|z_2|>|z_0|$) if the formal power series of Lemma \ref{chiral_generalized_formal} is absolutely convergent in $\left|\frac{z_2}{z_1}\right|<1$ (resp. $\left|\frac{z_0}{z_2}\right|<1$).

Set 
\begin{align*}
Y_2 = \Bigl\{(z_1,z_2)\in \C^2\mid z_1\neq z_2,z_1\neq 0,z_2\neq 0 \Bigr\},
\end{align*}
the configuration space of two points in $\C^\times$.

We consider the following conditions:
\begin{enumerate}
\item[B1)]
For any $u\in V^\vee$ and $a_1,a_2,a_3\in V$, $u(Y(a_1,z_1)Y(a_{2},z_{2})a_3)$ (resp. $u(Y(Y(a_1,z_0)a_{2},z_{2})a_3)$) is convergent in $|z_1|>|z_2|$ (resp. $|z_2|>|z_0|$).
\item[B2)]
For any $u\in V^\vee$ and $a_1,a_2,a_3\in V$, there exists a holomorphic function $\mu(z_1,z_2)$ on $Y_2$ such that:
\begin{align}
\begin{split}
u(Y(a,z_1)Y(b,z_2)c) &= \mu(z_1,z_2)|_{|z_1|>|z_2|}, \\
u(Y(Y(a,z_0)b,z_2)c) &= \mu(z_0+z_2,z_2)|_{|z_2|>|z_0|},\\
u(Y(b,z_2)Y(a,z_1)c)&=\mu(z_1,z_2)|_{|z_2|>|z_1|},\label{eq_add_bor}
\end{split}
\end{align}
as holomorphic functions, where $z_0=z_1-z_2$.
\end{enumerate}
The following proposition is significant in considering the generalization from chiral vertex algebras to full vertex algebras (see also Remark \ref{rem_diff_LL}):
\begin{prop}\label{prop_weak_equation}
Let $Y(-,z):V\rightarrow \End V[[z^\pm]]$ be a chiral vertex operator satisfying (B1) and (B2).
Then, for any $u\in V^\vee$ and $a_1,a_2,a_3\in V$, the holomorphic function $\mu(z_1,z_2)$ in (B2) is a rational polynomial in $\C[z_1^\pm,z_2^\pm,(z_1-z_2)^\pm]$.
\end{prop}
\begin{proof}
Let $u\in V_{n_0}^\vee$ and $a_i \in V_{n_i}$ ($i=1,2,3$) and $\mu(z_1,z_2)$ be the holomorphic function on $Y_2$ satisfying (B2).
We note that $\C^\times$ acts on $Y_2$ by $(z_1,z_2)\mapsto (\lambda z_1,\lambda z_2)$ for $(z_1,z_2)\in Y_2$ and $\lambda \in \C^\times$. The quotient map can be written as
\begin{align}
\eta:Y_2\rightarrow \CP\setminus \{0,1,\infty \},\quad
(z_1,z_2)\mapsto \frac{z_2}{z_1}.
\end{align}
By Lemma \ref{chiral_generalized_formal}, 
$(z_1\frac{d}{dz_1}+z_2\frac{d}{dz_2}) \left(z_2^{n_0-n_1-n_2-n_3}\mu(z_1,z_2)\right)=0$,
and thus  there exists a holomorphic function $f$ on $\CP\setminus \{0,1,\infty \}$ such that
$f(\frac{z_2}{z_1}) = z_2^{n_0-n_1-n_2-n_3} \mu(z_1,z_2)$.
The three formal power series in \eqref{eq_add_bor} corresponds to the expansion of $f(z)$ around $z=0$, $z=1$
and $z=\infty$, respectively. 
Hence, (CV1) implies that $f$ has poles at $z=0,1,\infty$.
Such meromorphic functions on $\CP$ are in the polynomial ring $\C[z^\pm,(1-z)^\pm]$.
Hence, the assertion follows.
\end{proof}
\begin{rem}\label{rem_diff_LL}
The claim that $u(Y(a,z_1)Y(b,z_2)c)$ is the expansion of a rational polynomial in $\C[z^\pm,w^\pm,(z-w)^\pm]$ can be found in \cite[Section 1.2]{FB} and \cite[Section 4.2 and 4.3]{LL} as a consequence of the locality axiom $(z-w)^N[Y(a,z),Y(b,w)]=0$. But we give here a more analytical proof.
\end{rem}

Condition (B2) says that the compositions of a chiral vertex operator in different ways will result in the
expansions in the different regions of the same holomorphic function. 
Since the holomorphic functions are just rational polynomials by Proposition \ref{prop_weak_equation},
 the expansions can be written combinatorially \cite{B,FLM}.
Set
\begin{align*}
\GCor_2^\text{hol} = \C[z_1^\pm,z_2^\pm,(z_1-z_2)^\pm],
\end{align*}
the ring of regular functions on the affine scheme $Y_2$.
A function $\mu \in \GCor_2^\text{hol}$ has series expansions in the domains
\begin{align*}
\Bigl\{|z_1|>|z_2|\Bigr\}, \;\;\Bigl\{|z_2|>|z_1|\Bigr\},\;\;\Bigl\{|z_2|>|z_1-z_2|\Bigr\}\subset Y_2.
\end{align*}
More explicitly, for $\mu=z_1^n z_2^m (z_1-z_2)^l$ with $n,m,l \in \Z$, the series expansion in $\Bigl\{|z_1|>|z_2|\Bigr\}$ is
\begin{align*}
z_1^n z_2^m (z_1-z_2)^l\Bigl|_{|z_1|>|z_2|}=
\sum_{i \geq 0} (-1)^i \binom{l}{i}z_1^{l+n-i} z_2^{m+i} \in \C\Bigl(\Bigl(\frac{z_2}{z_1}\Bigr)\Bigr)[z_1^\pm],
\end{align*}
and by setting $z_0=z_1-z_2$, the series expansion in $\Bigl\{|z_2|>|z_0|\Bigr\}$ is
\begin{align*}
z_1^n z_2^m (z_1-z_2)^l\Bigl|_{|z_2|>|z_1-z_2|}=
\sum_{j \geq 0} \binom{n}{j}z_2^{m+n-j} z_0^{l+j} \in \C\Bigl(\Bigl(\frac{z_0}{z_2}\Bigr)\Bigr)[z_2^\pm].
\end{align*}
Hence, we have linear maps
$|_{|z_1|>|z_2|}:\GCor_2^{\text{hol}} \rightarrow \C\Bigl(\Bigl(\frac{z_2}{z_1}\Bigr)\Bigr)[z_1^\pm]$,
$|_{|z_2|>|z_1|}:\GCor_2^{\text{hol}} \rightarrow \C\Bigl(\Bigl(\frac{z_1}{z_2}\Bigr)\Bigr)[z_2^\pm]$, and
$|_{|z_2|>|z_1-z_2|}:\GCor_2^{\text{hol}} \rightarrow \C\Bigl(\Bigl(\frac{z_0}{z_2}\Bigr)\Bigr)[z_2^\pm]$.

Now we can state here the axioms of a $\Z$-graded vertex algebra based on \cite{FB,LL,Li2}.

A $\Z$-graded vertex algebra is a $\Z$-graded $\C$-vector space $V=\bigoplus_{n\in \Z} V_n$ equipped with a linear map
$$Y(-,z):V \rightarrow \End (V)[[z^\pm]],\; a\mapsto Y(a,z)=\sum_{n \in \Z}a(n)z^{-n-1}$$
and an element $\va \in V_0$ satisfying the following conditions:
\begin{enumerate}
\item[V1)]
For any $a,b \in F$, $Y(a,z)b \in V((z))$;
\item[V2)]
For any $a \in V$, $Y(a,z)\va \in V[[z,\z]]$ and $\lim_{z \to 0}Y(a,z)\va = a(-1)\va=a$;
\item[V3)]
$Y(\va,z)=\mathrm{id} \in \End V$;
\item[V4)]
For any $a,b,c \in V$ and $u \in V^\vee$, there exists $\mu(z_1,z_2) \in \GCor_2^{\text{hol}}$ such that
\begin{align*}
u(Y(a,z_1)Y(b,z_2)c) &= \mu|_{|z_1|>|z_2|}, \\
u(Y(Y(a,z_0)b,z_2)c) &= \mu|_{|z_2|>|z_1-z_2|},\\
u(Y(b,z_2)Y(a,z_1)c)&=\mu|_{|z_2|>|z_1|},
\end{align*}
where $z_0=z_1-z_2$;
\item[V5)]
For any $a \in V$, $z\frac{d}{dz}Y(a,z) = [L(0),Y(a,z)]-Y(L(0)a,z)$.
\end{enumerate}

\begin{rem}\label{rem_vertex}
A standard definition of a vertex algebra uses the Borcherds identity (the Jacobi identity) \cite{B}.
From the Borcherds identity and $\Z$-grading, we can show (V4) from \cite[Remark 1.2.7]{FB} and \cite[Proposition 3.2.7, Proposition 3.3.5, and Proposition 3.3.8]{LL}.
Conversely, from (V4), the Borcherds identity follows from \cite[Proposition 3.4.1]{LL}.
Hence, the above definition of a $\Z$-graded vertex algebra is equivalent to the standard one.
We do not use the Borcherds identity since it seems impossible to
obtain such an algebraic identity in the case of non-chiral conformal field theory in general.
\end{rem}

Using what we have discussed so far, we can state the axioms of a vertex algebra in a slightly weaker form.
\begin{prop}
Let $V$ be a $\Z$-graded vector space equipped with a linear map
$Y(-,z):V \rightarrow \End (V)[[z^\pm]]$ and an element $\va \in V_0$ satisfying the following conditions:
\begin{enumerate}
\item
$V_n=0$ for sufficiently small $n$;
\item
For any $a \in V$, $z\frac{d}{dz}Y(a,z) = [L(0),Y(a,z)]-Y(L(0)a,z)$;
\item
For any $a \in V$, $Y(a,z)\va \in V[[z,\z]]$ and $\lim_{z \to 0}Y(a,z)\va = a(-1)\va=a$;
\item
$Y(\va,z)=\mathrm{id} \in \End V$;
\item
The compositions of chiral vertex operators satisfy (B1) and (B2).
\end{enumerate}
Then, $V$ is a $\Z$-graded vertex algebra.
\end{prop}
\begin{proof}
(1) and (2) implies (V1), i.e., $Y(a,z)b\in V((z))$ for any $a,b\in V$.
(V1) together with (5) implies that  (V4) by Proposition \ref{prop_weak_equation}.
\end{proof}

The discussions in this section can be paralleled for full vertex operators. 
The most important change is that two-point correlation functions are changed from 
rational polynomials of the form $\C[z^\pm,(1-z)^\pm]$ to real analytic functions on 
$\CP \setminus \{0,1,\infty\}$ with expansions of the form $C((z,\z,|z|^\R))$ around $z = 0,1,\infty$.
The space of such functions is denoted by $\F$ and is studied in the next section.

\begin{table}[htb]
\begin{tabular}{|l||c|c|}\hline
 & chiral vertex algebra & full vertex algebra \\ \hline \hline
vector space &  $V=\bigoplus_{n \in \Z} V_n$ & $F=\bigoplus_{r,s \in \R^2} F_{r,s}$ \\
vertex operator &$\End V[[z^\pm]]$ & $\End F[[z,\z,|z|^\R]]$ \\
pole & $\C((z))$ & $\C((z,\z,|z|^\R))$ \\
correlation function &$\C[z^\pm,(1-z)^\pm]$ & $\F$\\ \hline
\end{tabular}
\vspace{2mm}
\caption{Comparison of chiral and full vertex algebras}
 \label{table_add}
\end{table}

%
%
%
%
%

\section{Non-chiral correlation functions and formal calculus}\label{sec_formal}
In this section, we introduce the notion of a conformal singularity,
 a typical singularity appearing in generalized two-point correlation functions of a conformal field theory (see Table \ref{table_add}).
We also introduce a space of real analytic functions with possible conformal singularity, which is important to define a full vertex algebra.

\subsection{Convergence}\label{sec_convergent}
In this section, we discuss a convergence of a formal power series in $\C((z,\z,|z|^\R))$ and the uniqueness of expansions.
We will use $z,\z$ as both formal variables and 
the canonical coordinate of $\C$.
For any $R \in \R_{>0}$, set $A_R=\{z\in \C \;|\; 0<|z|<R\}$,
an annulus.

Let $f(\uz) \in \C((z,\z,|z|^\R)).$
Then, there exists $N \in \R$ such that
\begin{align}
|z|^N f(\uz)
= \sum_{\substack{s,\s \in \R
\\  s,\s \geq 0}} a_{s,\s}
z^{s} \bar{z}^{\s}.
\end{align}
We say the series $f(\uz)$ is absolutely convergent around $0$
if there exists $R \in \R_{>0}$ such that the sum
$\sum_{s,\s \in \R} |a_{s,\s}|R^{s+\s}$ is convergent.
In this case, $f(\uz)$ is locally uniformly convergent
to a continuous function defined on the annulus $A_R$. We note that the definition of the convergence is independent of the choice of $N$.

\begin{prop}\label{formal_derivation}
If $f(\uz) \in \C((z,\z,|z|^\R))$ is absolutely convergent around $0$,
then both $\dz f(\uz)$ and $\ddz f(\uz)$ are absolutely
convergent around $0$.
\end{prop}
\begin{proof}
We may assume that $f(\uz)= \sum_{\substack{s,\s \in \R
\\  s,\s \geq 0}} a_{s,\s}z^{s} \bar{z}^{\s}.$
Let $R >0$ be a real number such that
$\sum_{\substack{s,\s \in \R 
\\  s,\s \geq 0}} |a_{s,\s}| R^{s+\s} < \infty$.
Then, $\sum_{\substack{s,\s \in \R 
\\  s,\s \geq 0}} |s a_{s,\s}| (R/2)^{s+\s}
=\sum_{\substack{s,\s \in \R 
\\  s,\s \geq 0}} |s/2^{s+\s}||a_{s,\s}|R^{s+\s} <
\sum_{\substack{s,\s \in \R 
\\  s,\s \geq 0}} |a_{s,\s}|R^{s+\s}< \infty$.
\end{proof}
\begin{rem}
In the above proof, the fact that the sum runs over $s,\s \geq 0$
is essential.
In fact, $\sum_{n =1}^\infty \frac{1}{n^2}(z/\z)^n \in \C[[z/\z]]$
is convergent for $R=1$.
However,
its derivative is not absolutely convergent from
\begin{align*}
\sum_{n =1}^\infty \frac{1}{n}|z/\z|^n/|z| = \frac{1}{R}\sum_{n =1}^\infty \frac{1}{n}
\end{align*}
for any radius $R$.
\end{rem}

\begin{prop}\label{formal_real}
If $f(\uz) \in \C((z,\z,|z|^\R))$ is absolutely convergent around $0$,
then $f(\uz)$ is a real analytic function on the annulus $A_R$ for some $R >0$.
\end{prop}
For the proof, we use the following elementary lemma:
\begin{lem}
Let $s,r\in \R$.
If $s \geq 0$ and $1> |r|$,
then $\sum_{n=0}^\infty |\binom{s}{n}| r^n < (1+r)^s + 2\frac{r^{1+s}}{1-r}$.
\end{lem}
\begin{proof}[proof of Proposition \ref{formal_real}]
We may assume that $f(\uz)= \sum_{\substack{s,\s \in \R
\\  s,\s \geq 0}} a_{s,\s}z^{s} \bar{z}^{\s}.$
Let $R \in \R_{>0}$ such that $\sum_{s,\s \in \R} |a_{s,\s}|R^{s+\s}$ is convergent.
Let $\al \in A_R$. We will show that $f(\uz)$ is
a real analytic function around $\al$.
By the above lemma,
for $w \in \C$ with $|w/\al|<1$ and $|w|+|\al|<R$,
\begin{align}
\sum_{\substack{s,\s \in \R
\\  s,\s \geq 0}}\sum_{n,m=0}^\infty 
\left|\binom{s}{n}\binom{\s}{m}\right|
|a_{s,\s}| |\al|^{s+\s}|w/\al|^{n+m}
< \sum_{\substack{s,\s \in \R
\\  s,\s \geq 0}}|a_{s,\s}|
\Bigl((|\al|+|w|)^s + 2\frac{|w|^{s+1}}{|\al|-|w|}\Bigr)
\Bigl((|\al|+|w|)^\s + 2\frac{|w|^{\s+1}}{|\al|-|w|}\Bigr) \label{eq_ene}
\end{align}
Since the right-hand-side of \eqref{eq_ene} is convergent
by the assumption,
the sum 
$$\sum_{\substack{s,\s \in \R
\\  s,\s \geq 0}}\sum_{n,m=0}^\infty 
\binom{s}{n}\binom{\s}{m}
a_{s,\s}\al^{s-n}\bar{\al}^{\s-m}w^n \bar{w}^m
$$
is absolutely convergent
to 
$\sum_{\substack{s,\s \in \R
\\  s,\s \geq 0}}a_{s,\s} (\al+w)^s\overline{(\al+w)^\s}$.
\end{proof}

Let $\mathrm{Conv}((z,\z,|z|^\R))$ 
the subspace of $\C((z,\z,|z|^\R))$
consisting of $f(\uz) \in \C((z,\z,|z|^\R))$ such that $f(\uz)$ is absolutely convergent around $0$.

Let $\mathrm{St}_0^{\text{real}}$ be the colimit of the space of real analytic functions on $\{z\in \C\;|\;0<|z|<R\}$ as $R \rightarrow 0$.
Then, we have a map 
$$\mathrm{Conv}((z,\z,|z|^\R))\rightarrow \mathrm{St}_0^{\text{real}}.$$
Then, the following lemma is clear:
\begin{lem}\label{unique_coefficient}
The above map $\mathrm{Conv}((z,\z,|z|^\R))\rightarrow \mathrm{St}_0^{\text{real}}$
is injective.
\end{lem}
The above lemma says the coefficients of convergent formal power series are uniquely determined.

We note that $\mathrm{Conv}((z,\z,|z|^\R))$ is a differential subalgebra of $\mathrm{St}_0^{\text{real}}$ (closed under derivations and products).
\begin{rem}
The product $(\sum_{n\in \Z} (z/\z)^n) \cdot (\sum_{n\in \Z} (z/\z)^n)$ is not well-defined.
\end{rem}

%

\subsection{Conformal singularity}\label{sec_singularity}
Let $\al_1,\dots,\al_n \in \C P^1$ and
$f$ be a $\C$-valued real analytic function on $\C P^1\setminus \{\al_1,\dots,\al_n\}$.
A chart $(\chi,\alpha)$ of $\C P^1$ at a point $\alpha \in \C P^1$ is a biholomorphism $\chi$
from an open subset $U$ of $\C P^1$ to an open subset of $\C$ such that $\alpha \in U$ and $\chi(\alpha)=0$.
We say that $f$ has a {\it conformal singularity} at $\alpha_i$
if for any chart $(\chi,\alpha_i)$ of $\C P^1$ at $\alpha_i$,
there exists a formal power series
\begin{align}
\sum_{s,\s \in \R} a_{s,\s} z^s \z^\s \in \text{Conv}((z,\z,|z|^\R)) \label{eq_CS2}
\end{align}
such that it is convergent to $f\circ \chi^{-1}(z)$ in the annulus $A_R$ for some $R \in \R_{>0}$.
It is clear that the above condition is independent of a choice of a chart and by Lemma \ref{unique_coefficient},
 the coefficients of the series are uniquely determined by the chart.
Let $f$ have a conformal singularity at $\al_i$.

Denote by $j(\chi, f) \in \text{Conv}((z,\z,|z|^\R))$ the formal power series
which is convergent to $f\circ \chi^{-1}(z)$,
and by $F_{0,1,\infty}$ the space of real analytic functions on $\C P^1 \setminus \{0,1,\infty \}$
with possible conformal singularities at $\{0,1,\infty\}$.

Examples of functions belonging to $\F$ are
$$|z|^r, |1-z|^r, z^n (1-z)^n, (1-\z)^n \in \F,$$
where $r \in \R$ and $n \in \Z$.
For instance, the expansions of $|1-z|^r$ are
\begin{align*}
j(z, |1-z|^r) &= \sum_{n,m=0}^\infty \binom{r}{n}\binom{r}{m}z^n\z^m,\\
j(1-z^{-1}, |1-z|^r) &= \sum_{n,m=0}^\infty \binom{-r}{n}\binom{-r}{m}z^{n+r} \z^{m+r},\\
j(z^{-1}, |1-z|^r) &= \sum_{n,m=0}^\infty \binom{r}{n}\binom{r}{m}z^{n-r} \z^{m-r},
\end{align*}
where $z,1-z^{-1},z^{-1}$ are charts of $0,1,\infty$, respectively.
In fact, $\F$ is a $\C[z^\pm,(1-z)^\pm,\z^\pm,(1-\z)^\pm,
|z|^\R,|1-z|^\R]$-module.

A non-trivial example of a function in $\F$ is 
\begin{align}
f_{\mathrm{Ising}}(z)=\frac{1}{2}\Bigl(\left|1-\sqrt{1-z}\right|^{1/2}+\left|1+\sqrt{1-z}\right|^{1/2}\Bigr),  \label{eq_Ising}
\end{align}
which appears in a four-point correlation function of the two-dimensional Ising model (see \cite{BPZ}).
The expansion of $f_{\mathrm{Ising}}(z)$ around $0$ with the chart $z$ is
\begin{align}
2+|z|^{1/2}/2-z/4-\z/4+|z|^{1/2}(z+\z)/16+ z\z/32-5z^2/64-5\z^2/64+\dots.
\end{align}
Since $f_{\mathrm{Ising}}(z)$ satisfies the equations $f_{\mathrm{Ising}}(z)=f_{\mathrm{Ising}}(1-z)=
(z\z)^{1/4}f_{\mathrm{Ising}}(1/z)$,
the expansions around $1$ and $\infty$ are
also of the form \eqref{eq_CS2}.
Thus, $f_{\mathrm{Ising}}(z) \in \F$ (see \cite{M4} for the construction of the corresponding full vertex algebra).
More generally, a monodromy invariant combination of
solutions of (holomorphic and anti-holomorphic) KZ-equations
belongs to $\F$ (see also \cite{M5}).

Finally, we remark on the case that 
$f\in \F$ is a holomorphic function.
Recall that the ring of regular functions on the affine scheme
$\CPm$ is $\C[z^\pm,(1-z)^\pm]$.
It is easy to show that
a function in $\C[z^\pm,(1-z)^\pm]$ has conformal singularities at $\{0,1,\infty\}$.
 Thus, $\C[z^\pm,(1-z)^\pm]\subset \F$.
Conversely,
let $f \in \F$ satisfy $\frac{d}{d\z}f=\frac{1}{2}(\frac{d}{dx} + i\frac{d}{dy}) f=0$.
Then, by Lemma \ref{hol_formal},
$f$ is a holomorphic function on $\CPm$ with possible poles at
$\{0,1,\infty\}$, thus, a meromorphic function on $\CP$.
Hence, $f \in \C[z^\pm,(1-z)^\pm]$.
\begin{prop}\label{holomorphic_F}
If $f \in \F$ is a holomorphic function on $\C P^1 \setminus \{0,1,\infty \}$, then $f \in \C[z^\pm,(1-z)^\pm]$.
\end{prop}

\subsection{Generalized two-point correlation function}\label{sec_gen}
This section is devoted to defining and studying a space of generalized two-point functions.
Set $$U(y,z)=
\C((z/y,\z/\y,|z/y|^\R))[y^\pm,\y^\pm,|y|^\R].
$$
Let 
 $\eta(z_1,z_2):Y_2 \rightarrow \CPm$ be the real analytic function defined by
$\eta(z_1,z_2)=\frac{z_2}{z_1}$.
For $f \in \F$, $f \circ \eta$ is a real analytic function on $Y_2$.
Denote by $\GCor_2$ the space of real analytic functions on $Y_2$
spanned by
\begin{align}
z_1^{\al} \z_1^{\be} f\circ \eta(z_1,z_2), \label{eq_GCO}
\end{align}
where $f\in \F$ and $\al,\be \in \R$ satisfy
$\al-\be \in \Z$.

It is clear that $\GCor_2$ is closed under the product
and the derivations $\frac{d}{dz_1},\frac{d}{d\z_1},\frac{d}{dz_2},\frac{d}{d\z_2}$.
Since $(z_1\frac{d}{dz_1}+z_2\frac{d}{dz_2})  z_1^{\al} \z_1^{\be} f\circ \eta(z_1,z_2)
=\al  z_1^{\al} \z_1^{\be} f\circ \eta(z_1,z_2)$
and $(\z_1\frac{d}{d\z_1}+\z_2\frac{d}{d\z_2})  z_1^{\al} \z_1^{\be} f\circ \eta(z_1,z_2)
=\be  z_1^{\al} \z_1^{\be} f\circ \eta(z_1,z_2)$, by using a formal calculus,
we have:
\begin{lem}
\label{generalized_limit}
Let $\mu \in \GCor_2$ satisfy $(z_1\frac{d}{dz_1}+z_2\frac{d}{dz_2})\mu= \al \mu$
and $(\z_1\frac{d}{d\z_1}+\z_2\frac{d}{d\z_2})\mu= \be \mu$ for some $\al,\be \in \R$.
Then, there exists unique $f \in \F$ such that
$\mu(z_1,z_2)=z_1^\al\z_1^\be f(\frac{z_1}{z_2})$.
\end{lem}

Let $\mu(z_1,z_2)=z_1^{\al} \bar{z}_1^{\be} f\circ \eta(z_1,z_2)$ in \eqref{eq_GCO}.
The expansions of $\mu(z_1,z_2)$ in $\{|z_1|>|z_2|\}$
and $\{|z_2|>|z_1|\}$ are respectively
given by
\begin{align*}
z_1^{\al}\bar{z}_1^{\be}&\lim_{z \to z_2/z_1} j(z,f)\\
z_1^{\al}\bar{z}_1^{\be}&\lim_{z \to z_1/z_2} j(z^{-1},f),
\end{align*}
which define maps
$$|_{|z_1|>|z_2|}:\GCor_2 \rightarrow U(z_1,z_2),\; \mu(z_1,z_2) \mapsto \mu(z_1,z_2)|_{|z_1|>|z_2|}$$
and
$$|_{|z_2|>|z_1|}:\GCor_2 \rightarrow U(z_2,z_1),\; \mu(z_1,z_2) \mapsto \mu(z_1,z_2)|_{|z_2|>|z_1|}.$$
Since
$f(\frac{z_2}{z_1})=f(\frac{z_2}{z_2+(z_1-z_2)})$,
the expansions of $\mu$ in $\{|z_2|>|z_1-z_2|\}$
is given by
\begin{align*}
z_2^{\al} \bar{z}_2^{\be}
\sum_{i,j \geq 0}\binom{\al}{i}\binom{\be}{j}
(z_0/z_2)^i(\z_0/\z_2)^j &\lim_{z \to -z_0/z_2} j(1-z^{-1},f),
\end{align*}
where $z_0=z_1-z_2$.
We denote it by
$$|_{|z_2|>|z_1-z_2|}:\GCor_2 \rightarrow U(z_2,z_0),
\mu(z_1,z_2) \mapsto \mu(z_1,z_2)|_{|z_2|>|z_1-z_2|}.
$$
Then, we have:
\begin{lem}
For $f\in \F$,
\begin{align*}
f\circ \eta |_{|z_1|>|z_2|} &= \lim_{z \to z_1/z_2} j(z,f),\\
f\circ \eta |_{|z_2|>|z_1|} &= \lim_{z \to z_2/z_1} j(z^{-1},f),\\
f\circ \eta |_{|z_2|>|z_1-z_2|} &= \lim_{z \to -z_0/z_2} j(1-z^{-1},f).
\end{align*}
\end{lem}
The following lemma connects a full vertex algebra (real analytic) and a vertex algebra (holomorphic):
\begin{lem}
\label{hol_generalized}
Let $\mu(z_1,z_2) \in \GCor_2$ satisfies $\frac{d}{d\z_1} \mu=0$,
 $(z_1\frac{d}{dz_1}+z_2\frac{d}{dz_2})\mu=\al \mu$
and  $(\z_1\frac{d}{d\z_1}+\z_2\frac{d}{d\z_2})\mu=\be \mu$ for some $\al,\be\in \R$.
Then, $\mu(z_1,z_2)\in \C[z_1^\pm,(z_1-z_2)^\pm,z_2^\pm,\z_2^\pm,|z_2|^\R]$.
Furthermore, if $\frac{d}{d\z_2} \mu=0$, then $\mu(z_1,z_2) \in \GCor_2^{\text{hol}}=\C[z_1^\pm,z_2^\pm,(z_1-z_2)^\pm]$.
\end{lem}
\begin{proof}
By Lemma \ref{generalized_limit}, there exits $f \in \F$ such that
$\mu(z_1,z_2)=z_2^\al \z_2^\be f(z_1/z_2)$.
By $\frac{d}{d\z_1}\mu=0$, $f$ is holomorphic and by Proposition \ref{holomorphic_F},
$f \in \C[z^\pm,(1-z)^\pm]$.
Thus, $\mu \in \C[z_1^\pm,(z_1-z_2)^\pm, z_2^\pm,\z_2^\pm,|z_2|^\R]$.
If $\frac{d}{d\z_2}\mu=0$, then $\be=0$ and $\al \in \Z$.
Hence, the assertion holds.
\end{proof}

\section{Full vertex algebra}\label{sec_full_vertex}
In this section, we introduce the notion of a full vertex algebra,
which is a generalization of a $\Z$-graded vertex algebra.
This formulation is strongly influenced by the formulation of a $\Z$-graded vertex algebra in \cite{FB} and \cite{LL}, as described in Section \ref{sec_def_vertex}.

\subsection{Definition of full vertex algebra}\mbox{}
For an $\R^2$-graded vector space $F=\bigoplus_{h,\h \in \R^2} F_{h,\h}$, 
set $F^\vee=\bigoplus_{h,\h \in \R^2} F_{h,\h}^*$,
where $F_{h,\h}^*$ is the dual vector space of $F_{h,\h}$.
Define linear maps $L(0),\Ld(0)\in \End\,F$ by
 $L(0)|_{F_{h,\h}}=h\,\mathrm{id}_{F_{h,\h}}$ and $\Ld(0)|_{F_{h,\h}} = \h\,\mathrm{id}_{F_{h,\h}}$ for any $h,\h\in \R$.

A full vertex algebra is an $\R^2$-graded $\C$-vector space
$F=\bigoplus_{h,\h \in \R^2} F_{h,\h}$ equipped with a linear map 
$$Y(-,\uz):F \rightarrow \End (F)[[z^\pm,\z^\pm,|z|^\R]],\; a\mapsto Y(a,\uz)=\sum_{r,s \in \R}a(r,s)z^{-r-1}\z^{-s-1}$$
and an element $\va \in F_{0,0}$ 
satisfying the following conditions:
%
\begin{enumerate}
\item[FV1)]
For any $a,b \in F$, $Y(a,\uz)b \in F((z,\z,|z|^\R))$;
\item[FV2)]
$F_{h,\h}=0$ unless $h-\h \in \Z$;
\item[FV3)]
For any $a \in F$, $Y(a,\uz)\va \in F[[z,\z]]$ and $\lim_{z \to 0}Y(a,\uz)\va = a(-1,-1)\va=a$;
\item[FV4)]
$Y(\va,\uz)=\mathrm{id}_F \in \End F$;
\item[FV5)]
For any $a,b,c \in F$ and $u \in F^\vee$, there exists $\mu(z_1,z_2) \in \GCor_2$ such that
\begin{align*}
u(Y(a,\uz_1)Y(b,\uz_2)c) &= \mu|_{|z_1|>|z_2|}, \\
u(Y(Y(a,\uz_0)b,\uz_2)c) &= \mu|_{|z_2|>|z_1-z_2|},\\
u(Y(b,\uz_2)Y(a,\uz_1)c)&=\mu|_{|z_2|>|z_1|},
\end{align*}
where $z_0=z_1-z_2$;
\item[FV6)]
For any $a \in F$,
\begin{align}
\begin{split}
z\frac{d}{dz}Y(a,\uz) &= [L(0),Y(a,\uz)]-Y(L(0)a,\uz),\\
\z \frac{d}{d\z}Y(a,\uz) &=[\Ld(0),Y(a,\uz)]-Y(\Ld(0)a,\uz).
\end{split}
\label{eq_cov}
\end{align}
\end{enumerate}


\begin{rem}
Physically, the energy and the spin of a state in $F_{h,\h}$ are $h+\h$ and $h-\h$.
Thus, the condition (FV2) implies that we only consider the particles whose spin is an integer,
that is, we consider only bosons and not fermions.
The notion of a full vertex superalgebra can be defined by modifying (FV5) and (FV2).
\end{rem}
\begin{rem}\label{L0_operator}
The condition (FV6) is equivalent to the the following condition:
\begin{align*}
F_{h,\h}(r,s)F_{h',\h'} \subset F_{h+h'-r-1,\h+\h'-s-1}\quad \text{ for any } r,s,h,h',\h,\h' \in \R.
\end{align*}

\end{rem}

Since $V((z)) \subset V((z,\z,|z|^\R))$ and $\GCor_2^{\text{hol}} \subset \GCor_2$,
by \cite{LL,FB},
we have (see Remark \ref{rem_vertex}):
\begin{prop}\label{graded_vertex}
A $\Z$-graded vertex algebra is a full vertex algebra.
\end{prop}

Let $F$ be an $\R^2$-graded vector space.
The set $\{(h,\h)\in \R^2\;|\;F_{h,\h}\neq 0\}$ is called
a spectrum.
The spectrum of $F$ is said to be
{\it bounded below} if
there exists $N\in \R$ such that
$F_{h,\h}=0$ for any $h \leq N$ or $\h\leq N$
and {\it discrete} if for any $H \in \R$, $\sum_{h+\h < H} \dim F_{h,\h}$ is finite
and {\it compact} if it is both bounded below and discrete.
A full vertex algebra with a compact spectrum is called a {\it compact full vertex algebra}.
Many interesting models in conformal field theory (e.g., rational conformal field theory and its deformation) have a compact spectrum.


For vertex algebras, we can consider the notions of modules, homomorphisms, and ideals (see, for example, \cite{FLM,LL}).
We will consider these full analogues.
Let $(F^1,Y^1,\va^1)$ and $(F^2,Y^2,\va^2)$ be full vertex algebras.
A full vertex algebra homomorphism from $F^1$ to $F^2$ is a linear map 
$f:F^1\rightarrow F^2$ such that
\begin{enumerate}
\item
$f(\va^1)=\va^2$;
\item
$f(Y^1(a,\uz)b)=Y^2(f(a),\uz)f(b)$ for any $a,b\in F^1$.
\end{enumerate}

A module of a full vertex algebra $F$ is an $\R^2$-graded $\C$-vector space
$M=\bigoplus_{h,\h \in \R^2} M_{h,\h} $ equipped with a linear map 
$$Y_M(-,\uz):F \rightarrow \End (M)[[z^\pm,\z^\pm,|z|^\R]],\; a\mapsto Y_M(a,z)=\sum_{r,s \in \R}a(r,s)z^{-r-1}\z^{-s-1}$$
 satisfying the following conditions:
\begin{enumerate}
\item[FM1)]
For any $a \in F$ and $m \in M$, $Y(a,\uz)m \in M((z,\z,|z|^\R))$;
\item[FM2)]
$Y_M(\va,\uz)=\mathrm{id} \in \End M$;
\item[FM3)]
For any $a,b \in F$, $m \in M$ and $u \in M^\vee$, there exists $\mu \in \GCor_2$ such that
\begin{align*}
u(Y_M(a,\uz_1)Y_M(b,\uz_2)m) &= \mu|_{|z_1|>|z_2|}, \\
u(Y_M(Y(a,\uz_0)b,\uz_2)m) &= \mu|_{|z_2>|z_1-z_2|},\\
u(Y_M(b,\uz_2)Y_M(a,\uz_1)m)&=\mu|_{|z_2|>|z_1|};
\end{align*}
\item[FM4)]
$F_{h,\h}(r,s)M_{h',\h'} \subset M_{h+h'-r-1,\h+\h'-s-1}$ for any $r,s,h,h',\h,\h' \in \R$.
\end{enumerate}

As a consequence of (FM1) and (FM3), we have:
\begin{lem}\label{generalized_formal}
Let $h_i,\h_i\in \R$, $a_i \in F_{h_i,\h_i}$, ($i=1,2$), $m \in M_{h_3,\h_3}$ and $u \in M_{h_0,\h_0}^*$.\\
Then, $u(Y_M(a_1,\uz_1)Y_M(a_{2},\uz_{2})m) \in z_1^{h_0-h_1-h_2-h_3}\z_1^{\h_0-\h_1-\h_2-\h_3}
\C((z_2/z_1,\z_2/\z_1,|z_2/z_1|^\R))$.
\end{lem}
\begin{proof}
Set
$$\sum_{s_1,\s_1,\s_2,\s_2\in \R}c_{s_1,\s_1,\s_2,\s_2}z_1^{s_1}\z_1^{\s_1}z_{2}^{s_{2}}\z_{2}^{s_{2}}=u(Y_M(a_1,\uz_1)Y_M(a_2,\uz_2)m).$$
Then, $$c_{s_1,\s_1,s_{2},\s_{2}}=
u(a_1(-s_1-1,-\s_1-1)a_2(-s_2-1,-\s_2-1)m).$$
By (FM4),
 $a_1(-s_1-1,-\s_1-1)a_2(-s_2-1,-\s_2-1)m\in 
M_{h_1+h_2+h_3+s_1+s_2,\h_1+\h_2+\h_3+\s_1+\s_2}$.
Hence,
$c_{s_1,\s_1,s_{2},\s_{2}}=0$
unless
$h_0=h_1+h_2+h_3+s_1+s_2$ and
$\h_0=\h_1+\h_2+\h_3+\s_1+\s_2$.
Thus, we have 
\begin{align*}
u(Y_M(a_1,\uz_1)Y_M(a_2,\uz_2)m)
=z_1^{h_0-h_1-h_2-h_3}\z_1^{\h_0-\h_1-\h_2-\h_3}
\sum_{s_2,\s_2 \in \R}c_{s_1,\s_1,\s_2,\s_2}(z_{2}/z_1)^{s_{2}}
(\z_{2}/\z_1)^{s_{2}},
\end{align*}
where $s_1=h_0-(h_1+h_2+h_3+s_2)$
and $\s_1=\h_0-(\h_1+\h_2+\h_3+\s_2)$.
By (FM1), the assertion holds.
\end{proof}

By Lemma \ref{generalized_formal} and Lemma \ref{generalized_limit},
we have:
\begin{lem}\label{borcherds}
Let $h_i,\h_i\in \R$, $a_i \in F_{h_i,\h_i}$ ($i=1,2$), $m\in M_{h_3,\h_3}$ and $u \in M_{h_0,\h_0}^*$, there exists $f \in \F$ such that
\begin{align*}
z_2^{-h_0+h_1+h_2+h_3}\z_2^{-\h_0+\h_1+\h_2+\h_3}
u(Y_M(a,\uz_1)Y_M(b,\uz_2)m) &=\lim_{z \to z_2/z_1} j(z,f), \\
z_2^{-h_0+h_1+h_2+h_3}\z_2^{-\h_0+\h_1+\h_2+\h_3}u(Y_M(Y(a,\uz_0)b,\uz_2)m) &= \lim_{z \to -z_0/z_2} j(1-z^{-1},f),\\
z_2^{-h_0+h_1+h_2+h_3}\z_2^{-\h_0+\h_1+\h_2+\h_3}u(Y_M(b,\uz_2)Y_M(a,\uz_1)m)&=\lim_{z \to z_1/z_2} j(1/z,f).
\end{align*}
\end{lem}

Let $M,N$ be a $F$-module.
A $F$-module homomorphism from $M$ to $N$ is a linear map $f:M\rightarrow N$
such that $f(Y_M(a,\uz)m)=Y_N(a,\uz)f(m)$ for any $a\in F$ and $m\in M$.

Let $M$ be a $F$-module.
A vector $v \in M$ is said to be a {\it vacuum-like vector}
if $Y(a,\uz)v \in M[[z,\z]]$ for any $a\in F$.

The notion of a vacuum-like vector was introduced in \cite{Li} to study invariant bilinear forms on a vertex operator algebra. 
The space of invariant bilinear forms on a full vertex algebra can be defined and studied similarly in \cite{M2}.
The following lemma is an important property about vacuum-like vectors,
which is an analogue of Proposition 3.4 in \cite{Li}.

\begin{lem}
\label{vacuum}
Let $v \in M$ be a vacuum-like vector and $a,b \in F$ and $u \in M^\vee$ and 
$\mu \in \GCor_2$ satisfy $u(Y_M(a_1,\uz_1)Y_M(a_2,\uz_2)v)=\mu|_{|z_1|>|z_2|}$.
Then, $\mu(z_1,z_2) \in \C[z_2^\pm, \z_2^\pm, (z_1-z_2)^\pm, (\z_1-\z_2)^\pm, |z_1-z_2|^\R]$.
Furthermore, the linear function $F_v: F \rightarrow M$ defined by
$a \mapsto a(-1,-1)v$ is a $F$-module homomorphism.
\end{lem}
\begin{proof}
By (FM3), $u(Y_M(Y(a_1,\uz_0)a_2,\uz_2)v)=\mu|_{|z_1|>|z_1-z_2|}$.
Since $v$ is a vacuum like vector, by Lemma \ref{generalized_formal}
 $p(z_0,z_2)=\mu|_{|z_1|>|z_1-z_2|}
 \in \C[z_0^{\pm},\z_0^{\pm},|z_0|^\R, z_2,\z_2]\subset U(z_2,z_0)$,
which proves the first part of the lemma.
It suffices to show that $F_v(Y(a_1,\uz_0)a_2)=Y_M(a_1,\uz_0)F_v(a_2)$.
Since
\begin{align*}
u(Y_M(a_1,\uz_1)Y_M(a_2,\uz_2)v)&=\mu|_{|z_1|>|z_2|}\\
&=\lim_{z_0\to (z_1-z_2)|_{|z_1|>|z_2|}}p(z_0,z_2),
\end{align*}
we have
\begin{align}
u(Y_M(a_1,\uz_0)Y_M(a_2,\uz_2)v)=\exp(-z_2 \frac{d}{dz_0}-\z_2 \frac{d}{d\z_0})u(Y_M(Y(a_1,\uz_0)a_2,\uz_2)v).
\end{align}
Thus,
\begin{align*}
Y_M(a_1,\uz_0)F_v(a_2)&=\lim_{z_2 \mapsto 0} u(Y_M(a_1,\uz_0)Y_M(a_2,\uz_2)v)\\
&=\lim_{z_2 \mapsto 0}\exp(-z_2 \frac{d}{dz_0}-\z_2 \frac{d}{d\z_0})u(Y_M(Y(a_1,\uz_0)a_2,\uz_2)v)\\
&=F_v(Y(a_1,\uz_0)a_2).
\end{align*}
\end{proof}

Let $F$ be a full vertex algebra
and $D$ and $\D$ denote the endomorphism of $F$
defined by $Da=a(-2,-1)\bm{1}$ and $\D a=a(-1,-2)$ for $a\in F$,
i.e., $$Y(a,z)\va=a+Daz+\D a\z+\dots.$$
Define $Y(a,-\uz)$ by
$Y(a,-\uz)=\sum_{r,s}(-1)^{r-s} a(r,s)z^r \z^s$,
where we used $a(r,s)=0$ for $r-s \notin \Z$,
which follows from (FV2) and (FV6).

For a vertex algebra, the operator $D$ can also be defined and satisfies important properties such as {\it translation invariance} and 
{\it skew-symmetry}. The following proposition is its full analogue (see \cite[Section 3]{LL} for a proof in a vertex algebra case).

\begin{prop}
\label{translation}
For $a \in F$, the following properties hold:
\begin{enumerate}
\item
$Y(Da,\uz)=\dz Y(a,\uz)$ and $Y(\D a,\uz)=\ddz Y(a,\uz)$;
\item
$D\va=\D\va=0$;
\item
$[D,\D]=0$;
\item
$Y(a,\uz)b=\exp(zD+\z\D)Y(b,-\uz)a$;
\item
$Y(\D a,\uz)=[\D,Y(a,\uz)]$ and $Y(Da,\uz)=[D,Y(a,\uz)]$.
\end{enumerate}
\end{prop}
\begin{proof}
Let $u \in F^\vee$ and $a, b \in F$ and $\mu_1,\mu_2\in \GCor_2$
satisfy 
\begin{align*}
u(Y(a,\uz_1)Y(\va,\uz_2)b)=\mu_1|_{|z_1|>|z_2|},
u(Y(a,\uz_1)Y(b,\uz_2)\va)=\mu_2 |_{|z_1|>|z_2|}.
\end{align*}
By (FV4) and (FV5), $p_1(z_1)=\mu_1|_{|z_1|>|z_2|} \in \C[z_1^\pm,\z_1^\pm,|z_1|^\R]$.
Then,
\begin{align*}
u(Y(Y(a,\uz_0)\va,\uz_2)b)=\mu_1|_{|z_2|>|z_1-z_2|}
= \lim_{z_1\rightarrow z_2} \exp(z_0\frac{d}{dz_1}) \exp(\z_0\frac{d}{d\z_1})p_1(z_1).
\end{align*}
Thus, $u(Y(Da,\uz_2)b)=\lim_{z_1\rightarrow z_2} \frac{d}{dz_1} p_1(z_1)=
\frac{d}{dz_2}u(Y(a,z_2)b)$,
which implies that $Y(Da,\uz)=\frac{d}{dz}Y(a,\uz)$ and similarly $Y(\D a,\uz)=\frac{d}{d\z}Y(a,\uz)$.

By (FV4), $Y(D\va,\uz)=\frac{d}{dz}Y(\va,\uz)=0$.
Thus, by (FV3), $D\va=\D\va=0$.
Since $Y(D\D a,\uz)=\frac{d}{dz}\frac{d}{d\z}Y(a,\uz)=\frac{d}{d\z}\frac{d}{dz}Y(a,\uz)=Y(\D Da,\uz)$,
we have $[D,\D]=0$.

By Lemma \ref{vacuum}, $\mu_2|_{|z_2|>|z_1-z_2|} \in \C[z_2,\z_2][z_0^{\pm},\z_0^{\pm},|z_0|^\R]$.
Set $p(z_0,z_2)=\mu_2|_{|z_2|>|z_1-z_2|}=u(Y(Y(a,\uz_0)b,\uz_2)\va)$.
Since 
$u(Y(Y(b,-\uz_0)a,\uz_1)\va) = p(z_0,z_1-z_0)|_{|z_1|>|z_0|}$,
we have
\begin{align*}
u(Y(a,\uz_0)b)&=p(z_0,0)= \lim_{z_1\to 0} \exp(z_0\frac{d}{dz_1}+\z_0\frac{d}{d\z_1})p(z_0,z_1-z_0)\\
&=\lim_{z_1\to 0} \exp(z_0\frac{d}{dz_1}+\z_0\frac{d}{d\z_1})u(Y(Y(b,-\uz_0)a,\uz_1)\va) \\
&=\lim_{z_1\to 0} u(Y(\exp(z_0D+\z_0\D) Y(b,-\uz_0)a,\uz_1)\va) \\
&=u(\exp(z_0D+\z_0\D) Y(b,-\uz_0)a).
\end{align*}

Finally,
\begin{align*}
\frac{d}{dz}Y(a,\uz)b&=\frac{d}{dz}\exp(Dz+\D\z)Y(b,-\uz)a \\
&= D\exp(Dz+\D\z)Y(b,-\uz)a -\exp(Dz+\D\z)Y(Db,-\uz)a\\
&=DY(a,\uz)b-Y(a,\uz)Db.
\end{align*}
\end{proof}

The notion of a subalgebra can be defined in the usual way.
For a full vertex algebra $F$, a subspace $I \subset F$ is called an ideal if
$Y(a,\uz)m \in I((z,\z,|z|^\R))$ for any $a\in F$ and $m\in I$.

\begin{rem}
\label{rem_ideal}
The notion of an ideal for a fulll vertex algebra resembles the notion of a left ideal for an associative algebra:
As in \cite[Remark 3.9.8]{LL}, in the view of (4) in Proposition \ref{translation},
under the condition that $DI,\D I \subset I$,
the left-ideal condition, $Y(a,\uz)m \in I((z,\z,|z|^\R))$ for $a\in F$, $m \in I$,
is equivalent to the right ideal condition: $Y(m,\uz)a \in I((z,\z,|z|^\R))$ for $a\in F$, $m \in I$.
Also, if $F$ is a conformal vertex algebra introduced in Section \ref{sec_def_full_conformal},
$DI,\D I \subset I$ is automatically satisfied for both left ideals and right ideals
since $Da=a(-2,-1)\va = \om(0,-1)a$ and $\D a=a(-1,-2)\va = \omb(-1,0)a$.
Therefore, as with a vertex algebra, there is basically no need to distinguish between left and right ideals in the case of a full vertex algebra.
\end{rem}

A simple full vertex algebra is 
a full vertex algebra which contains no proper ideals.
We will use the following lemma, which is an analogue of \cite[Proposition 11.9]{DL}:
\begin{lem}\label{zero_divisor}
Let $F$ be a simple full vertex algebra and $a,b\in F$.
If $Y(a,\uz)b = 0$, then $a=0$ or $b=0$.
\end{lem}
\begin{proof}
Assume that $Y(a,\uz)b =0$ and $b\neq 0$.
Let $(b)$ be the ideal generated by $b$,
that is, $(b)=\{c_1(r_1,s_1)c_2(r_2,s_2)\dots c_k(r_k,s_k)b \}$.
We will show that $u(Y(a,z)c_1(r_1,s_1)c_2(r_2,s_2)\dots c_k(r_k,s_k)b)=0$ for any $u\in F^\vee$,
$k\in \Z_{\geq 0}$, $c_i \in F$ and $r_i,s_i \in \R$ ($i=1,\dots,k$) by induction on $k$.
For $k=0$, the assertion is clear.
For $k\geq 1$,
by the induction assumption, $u(Y(c_1,\uz_2) Y(a,\uz_1)c_2(r_2,s_2)\dots c_k(r_k,s_k)b)=0$.
Thus, by (FV5), $u(Y(a,\uz_1)Y(c_1,\uz_2)c_2(r_2,s_2)\dots c_k(r_k,s_k)b)=0$.
Hence, the assertion holds. Since $F$ is simple and $b\neq 0$, $\va \in (b)$. Thus, $a$ must be $0$ by (FV3).
\end{proof}

Let $(F,Y,\va)$ be a full vertex algebra.
Set $\bar{F}=F$ and
$\bar{F}_{h,\h}=F_{\h,h}$ for $h,\h\in \R$.
Define $\bar{Y}(-,\uz):\bar{F} \rightarrow \End (\bar{F})[[z,\z,|z|^\R]]$
by $\bar{Y}(a,\uz)=\sum_{s,\s \in \R}a(s,\s)\z^{-s-1}z^{-\s-1}$.
Let $C:Y_2\rightarrow Y_2$ be the conjugate map
$(z_1,z_2)\mapsto (\z_1,\z_2)$ for $(z_1,z_2)\in Y_2$.
For $u\in \bar{F}^\vee$ and $a,b,c\in \bar{F}$,
let $\mu \in \GCor_2$ satisfy
$u(Y(a,\uz_1)Y(b,\uz_2)c)=\mu(z_1,z_2)|_{|z_1|>|z_2|}$.
Then, 
$u(\bar{Y}(a,\uz)\bar{Y}(b,\uz)c)=\mu\circ C(z_1,z_2)$.
Since $\mu \circ C \in \GCor_2$, we have:
\begin{prop}\label{conjugate}
$(\bar{F},\bar{Y},\va)$ is a full vertex algebra.
\end{prop}
We call it a {\it conjugate full vertex algebra} of $(F,Y,\va)$.




\subsection{Chiral and anti-chiral vectors}
Let $F$ be a full vertex algebra.
A vector $a \in F$ is said to be a {\it chiral vector} (resp. an {\it anti-chiral vector})
if $\D a=0$ (resp. $D a=0$).
Let $a \in \ker \D$.
Then, since $0=Y(\D a,\uz)=\ddz Y(a,\uz)$,
we have $a(r,s)=0$ unless $s=-1$.
Hence, $Y(a,\uz)=\sum_{n \in \Z} a(n,-1) z^{-n-1}$.

\begin{lem}
\label{hol_commutator}
Let $a,b\in F$.
If $\D a=0$,
then for any $n\in \Z$,
\begin{align*}
[a(n,-1),Y(b,\uz)]&= \sum_{i \geq 0} \binom{n}{i} Y(a(i,-1)b,\uz)z^{n-i},\\
Y(a(n,-1)b,\uz)&= 
\sum_{i \geq 0} \binom{n}{i}(-1)^i a(n-i,-1)z^{i}Y(b,\uz)
-Y(b,\uz)\sum_{i \geq 0} \binom{n}{i}(-1)^{i+n} a(i,-1)z^{n-i}.
\end{align*}
\end{lem}
\begin{proof}
For any $u\in F^\vee$ and $c\in F$,
there exists $\mu \in \GCor_2$ such that (FV5) holds.
Since $\D a=0$, by Proposition \ref{translation},
$d/d\z_1 \mu(z_1,z_2)=0$.
Then, by Lemma \ref{hol_generalized}, $\mu \in \C[z_1^\pm,(z_1-z_2)^\pm,z_2^\pm,\z_2^\pm,|z_2|^\R]$.
Thus, by the Cauchy integral formula, the assertion holds.
\end{proof}

By Proposition \ref{translation}, $\D Y(a,\uz)b =Y(\D a,\uz)b+Y(a,\uz)\D b=0$.
Thus, the restriction of $Y$ on $\ker \D$ define a linear map $Y(-,z): \ker \D \rightarrow \End \ker \D[[z^\pm]]$.
By the above Lemma and Lemma \ref{hol_generalized}, we have:
\begin{prop}\label{vertex_algebra}
$\ker \D$ is a vertex algebra and $F$ is a $\ker \D$-module.
\end{prop}

\begin{proof}
In order to prove that $\ker \D$ is a vertex algebra,
it suffices to show that $\ker \D$ satisfies the Goddard's axioms \cite[Theorem 5.6.1]{LL}.
Since $[D,\D]=0$, $D$ acts on $\ker \D$.
By Proposition \ref{translation}, it suffices to show that $Y(a,z)$ and $Y(b,w)$ are mutually local for any $a,b \in \ker \D$.
Let $a,b \in \ker \D$ and $v \in F$, $u\in F^\vee$ and $\mu \in \GCor_2$ satisfy $u(Y(a,z_1)Y(a_2,z_2)v) = \mu|_{|z_1|>|z_2|}$.
By Lemma \ref{hol_generalized},
$\mu$ is a polynomial in $\C[z_1^\pm,z_2^\pm,(z_1-z_2)^{-1}]$.
Since $\mu|_{|z_2|>|z_1-z_2|}=u(Y(Y(a_1,z_0)a_2,z_2)v)$ and $a_1(n,-1)a_2=0$ for sufficiently large $n\in \Z$,
there exists $N \in \Z_{\geq 0}$ such that $(z_1-z_2)^N \mu(z_1,z_2)  \in \C[z_1^\pm,z_2^\pm]$.
Thus, $(z_1-z_2)^N u(Y(a,z_1)Y(a_2,z_2)v)=(z_1-z_2)^N u(Y(a_2,z_2)Y(a_1,z_1)v)$ for any $v \in F$ and $u \in F^\vee$,
which implies that $Y(a_1,z_1)$ and $Y(a_2,z_2)$ are mutually local
and $F$ is a $\ker \D$-module (see, for example, \cite[Proposition 4.4.3]{LL}).
\end{proof}

\begin{lem}
\label{hol_commute}
Let $a\in F$ be a chiral vector
and $b \in F$ an anti-chiral vector.
Then, $[Y(a,z),Y(b,\bar{w})]=0$, that is,
$[a(n,-1),b(-1,m)]=0$ and $a(k,-1)b=0$ for any $n,m \in \Z$ and $k \in \Z_{\geq 0}$.
\end{lem}
\begin{proof}
By Lemma \ref{hol_commutator},
it suffices to show that $a(k,-1)b=0$ for any $k \geq 0$.
Since $DY(a,z)b=[D,Y(a,z)]b+Y(a,z)Db=\dz Y(a,z)b$,
we have $Da(n,-1)b=-na(n-1,-1)b$ for any $n \in \Z$.
Thus,  the assertion follows from (FV1).
\end{proof}

\subsection{Tensor product of full vertex algebras}
A tensor product of vertex algebras was considered in \cite{B} (see also \cite{LL,FHL}).
In this section, we consider its full analogue and study the subalgebra of a full vertex algebra
generated by chiral and anti-chiral vectors.
%
%
%
Let $(F^1,Y^1,\va^1)$ and $(F^2,Y^2,\va^2)$  be full vertex algebras
and assume that the spectrum of $F^1$ is discrete
and the spectrum of $F^2$ is bounded below.
Define the linear map $Y(-,\uz):F^1 \otimes F^2 \rightarrow \End F^1 \otimes F^2[[z,\z,|z|^\R ]]$ by 
$Y(a\otimes b,\uz)=Y^1(a,\uz)\otimes Y^2(b,\uz)$ for $a\in F^1$ and $b \in F^2$.
Then, for $a,c \in F^1$ and $b,d \in F^2$,
$$Y(a\otimes b, \uz)c\otimes d
=\sum_{s,\s,r,\bar{r} \in \R}a(s,\s)c\otimes b(r,\bar{r})d\,z^{-s-r-2}\z^{-\s-\bar{r}-2}.$$
By (FV1), the coefficient of $z^k\z^{\bar{k}}$ is a finite sum for any $k,\bar{k}\in \R$.
Thus, $Y(-,\uz)$ is well-defined.
For any $h_0,\h_0 \in \R$,
set $(F^1\otimes F^2)_{h_0,\h_0}=\bigoplus_{a,\bar{a} \in \R} F_{a,\bar{a}}^1\otimes F_{h_0-a,\h_0-\bar{a}}^2$.
Since the spectrum of $F^2$ is bounded below, there exists $N\in \R$ such that
$(F^1\otimes F^2)_{h_0,\h_0}=\bigoplus_{a,\bar{a} \leq N} F_{a,\bar{a}}^1\otimes F_{h_0-a,\h_0-\bar{a}}^2$.
Since the spectrum of $F^1$ is discrete, the sum is finite.
Thus, $(F^1\otimes F^2)_{h_0,\h_0}^*=\bigoplus_{a,\bar{a} \in \R} (F_{a,\bar{a}}^1)^*\otimes (F_{h_0-a,\h_0-\bar{a}}^2)^*$,
which implies that $F^\vee=(F^1)^\vee\otimes (F^2)^\vee$.
Let $u_i \in (F^i)^\vee$ and $a_i,b_i,c_i \in F^i$ for $i=1,2$.
Since
$$
u_1\otimes u_2(Y(a_1\otimes a_2,\uz_1)Y(b_1\otimes b_2,\uz_2)c_1\otimes c_2)
=u_1(Y(a_1,\uz_1)Y(b_1,\uz_2)c_1)u_2(Y(a_2,\uz_1)Y(b_2,\uz_2)c_2),
$$
we have:
\begin{prop}
\label{tensor}
Let $(F^1,Y^1,\va^1)$ and $(F^2,Y^2,\va^2)$ be full vertex algebras.
If the spectrum of $F^1$ is discrete
and the spectrum of $F^2$ is bounded below,
then $(F^1 \otimes F^2, Y^1\otimes Y^2 ,\va^1 \otimes \va^2)$ is a full vertex algebra.
Furthermore, if the spectrum of $F^1$ and $F^2$ are bounded below (resp. discrete),
then the spectrum of $F^1\otimes F^2$ is also bounded below (resp. discrete).
\end{prop}

\begin{cor}\label{tensor_symmetric}
Let $(F^1,Y^1,\va^1)$ and $(F^2,Y^2,\va^2)$ be compact full vertex algebras.
Then, $(F^1 \otimes F^2, Y^1\otimes Y^2 ,\va^1 \otimes \va^2)$ is a compact full vertex algebra.
\end{cor}

\begin{rem}
\label{rem_well_tensor}
Although there exist cases in which the tensor product can be defined even if the assumption on the spectrum of Proposition 
\ref{tensor} is not satisfied, this assumption is necessary in general.
%
In the case of a vertex algebra $V$, (V4) is essentially replaced by the following condition:
\begin{itemize}
\item[(locality)]
For any $a,b \in V$, there exists $N\in \Z_{\geq 0}$ such that 
$(z-w)^N [Y(a,z),Y(b,w)]=0$.
\end{itemize}
However, in the case of a full vertex algebra, such an algebraic definition cannot be obtained and must be defined analytically by
the correlation functions. In order to guarantee the convergence of the correlation functions of the tensor product of full vertex algebras, 
$(F_1\otimes F_2)^\vee=F_1^\vee \otimes F_2^\vee$ must hold, and for this we make the assumption of Proposition \ref{tensor}.
\end{rem}

By Proposition \ref{graded_vertex} and Proposition \ref{conjugate},
we have:
\begin{cor}\label{chiral_tensor}
Let $V, W$ be a $\Z_{\geq 0}$-graded vertex algebras
such that $\dim V_n$ and $\dim W_n$ are finite for any $n\in \Z_{\geq 0}$.
Then, $V\otimes \bar{W}$ is a compact full vertex algebra,
where $\bar{W}$ is the conjugate full vertex algebra.
\end{cor}
Let $F$ be a full vertex algebra.
By Proposition \ref{vertex_algebra},
$\ker \D$ and $\ker D$ are subalgebras of $F$.
Let $\ker \D \otimes \ker D$ be the tensor product full vertex algebra.
Define the linear map $t:\ker \D \otimes \ker D \rightarrow F$
by $(a\otimes b)\mapsto a(-1,-1)b$ for $a\in \ker \D$ and $b \in \ker D$.
Then, we have:
\begin{prop}\label{ker_hom}
Let $F$ be a full vertex algebra.
Then, $t:\ker \D \otimes \ker D \rightarrow F$ is a full vertex algebra homomorphism.
\end{prop}
\begin{proof}
Let $a,c \in \ker \D$, $b,d \in \ker D$.
By Lemma \ref{hol_commute} and Lemma \ref{hol_commutator},
$$Y(a(-1,-1)b,\uz)=Y(a,z)Y(b,\z)=Y(b,\z)Y(a,z).$$
Thus, it suffices to show that
$t(a\otimes b(n,m)c\otimes d)=t(a\otimes b)(n,m)t(c\otimes d)$
for any $n,m\in \Z$.
By Lemma \ref{hol_commutator}
\begin{align*}
t(a\otimes b(n,m)c\otimes d)&=t(a(n,-1)c \otimes b(-1,m)d)\\
&=(a(n,-1)c)(-1,-1)b(-1,m)d\\
&=\sum_{i=0}\binom{n}{i}(-1)^i (a(n-i,-1)c(-1+i,-1)+c(-1+n-i,-1)a(i,-1))b(-1,m)d.
\end{align*}
Since $b(-1,m)d \in \ker D$, by Lemma \ref{hol_commute},
$t(a\otimes b(n,m)c\otimes d)=a(n,-1)c(-1,-1)b(-1,m)d
=a(n,-1)b(-1,m)c(-1,-1)d=t(a\otimes b)(n,m)t(c\otimes d).$
Thus, the assertion holds.
\end{proof}

We remark that if $\ker \D\otimes \ker D$ is simple, then the above map is injective.

\subsection{Full conformal vertex algebra}\label{sec_def_full_conformal}
In this section, we introduce the notion of a full conformal vertex algebra,
which is a generalization of a vertex operator algebra and a conformal vertex algebra \cite{B,FLM}.
An {\it energy-momentum tensor}
of a full vertex algebra is a pair of vectors
$\om \in F_{2,0}$ and $\omb\in F_{0,2}$ such that
\begin{enumerate}
\item
$\D \om=0$ and $D \omb=0$;
\item
There exist scalars $c, \bar{c} \in \C$ such that
$\om(3,-1)\om=\frac{c}{2} \va$,
$\omb(-1,3)\omb=\frac{\bar{c}}{2} \va$ and
$\om(k,-1)\om=\omb(-1,k)\omb=0$
for any $k=2$ or $k\in \Z_{\geq 4}$.
\item
$\om(0,-1)=D$ and $\omb(-1,0)=\D$;
\item
$\om(1,-1)|_{F_{t,\td}}=t$ and
$\omb(-1,1)|_{F_{t,\td}}=\td$ for any $t,\td \in \R$.
\end{enumerate}
We remark that $\{ \om(n,-1)\}_{n\in \Z}$ and $\{ \omb(-1,n)\}_{n\in \Z}$
satisfy the commutation relation of Virasoro algebra by Lemma \ref{hol_commutator}.
A {\it full conformal vertex algebra} is a pair of a full vertex algebra and its energy momentum tensor.

Let $(F,\om,\omb)$ be a full conformal vertex algebra
and $a \in \ker \D$.
Then, by Lemma \ref{hol_commute},
$\omb(1)a =0$. Thus, $\ker \D \subset \bigoplus_{n \in \Z} F_{n,0}$.
Since $\om \in \ker \D$, we have:
\begin{prop}
If $(F,\om,\omb)$ is a full conformal vertex algebra,
then $(\ker \D,\om)$ is a $\Z$-graded conformal vertex algebra.
\end{prop}

\section{Generalized full vertex algebras}\label{sec_generalized}
In this section, we define and study a generalized full vertex algebra,
which is a ``full'' analogy of the notion of a generalized vertex algebra
introduced in \cite{DL}.

\subsection{Definition of generalized vertex algebra}
We first recall the notion of generalized vertex algebra
introduced in \cite{DL}. 
We remark that in the original definition in \cite{DL} they use the Borcherds identity, however, 
in order to generalize it to non-chiral CFT we need to use generalized two-point correlation function (see Remark \ref{rem_vertex}).

For $\al_1,\al_2,\al_{12} \in \R$,
set
\begin{align}
z_1^{\al_1}z_2^{\al_2}(z_1-z_2)^{\al_{12}}|_{|z_1|>|z_2|}
&=z_1^{\al_1+\al_{12}}z_2^{\al_2}\sum_{i\geq 0} (-z_2/z_1)^i, \nonumber \\
z_1^{\al_1}z_2^{\al_2}(z_2-z_1)^{\al_{12}}|_{|z_2|>|z_1|}
&=z_1^{\al_1}z_2^{\al_2+\al_{12}} \sum_{i\geq 0} (-z_1/z_2)^i, \nonumber \\
(z_2+z_0)^{\al_1}z_2^{\al_2}z_0^{\al_{12}}|_{|z_2|>|z_0|}
&=z_0^{\al_{12}}z_2^{\al_2+\al_{1}} \sum_{i\geq 0}(z_0/z_2)^i, \label{eq_branch}
\end{align}
which are formal power series in $\C[[z_2/z_1]][z_1^\R,z_2^\R]$,
$\C[[z_1/z_2]][z_1^\R,z_2^\R]$ and $\C[[z_0/z_2]][z_0^\R,z_2^\R]$, respectively.

\begin{rem}\label{rem_single}
These notations do not conflict with the notation introduced in section \ref{sec_gen},
which represents series expansion in some regions.
In fact, if $\al_1,\al_2,\al_{12} \in \Z$,
then $z_1^{\al_1}z_2^{\al_2}(z_1-z_2)^{\al_{12}} \in \GCor_2$ and both notations give the same formal power series.
However, unless $\al_1,\al_2,\al_{12} \in \Z$,
$z_1^{\al_1}z_2^{\al_2}(z_1-z_2)^{\al_{12}}$ is not a single valued function.
Thus, in order to expand it, we have to choose a branch.
We decide to choose the branch given in \eqref{eq_branch}.
\end{rem}

A generalized vertex algebra is a real finite-dimensional vector space
$H$ equipped with a non-degenerate symmetric bilinear form 
$$(-,-):H\times H \rightarrow \R$$
and an $\R \times H$-graded $\C$-vector space
$\Om=\bigoplus_{t \in \R,\al \in H} \Om_{t}^\al$ equipped with a linear map 
$$\hY(-,z):\Om \rightarrow \End \Om[[z^\R]],\; a\mapsto \hY(a,z)=\sum_{r \in \R}a(r)z^{-r-1}$$
and an element $\va \in \Om_{0}^0$
satisfying the following conditions:
%
\begin{enumerate}
\item[GV1)]
For any $\al,\be \in H$ and $a \in \Om^\al$, $b \in \Om^\be$,
$z^{(\al,\be)}\hY(a,z)b \in \Om((z))$;
\item[GV2)]
$\Om_{t}^\al=0$ unless $(\al,\al)/2+t \in \Z$;
\item[GV3)]
For any $a \in \Om$, $\hY(a,\uz)\va \in \Om[[z,\z]]$ and $\lim_{z \to 0}\hY(a,\uz)\va = a(-1,-1)\va=a$;
\item[GV4)]
$\hY(\va,\uz)=\mathrm{id} \in \End \Om$;
\item[GV5)]
For any $\al_i \in M_\Om$ and  $a_i \in \Omega^{\al_i}$ ($i=1,2,3$) and $u \in \Omega^\vee=\bigoplus_{t \in \R, \al \in H}
(\Om_{t}^\al)^*$,
 there exists $\mu(z_1,z_2) \in \GCor_2^{\text{hol}}$ such that
\begin{align*}
(z_1-z_2)^{(\al_1,\al_2)}z_1^{(\al_1,\al_3)}z_2^{(\al_2,\al_3)}|_{|z_1|>|z_2|}
u(\hY(a_1,z_1)\hY(a_2,z_2)a_3) &= \mu(z_1,z_2)|_{|z_1|>|z_2|},\\
z_0^{(\al_1,\al_2)}(z_2+z_0)^{(\al_1,\al_3)}z_2^{(\al_2,\al_3)}|_{|z_2|>|z_0|}
u(\hY(\hY(a_1,z_0)a_2,z_2)a_3) &= \mu(z_0+z_2,z_2)|_{|z_2|>|z_0|},\\
(z_2-z_1)^{(\al_1,\al_2)}z_1^{(\al_1,\al_3)}z_2^{(\al_2,\al_3)}|_{|z_2|>|z_1|}
u(\hY(a_2,z_2)\hY(a_1,z_1)a_3) &= \mu(z_1,z_2)|_{|z_2|>|z_1|};
\end{align*}
\item[GV6)]
$\Om_{t}^\al(r)\Om_{t'}^\be \subset \Om_{t+t'-r-1}^{\al+\be}$ for any $r,t,t' \in \R$
and $\al,\be \in H$;
\item[GV7)]
For any $\al \in H$, there exists $N_\al \in \R$ such that $\Om_t^\al=0$ for any $t \leq N_\al$.
\end{enumerate}

\begin{rem}
As remarked in \ref{rem_single},
a generalized correlation function $u(\hY(a_1,z_1)\hY(a_2,z_2)a_3)$ 
is not a single-valued holomorphic function and no longer an analytic continuation of \\
$u(\hY(a_2,z_2)\hY(a_1,z_1)a_3)$ along any path. But, the monodromy is controlled by the $H$-grading.
\end{rem}
\begin{rem}
In the original definition in \cite{DL} they use an $\R/2\Z$-valued bilinear form $H \times H \rightarrow \R/2\Z$
instead of $H \times H \rightarrow \R$.
In fact, for the definition, we only need this $\R/2\Z$-valued bilinear form (see for example Lemma \ref{even_omega}), however, for our purpose it is convenient to define a generalized vertex algebra in this way.
We also remark that (GV7) is assumed for the sake of simplicity (It seems that all generalized vertex algebras which naturally arise satisfy (GV7)). We may drop it and it is not assumed in the original definition.
\end{rem}

\subsection{Definition of generalized full vertex algebra}\mbox{}
It is now straight forward to generalize the definition of a generalized vertex algebra to 
a  (non-chiral) generalized full vertex algebra.


A {\it generalized full vertex algebra} is a real finite-dimensional vector space
$H$ equipped with a non-degenerate symmetric bilinear form 
$$(-,-):H\times H \rightarrow \R$$
and an $\R^2\times H$-graded $\C$-vector space
$\Om=\bigoplus_{t,\td \in \R,\al \in H} \Om_{t,\td}^\al$ equipped with a linear map 
$$\hY(-,\uz):\Om \rightarrow \End \Om[[z^\R,\z^\R]],\; a\mapsto \hY(a,\uz)=\sum_{r,s \in \R}a(r,s)z^{-r-1}\z^{-s-1}$$
and an element $\va \in \Om_{0,0}^0$ 
satisfying the following conditions:
%
\begin{enumerate}
\item[GFV1)]
For any $\al,\be \in H$ and $a \in \Om^\al$, $b \in \Om^\be$,
$z^{(\al,\be)}\hY(a,z)b \in \Om((z,\z,|z|^\R))$;
\item[GFV2)]
$\Om_{t,\td}^\al=0$ unless $(\al,\al)/2+t-\td \in \Z$;
\item[GFV3)]
For any $a \in \Om$, $\hY(a,\uz)\va \in \Om[[z,\z]]$ and $\lim_{z \to 0}\hY(a,\uz)\va = a(-1,-1)\va=a$;
\item[GFV4)]
$\hY(\va,\uz)=\mathrm{id} \in \End \Om$;
\item[GFV5)]
For any $\al_i \in H$ and  $a_i \in \Omega^{\al_i}$ ($i=1,2,3$) and $u \in \Omega^\vee=\bigoplus_{t,\td\in \R,\al\in H}(\Om_{t,\td}^\al)^*$,  there exists $\mu(z_1,z_2) \in \GCor_2$ such that
\begin{align*}
(z_1-z_2)^{(\al_1,\al_2)}z_1^{(\al_1,\al_3)}z_2^{(\al_2,\al_3)}|_{|z_1|>|z_2|}
u(\hY(a_1,\uz_1)\hY(a_2,\uz_2)a_3) &= \mu(z_1,z_2)|_{|z_1|>|z_2|},\\
z_0^{(\al_1,\al_2)}(z_2+z_0)^{(\al_1,\al_3)}z_2^{(\al_2,\al_3)}|_{|z_2|>|z_0|}
u(\hY(\hY(a_1,\uz_0)a_2,\uz_2)a_3) &= \mu(z_0+z_2,z_2)|_{|z_2|>|z_0|},\\
(z_2-z_1)^{(\al_1,\al_2)}z_1^{(\al_1,\al_3)}z_2^{(\al_2,\al_3)}|_{|z_2|>|z_1|}
u(\hY(a_2,\uz_2)\hY(a_1,\uz_1)a_3) &= \mu(z_1,z_2)|_{|z_2|>|z_1|};
\end{align*}
\item[GFV6)]
$\Om_{t,\td}^\al(r,s)\Om_{t',\td'}^\be \subset \Om_{t+t'-r-1,\td+\td'-s-1}^{\al+\be}$ for any $r,s,t,\td,t',\td' \in \R$
and $\al,\be \in H$;
\item[GFV7)]
For any $\al \in H$, there exists $N_\al \in \R$ such that $\Om_{t,\td}^\al=0$ for any $t \leq N_\al$ or
$\td \leq N_\al$.
\end{enumerate}

Let $(\Om,H)$ be a generalized full vertex algebra
and set $$\Om^\al=\bigoplus_{t,\td\in\R} \Om_{t,\td}^\al.
$$ for $\al \in H$,
and $M_\Om$ be a subgroup of $H$ generated by $\{\al \in H\;|\; \Om^\al \neq 0  \}$.
Let $D$ and $\D$ denote the endomorphism of $\Om$
defined by $Da=a(-2,-1)\bm{1}$ and $\D a=a(-1,-2)\va$ for $a\in \Om$,
i.e., $$\hY(a,z)\va=a+Daz+\D a\z+\dots \in \Om[[z,\z]].$$
Let $a \in \Om^\al$ and $b \in \Om^\be$ for $\al,\be \in M_\Om$.
Since $z^{(\al,\be)}\hY(a,\uz)b \in \Om((z,\z,|z|^\R))$,
$\lim_{z \to -z} z^{(\al,\be)}Y(a,\uz)b$ is well-defined.
Then, similarly to the case of full vertex algebras, we have:
\begin{prop}\label{skew_symmetry}
Let $\Om$ be a generalized full vertex algebra.
For $v \in \Om$ and $\al,\be\in M_\Om$,
 $a \in \Om^\al$, $b \in \Om^\be$, the following properties hold:
\begin{enumerate}
\item
$\hY(Dv,\uz)=\dz \hY(v,\uz)$ and $\hY(\D v,\uz)=\ddz \hY(v,\uz)$;
\item
$D\va=\D\va=0$;
\item
$[D,\D]=0$;
\item
$z^{(\al,\be)}\hY(a,\uz)b=\exp(zD+\z\D)\lim_{z \to -z} z^{(\al,\be)}\hY(b,\uz)a$;
\item
$\hY(\D v,\uz)=[\D,\hY(v,\uz)]$ and $\hY(Dv,\uz)=[D,\hY(v,\uz)]$.
\end{enumerate}
\end{prop}
\begin{proof}
The proof of Proposition \ref{translation} also works
for (1), (2), (3).
Thus, we only prove (4) and (5).

Let $u \in \Om^\vee$ and $a \in \Om^\al$ and $b\in \Om^\be$
and $\mu(z_1,z_2)\in \GCor_2$
satisfy 
\begin{align*}
(z_1-z_2)^{(\al,\be)}u(\hY(a,\uz_1)\hY(b,\uz_2)\va)=\mu(z_1,z_2) |_{|z_1|>|z_2|}.
\end{align*}
Since
\begin{align*}
\mu(z_0+z_2,z_2)|_{|z_2|>|z_0|} 
&=z_0^{(\al,\be)}u(\hY(\hY(a,\uz_0)b,\uz_2)\va) \in U(z_2,z_0)
\end{align*}
and the right-hand-side contains only the positive power of $z_2$ and $\z_2$,
$z_0^{(\al,\be)}u(\hY(\hY(a,\uz_0)b,\uz_2)\va) \in 
\C[z_2,\z_2][z_0^{\pm},\z_0^{\pm},|z_0|^\R]$.
Set $p(z_0,z_2)=z_0^{(\al,\be)}u(\hY(\hY(a,\uz_0)b,\uz_2)\va)$.
By (GFV5) and setting $z_0'=z_2-z_1$, we have
$$u(\hY(\hY(b,\uz_0')a,\uz_1)\va) = \mu(z_1,z_0'+z_1)|_{|z_1|>|z_0'|}
=p(-z_0',z_1+z_0')|_{|z_1|>|z_0'|}.$$
Thus,
\begin{align*}
z_0^{(\al,\be)}u(\hY(a,\uz_0)b)&=p(z_0,0)= \lim_{z_1\to 0} \exp(z_0\frac{d}{dz_1}+\z_0\frac{d}{d\z_1})p(z_0,z_1-z_0)\\
&=\lim_{z_1\to 0} \exp(z_0\frac{d}{dz_1}+\z_0\frac{d}{d\z_1})
\lim_{z_0'\to -z_0} z_0'^{(\al,\be)}
u(\hY(\hY(b,\uz_0')a,\uz_1)\va) \\
&=\lim_{\substack{z_1\to 0\\z_0'\to-z_0}} u(\hY(\exp(-z_0'D-\z_0'\D) z_0'^{(\al,\be)}Y(b,z_0')a,\uz_1)\va) \\
&=u(\exp(z_0D+\z_0\D)\lim_{z_0'\to -z_0} z_0'^{(\al,\be)}Y(b,\uz_0')a).
\end{align*}

Finally,
\begin{align*}
z^{(\al,\be)}\hY&(Da,\uz)b+(\al,\be)z^{(\al,\be)-1}\hY(a,\uz)b\\
&=\frac{d}{dz}z^{(\al,\be)}\hY(a,\uz)b\\
&=\frac{d}{dz}\exp(Dz+\D\z)(-z)^{(\al,\be)}\hY(b,-\uz)a \\
&= D\exp(Dz+\D\z)(-z)^{(\al,\be)}\hY(b,-\uz)a \\
&+(\al,\be)z^{-1}\exp(Dz+\D\z)(-z)^{(\al,\be)}Y(b,-\uz)a
-\exp(Dz+\D\z)(-z)^{(\al,\be)}Y(Db,-\uz)a\\
&=z^{(\al,\be)}D\hY(a,\uz)b-z^{(\al,\be)}\hY(a,\uz)Db
+(\al,\be)z^{(\al,\be)-1}\hY(a,\uz)b.
\end{align*}
Thus, the assertion holds.
\end{proof}

A  homomorphism from a generalized full vertex algebra $(\Om_1,\hY_1,\va_1,H_1)$ to
a generalized full vertex algebra $(\Om_2,\hY_2,\va_2,H_2)$ is
a pair of a linear map $\psi :\Om_1\rightarrow \Om_2$ and an $\R$-linear isomorphism
$\psi':H_1\rightarrow H_2$
such that:
\begin{enumerate}
\item
$\psi'$ is isometric;
\item
$\psi((\Om_1)_{t,\td}^\al)\subset (\Om_2)_{t,\td}^{\psi'(\al)}$ for any $t,\td\in \R$ and $\al \in M_{\Om_1}$;
\item
$\psi(\va_1)=\va_2$;
\item
$\psi(\hY_1(a,\uz)b)=\hY_2(\psi(a),\uz)\psi(b)$
for any $a,b\in \Om_1$.
\end{enumerate}

A subalgebra of a generalized full vertex algebra $\Om$
is an $\R^2\times H$-graded subspace $\Om' \subset \Om$
such that $\va \in \Om'$ and $a(r,s)b \in \Om'$ for any $r,s\in \R$
and $a,b \in \Om'$.
\begin{lem}\label{subgroup}
Let $\Om$ be a generalized full vertex algebra.
Then, for a subgroup $A \subset M_\Om$,
$\Om^A = \bigoplus_{\al \in A} \Om^\al$ is a subalgebra of $\Om$.
\end{lem} 

The following lemma is clear from the definition:
\begin{lem}
\label{even_omega}
Let $\al_i \in M_\Om$ and $a_i\in \Om^\al_i$ for $i=1,2,3$.
Suppose that $(\al_i,\al_j) \in \Z$ for $i\neq j$.
Then, for any $u \in \Om^\vee$
 there exists $\mu(z_1,z_2)\in \GCor_2$ such that
\begin{align*}
u(\hY(a_1,\uz_1)\hY(a_2,\uz_2)a_3) &= \mu(z_1,z_2)|_{|z_1|>|z_2|}\\
u(\hY(\hY(a_1,\uz_0)a_2,\uz_2)a_3) &= \mu(z_0+z_2,z_2)|_{|z_2|>|z_0|}\\
(-1)^{(\al_1,\al_2)}u(\hY(a_2,\uz_2)\hY(a_1,\uz_1)a_3) &= \mu(z_1,z_2)|_{|z_2|>|z_1|}.
\end{align*}
In particular, 
if a subgroup $A \subset M_\Om$ satisfies $(\al,\al') \in 2\Z$ for any $\al,\al'\in A$,
then $\Om^A = \bigoplus_{\al \in A} \Om^\al$ is a full vertex algebra.
\end{lem}

Let $a\in \Om^0$ satisfy $\D a=0$.
Since $\hY(\D a,\uz)=\ddz \hY(a,\uz)=0$,
$\hY(a,\uz)=\sum_{n\in \Z}a(n,-1)z^{-n-1}$.
Thus, similarly to the proof of Lemma \ref{hol_commutator}
and Lemma \ref{hol_commute},
we have:
\begin{lem}\label{zero_commutator}
Let $a\in \Om^0$ satisfy $\D a=0$.
Then, for any $b \in \Om$,
$$[a(n,-1),\hY(b,\uz)]
=\sum_{i \geq 0} \binom{n}{i}
\hY(a(i,-1)b,\uz)z^{n-i}.
$$
Furthermore,
if $D b=0$,
then $a(i,-1)b=0$ for any $i\geq 0$.
\end{lem}

A {\it generalized full conformal vertex algebra} is a generalized full vertex algebra
$\Om$ with distinguished vectors $\om \in \Om_{2,0}^0$ and 
$\omb\in \Om_{0,2}^0$ such that
\begin{enumerate}
\item
$\D \om=0$ and $D \omb=0$;
\item
There exist scalars $c, \bar{c} \in \C$ such that
$\om(3,-1)\om=\frac{c}{2} \va$,
$\omb(-1,3)\omb=\frac{\bar{c}}{2} \va$ and
$\om(k,-1)\om=\omb(-1,k)\omb=0$
for any $k=2$ or $k\in \Z_{\geq 4}$;
\item
$\om(0,-1)=D$ and $\omb(-1,0)=\D$;
\item
$\om(1,-1)|_{\Om_{t,\td}^\al}=t\,\mathrm{id}_{\Om_{t,\td}^\al}$ and
$\omb(-1,1)|_{\Om_{t,\td}^\al}=\td\,\mathrm{id}_{\Om_{t,\td}^\al}$ for any $t,\td \in \R$ and $\al \in M_\Om$.
\end{enumerate}
We remark that $\{ \om(n,-1)\}_{n\in \Z}$ and $\{ \omb(-1,n)\}_{n\in \Z}$
satisfy the commutation relation of the Virasoro algebra by Lemma \ref{zero_commutator}.
The pair $(\om,\omb)$ is called an energy-momentum tensor of the generalized full conformal vertex algebra in this paper.

\subsection{Locality of generalized full vertex algebra}
The most difficult part in the construction of a generalized full vertex algebra is to verify the condition (GFV5).
In the following proposition, (GFV5) is replaced by the conditions (GFL1), (GFL2) and (GFL3), which are easier to prove.
\begin{prop}\label{general_locality}
Let $(\Om,\hY,\va,H)$ satisfy (GFV1), (GFV2), (GFV3), (GFV4), (GFV6) and (GFV7)
and $D,\D \in \End \,\Om$ be linear maps.
We assume that the following conditions hold:
\begin{enumerate}
\item[GFL1)]
$[D,\D]=0$ and $D\va=\D\va=0$;
\item[GFL2)]
$[D,\hY (a,\uz)]=\dz \hY(a,\uz)$ and $[\D,\hY(a,\uz)]=\ddz \hY(a,\uz)$ for any
$a\in \Om$;
\item[GFL3)]
For any $\al_i \in M_\Om$ and  $a_i \in \Omega^{\al_i}$ ($i=1,\dots,3$)
and $u \in \Om^\vee$,
there exists $\mu(z_1,z_2) \in \GCor_2$ such that
\begin{align*}
(z_1-z_2)^{(\al_1,\al_2)}z_1^{(\al_1,\al_3)}z_2^{(\al_2,\al_3)}|_{|z_1|>|z_2|}
u(\hY(a_1,\uz_1)\hY(a_2,\uz_2)a_3) = \mu(z_1,z_2)|_{|z_1|>|z_2|},\\
(z_2-z_1)^{(\al_1,\al_2)}z_1^{(\al_1,\al_3)}z_2^{(\al_2,\al_3)}|_{|z_2|>|z_1|}
u(\hY(a_2,\uz_2)\hY(a_1,\uz_1)a_3) = \mu(z_1,z_2)|_{|z_2|>|z_1|}.
\end{align*}
\end{enumerate} 
Then, $\Om$ is a generalized full vertex algebra.
\end{prop}
\begin{proof}
Let $a_i \in \Om_{t_i,\td_i}^{\al_i}$ and $u \in (\Om_{t_0,\td_0}^{\al_0})^*$
for $\al_i \in M_\Om$, $t_i,\td_i \in \R$ and $i=0,1,2$.

First, we prove the skew-symmetry,
that is,
$$z^{(\al_1,\al_2)}Y(a_2,\uz)a_1=\exp(Dz+\D \z) \lim_{z \to -z} z^{(\al_1,\al_2)}Y(a_1,\uz)a_2.
$$
Since $DY(a_2,z)\va=\dz Y(a_2,z)\va$, we have $Y(a_2,z)\va=\exp(Dz+\D\z)a_2$,
which implies that $Da_2=a_2(-2,-1)\va\in F_{t_2+1,\td_2}$
and thus $D\Om_{t,\td}^\al \subset \Om_{t+1,\td}^\al$ and $\D \Om_{t,\td}^\al \subset \Om_{t,\td+1}^\al$ for any $t,\td\in \R$ and $\al \in M_\Om$.
Then,
\begin{align*}
u(\hY(a_1,\uz_1)\hY(a_2,\uz_2)\va)=u(\hY(a_1,\uz_1)\exp(Dz_2+\D\z_2)a_2)\\
=\lim_{z_{12} \to (z_1-z_2)|_{|z_1|>|z_2|}} u(\exp(Dz_2+\D\z_2)\hY(a_1,\uz_{12})a_2).
\end{align*}
Set $t=t_1+t_2-t_0$ and $\td=\td_1+\td_2-\td_0$.
Then, by (GFV6),
\begin{align*}
u(\exp(Dz_2+\D\z_2)\hY(a_1,\uz_{12})a_2)
&=\sum_{s_1,\s_1 \in \R}\sum_{n,\n \in \Z_{\geq 0}}\frac{1}{n!\n!}
u(D^n\D^\n a_1(s_1,\s_1)a_2)z_{12}^{-\s_1-1}\z_{12}^{-\s_1-1}z_2^n \z_2^\n \\
&=\sum_{n,\n \in \Z_{\geq 0}}\frac{1}{n!\n!}
u(D^n\D^\n a_1(h+n-1,\h+\n-1)a_2)z_{12}^{-t-n}\z_{12}^{-\td-\n}z_2^n \z_2^\n.
\end{align*}
By (GFV1), there exists an integer $N$ such that $a_1(t+n-1,\td+\n-1)a_2=0$ for any $n \geq N$
or $\n \geq N$.
Thus,
$z_{12}^{N+t}\z_{12}^{N+\td}  u(\exp(Dz_2+\D\z_2)a_2)\hY(a_1,\uz_{12})a_2) \in \C[z_{12},z_2,\z_{12},\z_2]$.
By (GFV1), we may assume that
$(\al_1,\al_2)-t+\td \in \Z$.
Set 
$$
p(z_{12},z_2)=z_{12}^{(\al_1,\al_2)}u(\exp(Dz_2+\D\z_2)a_2)\hY(a_1,\uz_{12})a_2),
$$
which is a polynomial in $\C[z_{12}^\pm,\z_{12}^\pm,|z_{12}|^\R,z_2,\z_2]$
by $z_{12}^{t}\z_{12}^{\td}=
 z_{12}^{t-\td} |z_{12}|^{\td}$.
Then, by (GFL3), $p(z_{12},z_2)$ satisfies
\begin{align*}
\lim_{z_{12} \to (z_1-z_2)|_{|z_1|>|z_2|}}p(z_{12},z_2)
&=(z_1-z_2)^{(\al_1,\al_2)}u(\hY(a_1,\uz_1)\hY(a_2,\uz_2)\va)\\
\lim_{z_{12} \to (-z_2+z_1) |_{|z_2|>|z_1|}}p(z_{12},z_2)
&=(z_2-z_1)^{(\al_1,\al_2)}u(\hY(a_2,\uz_2)\hY(a_1,\uz_1)\va).
\end{align*}
By taking $z_1 \rightarrow 0$,
we have
\begin{align*}
z_2^{(\al_1,\al_2)}u(\hY(a_2,\uz_2)a_1)
=p(-z_2,z_2)=\lim_{z_{12}\to -z_2}z_{12}^{{(\al_1,\al_2)}}u(\exp(Dz_2+\D\z_2)\hY(a_1,\uz_{12})a_2).
\end{align*}
Thus, the skew-symmetry holds.

Now, we will show (GFV5).
By the assumption, there exists $\mu(z_1,z_2) \in \GCor_2$
such that
\begin{align*}
(z_1-z_2)^{(\al_1,\al_2)}z_1^{(\al_1,\al_3)}z_2^{(\al_2,\al_3)}|_{|z_1|>|z_2|}
u(\hY(a_1,\uz_1)\hY(a_2,\uz_2)a_3) = \mu(z_1,z_2)|_{|z_1|>|z_2|}.
\end{align*}
By the skew-symmetry,
\begin{align*}
(z_1-&z_2)^{(\al_1,\al_2)}z_1^{(\al_1,\al_3)}z_2^{(\al_2,\al_3)}|_{|z_1|>|z_2|}
u(\hY(a_1,\uz_1)\hY(a_2,\uz_2)a_3) \\
&=(z_1-z_2)^{(\al_1,\al_2)}z_1^{(\al_1,\al_3)}|_{|z_1|>|z_2|}
u(\hY(a_1,\uz_1)\exp(Dz_2+\D\z_2)\lim_{\uz'_2\to -\uz_2} {z'}_2^{(\al_2,\al_3)}\hY(a_3,\uz'_2)a_2)\\
&=
\lim_{\substack{z_{12}\to (z_1-z_2)|_{|z_1|>|z_2|} \\ {z'}_2\to -z_2}}
z_{12}^{(\al_1,\al_2)}(z_{12}-z_2')^{(\al_1,\al_3)}
u(\exp(-Dz'_2-\D\z'_2) \hY(a_1,\uz_{12}){z'}_2^{(\al_2,\al_3)}\hY(a_3,\uz'_2)a_2).
\end{align*}
Since $\Om_{t,\td}^\al=0$ for sufficiently small $t$ or $\td$,
$u(\exp(-Dz'_2-\D\z'_2)-)$ is in $\Om^\vee[z_2',\z_2']$, i.e., a finite sum.
Since
$$\mu(z_{12}-z_2', -z'_2)|_{|z_{12}|>|z_2'|}=(z_{12}-z_2')^{(\al_1,\al_3)} z_{12}^{(\al_1,\al_2)}
u(\exp(-Dz'_2-\D\z'_2) \hY(a_1,\uz_{12}){z'}_2^{(\al_2,\al_3)}\hY(a_3,\uz'_2)a_2),$$
by (GFL3) and the skew-symmetry,
we have
\begin{align*}
\mu(z_{12}-&z_2', -z'_2)|_{|z_2'|>|z_{12}|}\\
&=(z_2'-z_{12})^{(\al_1,\al_3)} z_{12}^{(\al_1,\al_2)}
u(\exp(-Dz'_2-\D\z'_2)\hY(a_3,\uz'_2)\hY(a_1,\uz_{12}){z'}_2^{(\al_2,\al_3)}a_2)\\
&=(1-z_{12}/z_2')^{(\al_1,\al_3)} z_{12}^{(\al_1,\al_2)}
u(\exp(-Dz'_2-\D\z'_2){z'}_2^{(\al_2+\al_1,\al_3)} \hY(a_3,\uz'_2)\hY(a_1,\uz_{12})a_2)\\
&=(1-z_{12}/z_2')^{(\al_1,\al_3)} z_{12}^{(\al_1,\al_2)}
\lim_{\uz_2\to -\uz_2'} u({z}_2^{(\al_2+\al_1,\al_3)}\hY(\hY(a_1,\uz_{12})a_2),\uz_2)a_3)\\
&=\lim_{\uz_2\to -\uz_2'}(z_2+z_{12})^{(\al_1,\al_3)} z_{12}^{(\al_1,\al_2)}{z}_2^{(\al_2,\al_3)}|_{|z_2|>|z_{12}|}
u(\hY(\hY(a_1,\uz_{12})a_2),\uz_2)a_3).
\end{align*}
Thus, we have (GFV5).
\end{proof}

\subsection{Standard construction}\label{sec_standard}
From a lattice, an example of a generalized vertex algebra is  constructed in \cite{DL},
which is a generalization of the construction in \cite{B,FLM}.
They call it a {\it generalized lattice vertex algebra}.
In this section, we generalize it to non-chiral setting, which plays an essential role in this paper.

Let $H$ be a real finite-dimensional vector space equipped with a non-degenerate symmetric bilinear form 
$$(-,-)_\lat:H\times H \rightarrow \R.$$
Let $P(H)$ be a set of $\R$-linear maps $p\in \End H$ such that:
\begin{enumerate}
\item[P1)]
$p^2=p$, that is, $p$ is a projection;
\item[P2)]
The subspaces $\ker(1-p)$ and $\ker(p)$ are orthogonal to each other.
\end{enumerate}
Let $P_>(H)$ be a subset of $P(H)$
consisting of $p \in P(H)$ such that:
\begin{enumerate}
\item[P3)]
$\ker(1-p)$ is positive-definite and $\ker(p)$ is negative-definite.
\end{enumerate}
For $p\in P(H)$, set $\p=1-p$ and $H_l=\ker(\p)$ and $H_r=\ker(p)$.
We will construct a generalized full vertex algebra $G_{H,p}$ for each $p\in P(H)$.

Let $p\in P(H)$.
Define the new bilinear forms $(-,-)_p: H\times H \rightarrow \R$
by 
\begin{align}
(h,h')_p=(ph,ph')_\lat-(\p h,\p h')_\lat
\end{align}
for $h,h' \in H$.
By (P1) and (P2), $(-,-)_p$ is non-degenerate.
Let $\hat{H}^p=\bigoplus_{n \in \Z} H\otimes t^n \oplus \C c$ be the affine Heisenberg Lie algebra associated with $({H},(-,-)_p)$
and $\hat{H}_{\geq 0}^p=\bigoplus_{n \geq 0} H\otimes t^n \oplus \C c$ a subalgebra of $\hat{H}^p$.
Define the action of $\hat{H}_{\geq 0}^p$ on the group algebra of $H$, 
$\C[H]=\bigoplus_{\al \in H} \C e_\al$, by
\begin{align*}
c e_\alpha&=e_\al \\
h\otimes t^n e_\alpha &=
\begin{cases}
0, &n \geq 1,\cr
(h,\al)_p e_\alpha, &n = 0
\end{cases}
\end{align*}
for $\al \in H$.
Let $G_{H,p}$ be the $\hat{H}^p$-module induced from $\C[H]$.
Denote by $h(n)$ the action of $h \otimes t^n$ on $G_{H,p}$ for $n \in \Z$.
For $h \in H$, set 
\begin{align*}
h(\uz) &= \sum_{n \in \Z}( (ph)(n)z^{-n-1}+ (\p h)(n)\z^{-n-1}) \in
\End G_{H,p}[[z^\pm,\z^\pm]]\\
h^+(\uz) &=\sum_{n \geq 0}( (ph)(n)z^{-n-1}+ (\p h)(n)\z^{-n-1}) \\
h^-(\uz) &=\sum_{n \geq 0}( (ph)(-n-1)z^{n}+ (\p h)(-n-1)\z^{n}).\\
E^+(h,\uz)&=
\exp\biggl(-\sum_{n\geq 1}(\frac{p h(n)}{n}z^{-n}+\frac{\p h(n)}{n}\z^{-n})
\biggr)\\
E^-(h,\uz)&=\exp\biggl(\sum_{n\geq 1}(\frac{p h(-n)}{n}z^{n} +\frac{\p h(-n)}{n}\z^{n})
\biggr).
\end{align*}
For $h_r \in H_r$ and $h_l \in H_l$,
$h_r(\uz)$ and $h_l(\uz)$ are denoted by $h_l(z)$ and $h_r(\z)$, respectively.

Then, similarly to the case of a lattice vertex algebra \cite{FLM}, we have:
\begin{lem}\label{lattice_commutator}
For any $h_1, h_2 \in H$,
$$E^+(h_1,\uz_1)E^-(h_2,\uz_2)
= \biggl( \sum_{n,\n \geq 0} \binom{(ph_1,ph_2)_p}{n}\binom{(\p h_1,\p h_2)_p}{\n} (z_2/z_1)^{n}(\z_2/\z_1)^{\n}\biggr) E^-(h_2,\uz_2) E^+(h_1,\uz_1).$$
\end{lem}
We remark that the formal power series $\sum_{n,\n \geq 0} \binom{(ph_1,ph_2)_p}{n}\binom{(\p h_1,\p h_2)_p}{\n} (z_2/z_1)^{n}(\z_2/\z_1)^{\n}$ is equal to $(1-z_2/z_1)^{(ph_1,ph_2)_p} (1-\z_2/\z_1)^{(\p h_1,\p h_2)_p}|_{|z_1|>|z_2|}.$

Let $\al \in H$.
Denote by $l_{e_\al} \in \End \C[H]$ the left multiplication by ${e_\al}$ and define the linear map
$z^{p\alpha} \z^{\p \alpha} :\C[H] \rightarrow  \C[H][z^\R,\z^\R]$
by $z^{p\alpha} \z^{\p \alpha}e_\be= z^{(p\alpha,p\beta)_p}\z^{(\p\al,\p\be)_p}e_\be$ 
for $\beta \in H$.
Then, set
\begin{align*}
e_\al(\uz)&=E^-(\al,\uz)E^+(\al,\uz)l_{e_\al} z^{p \alpha} \z^{\p \alpha} \in \End\, G_{H,p}[[z^\pm,\z^\pm]][z^\R,\z^\R].
\end{align*}

By Poincar{\'e}-Birkhoff-Witt theorem,
$G_{H,p}$ is spanned by 
$$\{h_l^1(-n_1-1)\dots h_l^l(-n_l-1)h_r^1(-\n_1-1)\dots h_r^k(-\n_k-1)e_\al \},
$$
where $h_l^i \in H_l$, $n_i \in \Z_{\geq 0}$ and $h_r^j \in H_r$, $\n_j \in \Z_{\geq 0}$
 for any $1\leq i \leq l$ and $1\leq j \leq k$ and $\al \in H$.
Then, a map
$\hY:G_{H,p} \rightarrow \End\, G_{H,p}[[z^\pm,\z^\pm]][z^\R,\z^\R]$ is defined inductively as follows:
For $\al \in H$,
define $\hY({e_\al},\uz)$ by $\hY({e_\al},\uz)={e_\al}(\uz)$.
Assume that $\hY(v,\uz)$ is already defined for $v \in G_{H,p}$.
Then, for $h_r \in H_r$ and $h_l \in H_l$ and $n,\n \in \Z_{\geq 0}$,
$\hY(h_l(-n-1)v,\uz)$ and $\hY(h_r(-\n-1)v,\uz)$ is defined by 
\begin{align*}
\hY(h_l(-n-1)v,\uz)&=\Bigl(\frac{1}{n!}\frac{d}{dz}^n h_l^-(z)\Bigr)\hY(v,\uz)+\hY(v,\uz)\Bigl(\frac{1}{n!}\frac{d}{dz}^n h_l^+(z)\Bigr) \\
\hY(h_r(-\n-1)v,\uz)&=\Bigl(\frac{1}{\n!}\frac{d}{d\z}^n h_r^-(\z)\Bigr)\hY(v,\uz)+\hY(v,\uz)\Bigl(\frac{1}{\n!}\frac{d}{d\z}^n h_r^+(\z)\Bigr).
\end{align*}

Set
\begin{align*} 
\va&=1\otimes e_0, \\
\omega_{H_l}&  = \frac{1}{2}\sum_{i=1}^{\dim H_l} h_l^i (-1)h_l^i,\\
\bar{\om}_{H_r}& = \frac{1}{2} \sum_{j=1}^{\dim H_r} h_r^j (-1)h_r^j,
\end{align*}
where $h_l^i$ and $h_r^j$ is an orthonormal basis of $H_l\otimes_{\R} \C$ and $H_r\otimes_{\R} \C$ with respect to the bilinear form $(-,-)_p$.
Set $G=G_{H,p}$ and $$G_{t,\td}^\al=\{v \in G\;|\; \omega(1,-1)v=t v, \bar{\om}(-1,1)v=\td v, h(0)v=(\al,h)_p v \fora h \in H \}$$
for $t,\td \in \R$ and $\al \in H$.

For $\al \in H$ and $n,m\in \Z_{\geq 0}$,
it is easy to show that 
$G_{\frac{1}{2}(p\al,p\al)_p+n,\frac{1}{2}(\p\al,\p\al)_p+m}^\al$
is spanned by
$\{h_l^1(-i_1)\dots h_l^k(-i_k)h_r^1(j_1)\dots h_r^l(j_l)e_\al\}$,
where $k,l \in \Z_{\geq 0}$, $h_l^a \in H_l$, $h_r^b \in H_r$,
 $i_a,j_b \in \Z_{\geq 1}$,
$i_1+\dots+i_k=n$ and $j_1+\dots+j_l=m$ for any $a=1,\dots,k$ and $b=1,\dots,l$.
Then,
$$G=\bigoplus_{\al\in H}\bigoplus_{n,m \in \Z_{\geq 0}}\Om_{\frac{1}{2}(p\al,p\al)_p+n,\frac{1}{2}(\p\al,\p\al)_p+m}^\al
$$
and
$$G_{\frac{1}{2}(p\al,p\al)_p,\frac{1}{2}(\p\al,\p\al)_p}^\al=\C e_\al.
$$
Let $a^* \in \C[H]^\vee=\bigoplus_{\al \in H} (\C e_\al)^*$
and $\langle a^*,-\rangle$ be the linear map $\Om \rightarrow \C$ defined by
the composition of the projection
$G = \C[H] \oplus \bigoplus_{\substack{n,m\in \Z_{\geq 0}\\(n,m)\neq (0,0)}} 
G_{\frac{1}{2}(p\al,p\al)_p+n,\frac{1}{2}(\p\al,\p\al)_p+m}^\al \rightarrow \C[H]
$ and $a^*: \C[H]\rightarrow \C$.
Then, it is easy to verify $\langle a^*,-\rangle$ is a highest weight vector,
that is, $\langle a^*,h(-n)-\rangle=0$ for any $n \geq 1$ and $h \in H$.
Thus, for any $\al \in H$,
we have:
\begin{align*}
E^+(\al,\uz)\va&=\va,\\
\langle a^*, E^-(\al,\uz) - \rangle &= \langle a^*, - \rangle.
\end{align*}

Thus, by using the above fact and Lemma \ref{lattice_commutator},
for $\al_i \in H$ ($i=1,2,3$) and $a^* \in \C[H]^\vee$, we have
\begin{align*}
\langle a^*, Y(e_{\al_1},\uz_1)Y(e_{\al_2},\uz_2)e_{\al_3}\rangle
&=z_1^{(p\al_1,p\al_3)_p}\z_1^{(\p\al_1,\p\al_3)_p}
z_2^{(p\al_2,p\al_3)_p}\z_2^{(\p\al_2,\p\al_3)_p}\\
&(z_1-z_2)^{(p\al_1,p\al_2)_p}(\z_1-\z_2)^{(\p\al_1,\p\al_2)_p}|_{|z_1|>|z_2|} \langle a^*, e_{\al_1}e_{\al_2}e_{\al_3}\rangle.
\end{align*}
Since 
\begin{align*}
(z_i-z_j)^{(p\al_i,p\al_j)_p}(\z_i-\z_j)^{(\p\al_i,\p\al_j)_p}
&=|(\z_i-\z_j)|^{(\p\al_i,\p\al_j)_p}(z_i-z_j)^{(p\al_i,p\al_j)_p-(\p\al_i,\p\al_j)_p}\\
&=|(\z_i-\z_j)|^{(\p\al_i,\p\al_j)_p}(z_i-z_j)^{(\al_i,\al_j)_\lat},
\end{align*}
the formal power series
\begin{align*}
z_1^{-(\al_1,\al_3)_\lat}z_2^{-(\al_2,\al_3)_\lat}
(z_1-z_2)^{-(\al_1,\al_2)_\lat}|_{|z_1|>|z_2|}
\langle a^*, Y(e_{\al_1},\uz_1)Y(e_{\al_2},\uz_2)e_{\al_3}\rangle
\end{align*}
is a single-valued real analytic function in $\GCor_2$.
Then, similarly to the proof of Proposition 5.1 in \cite{M2} 
with Proposition \ref{general_locality},
we have:
\begin{prop}\label{standard}
For $p \in P(H)$,
$(G_{H,p},\hY,\va,H, -(-,-)_\lat, \omega_{H_l},\omb_{H_r})$ is a generalized full conformal vertex algebra.
\end{prop}
We remark that  in the above proposition the minus sign $-(-,-)_\lat$ appears in our convention.

We end this section by studying generalized full vertex algebra homomorphisms among $G_{H,p}$.
Let $(H,(-,-))$ and $(H',(-,-)')$ be real finite-dimensional vector spaces with non-degenerate symmetric bilinear forms
and $p \in P(H)$ and $\si: H \rightarrow H'$ be an 
isometric isomorphism.
Set $\si \cdot p=\si \circ p \circ \si^{-1} \in P(H')$.
Then,
\begin{align}
(\si h_1,\si h_2)_{\si \cdot p}'&=((\si \cdot p) \si h_1,(\si \cdot p) \si h_2)'-((\si \cdot \p) \si h_1,(\si \cdot \p) \si h_2)'\nonumber \\
&=(\si p h_1, \si p h_2)'-(\si \p h_1, \si \p h_2)' \nonumber \\
&=(p h_1,p h_2)-(\p h_1,\p h_2) \label{eq_p} \\
&=(h_1,h_2)_p.\nonumber
\end{align}
Thus, $\si$ induces an isometry from $(H,(-,-)_p)$ to $(H',(-,-)_{\si\cdot p}')$
and an isomorphism of Lie algebras $\si_{Lie}: \hat{H}^p \rightarrow \hat{H'}^{\si\cdot p}$,
where $\hat{H}^p$ (resp. $\hat{H'}^{{\si\cdot p}}$) is the Heisenberg Lie algebra associated with
$(H,(-,-)_p)$ (resp. $(H',(-,-)_{\si\cdot p}')$).
Let $\si_{alg}:\C[H] \rightarrow \C[H']$ be the linear map defined by
$e_\al \rightarrow e_{\si(\al)}$ for $\al \in H$.
Then, $\si_{alg}:\C[H] \rightarrow \C[H']$ is a $\hat{H}_{\geq 0}^p$-module homomorphism, where $\C[H']$ is regarded as a ${\hat{H}}_{\geq 0}^p$-module by $\si_{Lie}$.
Thus, we have an $\hat{H}^p$-module homomorphism $\tilde{\si}: G_{H,p} \rightarrow G_{H',\si\cdot p}$.
Since $\si_{alg}:C[H] \rightarrow \C[H']$ is a $\C$-algebra homomorphism,
it is easy to show that $(\tilde{\si},\si)$ is a generalized full vertex algebra homomorphism.
Thus, we have:
\begin{lem}\label{translation_p}
For $p \in P(H)$ and an isometry $\si: H\rightarrow H'$,
$(\tilde{\si},\si):G_{H,p} \rightarrow G_{H',\si\cdot p}$ is
an isomorphism of generalized full vertex algebras.
\end{lem}

\subsection{Tensor product}
Similarly to full vertex algebras,
the spectrum of a generalized full vertex algebra $\Om$ is said to be discrete
if for any $\al \in M_\Om$ and $H \in \R$,
$\sum_{t+\td<H}\dim \Om_{t,\td}^\al$ is finite. (The bounded below condition is already included in the definition).

Let $(\Omega_1,\hY_1,\va_1,H_1,(-,-)_1)$ and 
$(\Omega_2,\hY_2,\va_2,H_2,(-,-)_2,)$  be generalized full vertex algebras
and assume that
the spectrum of $\Om_1$ is discrete
and the spectrum of $\Om_2$ is bounded below.
Set $H=H_1 \oplus H_2$ and
 $\Omega_{t,\td}^{\al_1,\al_2}=\bigoplus_{s,\s \in \R}(\Omega_1)_{s,\s}^{\al_1} \otimes 
(\Omega_2)_{t-s,\td-\s}^{\al_2}$ for $(\al_1,\al_2) \in M_{\Om_1}\oplus M_{\Om_2} \subset H_1\oplus H_2$
and $\Omega= \bigoplus_{\al \in H, t,\td\in \R} \Omega_{t,\td}^\al$ and $\va=\va_1 \otimes \va_2$.

Define the linear map $\hY: \Om \rightarrow \Om[[z^\R,\z^\R]]$ by 
$\hY(a\otimes b,\uz)=\hY_1(a,\uz)\otimes \hY_2(b,\uz)$ for $a\in \Om_1$ and $b \in \Om_2$.
Then, for $a,c \in \Om_1$ and $b,d \in \Om_2$,
$$\hY(a\otimes b, \uz)c\otimes d
=\sum_{s,\s,r,\bar{r} \in \R}a(s,\s)c\otimes b(r,\bar{r})dz^{-s-r-2}\z^{-\s-\bar{r}-2}.$$
By (GFV1),
the coefficient of $z^k\z^{\bar{k}}$ is a finite sum for any $k,\bar{k}\in \R$.
Thus, $\hY$ is well-defined.
Since the spectrum of $\Om_2$ is bounded below,
there exists $N \in \R$ such that
$\Om_{t_0,\td_0}^{(\al_1,\al_2)}=
\bigoplus_{t,\td \leq N} (\Om_1)_{t,\td}^{\al_1} \otimes (\Om_2)_{t_0-t,\td_0-\td}^{\al_2}$.
Since the spectrum of $\Om_1$ is discrete, the sum is finite.
Thus, $\bigl( \Om_{t_0,\td_0}^{(\al_1,\al_2)} \bigr)^*=\bigoplus_{t,\td \in \R}
 \bigl((\Om_1)_{t,\td}^{\al_1}\bigr)^* \otimes \bigl( (\Om_2)_{t_0-t,\td_0-\td}^{\al_2}\bigr)^*$.
Define the bilinear form on $H$ by
$((\al_1,\al_2),(\be_1,\be_2))=(\al_1,\be_1)_1+(\al_2,\be_2)_2$
for $\al_i,\be_i \in H_i$ ($i=1,2$).
Then, we have:
\begin{prop}\label{tensor_Om}
$(\Omega,\hY,\va,H,(-,-))$ defined above is a generalized full vertex algebra.
Furthermore, if both $\Om_1$ and $\Om_2$ have energy-momentum tensors,
then $\Om$ is a generalized full conformal vertex algebra.
\end{prop}
The subalgebra of $\Om_1\otimes \Om_2$ associated with a subgroup $A \subset H_1\oplus H_2$
is denoted by $\Om_1\otimes_A \Om_2$ (see Lemma \ref{subgroup}).

\subsection{Cancellation of monodromy}
Let $(\Om,\hY,\va,H)$ be a generalized full vertex algebra and $p \in P(H)$.
The following lemma follows from the construction:
\begin{lem}
The spectrum of the generalized full vertex algebra $G_{H,p}$
constructed in Proposition \ref{standard} is discrete and bounded below.
\end{lem}
Assume that the spectrum of $\Om$ is bounded below.
We consider the tensor product of generalized full vertex algebras $\Om$ and $G_{H,p}$.
Set $\Delta H=\{(\al,\al)\in H \oplus H\}_{\al \in H}$,
which is a subgroup of $H\oplus H$.
Then, by Lemma \ref{subgroup}, 
$(G_{H,p}\otimes_{\Delta H} \Om,H\oplus H)$ is a generalized full vertex algebra.
We denote it by $F_{\Om,H,p}$.
Since the inner product of 
 $(\al,\al),(\be,\be)\in \Delta H \subset H\oplus H$
is $((\al,\al),(\be,\be))=(\al,\be)-(\al,\be)=0$ by the minus sign in Proposition \ref{standard},
$F_{\Om,H,p}$ is a full vertex algebra by Lemma \ref{even_omega}.
Thus, we have:
\begin{thm}
\label{construction}
For a generalized full vertex algebra $(\Om,\hY,\va,H)$ and $p \in P(H)$,
$F_{\Om,H,p}$ is a full vertex algebra.
Furthermore, if $\Om$ has an energy-momentum tensor,
then $F_{\Om,H,p}$ is a full conformal vertex algebra.
\end{thm}


\section{Categorical aspects}
In this section, we introduce the notion of a full $\mathcal{H}$-vertex algebra
and show that the vacuum space of a full $\mathcal{H}$-vertex algebra is a generalized full vertex algebra.

\subsection{Full $\mathcal{H}$-vertex algebras to generalized full vertex algebra}\label{sec_FHp}
Let $H_l$ and $H_r$ be real finite-dimensional vector spaces equipped with
non-degenerate symmetric bilinear forms $(-,-)_l:H_l \times H_l \rightarrow \R$ and $(-,-)_l:H_r \times H_r \rightarrow \R$.
Let $M_{H_l}(0)$ and $M_{H_r}(0)$ be affine Heisenberg vertex algebras
associated with $(H_l,(-,-)_l)$ and $(H_r,(-,-)_r)$.
Set $H=H_l \oplus H_r$ and let $p, \overline{p}:H \rightarrow H$
 be the projections onto $H_l$ and $H_r$
and 
$$
M_{H,p}=M_{H_l}(0)\otimes \overline{M_{H_r}(0)},
$$
the tensor product of the vertex algebra $M_{H_l}(0)$
and the conjugate vertex algebra $\overline{M_{H_r}(0)}$
(see Proposition \ref{conjugate} and Corollary \ref{chiral_tensor}).

In this section, we consider a class of a full vertex algebra
which is an $M_{H,p}$-module (like an algebra over a ring).
More precisely, let $F$ be a full vertex algebra
and assume that $M_{H,p}$ is a subalgebra of $F$, $M_{H,p} \subset F.$
Then, since $H_l \subset (M_{H,p})_{1,0}$ and $H_r \subset (M_{H,p})_{0,1}$,
$H_l \subset F_{1,0}$ and $H_r \subset F_{0,1}$.

We note that the subspaces $H_l$ and $H_r$ satisfy the following conditions:
For any $h_l,h_l' \in H_l$ and $h_r,h_r' \in H_r$,
\begin{enumerate}
\item[H1)]
$H_l \subset F_{1,0}$ and $H_r \subset F_{0,1}$;
\item[H2)]
 $\D H_l=0$ and $D H_r=0$;
\item[H3)]
$h_l(1,-1)h_l'=(h_l,h_l')_l \va$, $h_r(-1,1)h_r'=(h_r,h_r')_r \va$;
\item[H4)]
$h_l(n,-1)h_l'=0$, $h_r(-1,n)h_r'=0$
 for any $n=0$ or $n \in \Z_{\geq 2}$.
\end{enumerate}

In fact, these conditions characterize the existence of a subalgebra $M_{H,p} \subset F$:
\begin{prop}
If subspaces $H_l$ and $H_r$ of a full vertex algebra $F$ satisfy (H1) -- (H4),
then $H_l$ and $H_r$ generate a subalgebra which is isomorphic to $M_{H,p}$
as a full vertex algebra.
\end{prop}
\begin{proof}
By Proposition \ref{vertex_algebra},
the full vertex algebra generated by
$H_l \subset \ker \D$ (resp. $H_r \subset \ker D$)
is isomorphic to $M_{H_l}(0)$ (resp. $M_{H_r}(0)$).
By Proposition \ref{ker_hom},
the full vertex algebra generated by $H_l$ and $H_r$
is in the image of $M_{H,p} \subset \ker \D \otimes \ker D$.
Since $M_{H,p}$ is simple,
the assertion follows.
\end{proof}

Since $h_l \in H_l$ is a chiral vector, by Lemma \ref{hol_commutator},
$Y(h_l,\uz)=\sum_{n\in \Z} h_l(n,-1)z^{-n-1}$.
Hereafter, we will use a shorthand notation for $h_l \in H_l$, $h_r \in H_r$
and $n\in \Z$,
$h_l(n)=h(n,-1)$ and $h_r(n)=h_r(-1,n)$.
Set $h_l(z)=Y(h_l,\uz)=\sum_{n \in \Z}h_l(n)z^{-n-1}$ and $h_r(\z)= Y(h_r,\uz)=\sum_{n \in \Z} h_r(n)\z^{-n-1}$.
By Lemma \ref{hol_commutator} and Lemma \ref{hol_commute},
\begin{align}
[h_l(n), h_l'(m)]&=(h_l,h_l')_l n\delta_{n+m,0} \nonumber \\
[h_r(n), h_r'(m)]&=(h_r,h_r')_r n\delta_{n+m,0} \nonumber \\
[h_l(n), h_r(m)]&=0 \nonumber,
\end{align}
for any $n,m \in \Z$ and $h_l,h_l' \in H$ and $h_r,h_r' \in H_r$.

For $\al \in H$ and $h,\h \in \R$, we let $\Omega_{F,H}^\al$ be the set of all vectors $v \in F$ satisfying the following conditions:
\begin{enumerate}
\item
$h_l(n)v=0$ and $h_r(n)v$ for any $h_l \in H_l$ and $h_r \in H_r$ and $n \geq 1$.
\item
$h_l(0)v=(h_l,p\al)_l v$ and $h_r(0)v=(h_r,\p\al)_r v$ for any $h_l \in H_l$ and $h_r \in H_r$.
\end{enumerate}
Set $$\Omega_{F,H}=\bigoplus_{\al \in H} \Omega_{F,H}^\al$$
and
$$(\Om_{F,H})_{t,\td}^{\al}
= F_{t+\frac{(p\al,p\al)_l}{2},\td+\frac{(\p\al,\p\al)_r}{2}}\cap \Om_{F,H}^\al
$$ for $\al \in H$ and $t,\td\in \R$.
This subspace $\Omega_{F,H} \subset F$ is called the {\it vacuum space} in the chiral setting in \cite[Section 1.7]{FLM}.

{\it A full $\mathcal{H}$-vertex algebra}, denoted by $(F,H,p)$, is a full vertex algebra $F$ with a subalgebra $M_{H,p}$
such that 
\begin{enumerate}
\item[FH1)]
$h_l(0)$ and $h_r(0)$ are semisimple on $F$ with real eigenvalues for any $h_l\in H_l$ and $h_r \in H_r$;
\item[FH2)]
For any $\al \in H$, there exists $N \in \R$ such that
$F_{t,\td}^\al=0$ for $t\leq N$ or $\td \leq N$.
\end{enumerate}

Let $(F,H,p)$ be a full $\mathcal{H}$-vertex algebra.
By (FH1) and (FH2) and the representation theory of an affine Heisenberg Lie algebra (\cite[Theorem 1.7.3]{FLM}),
$F$ is isomorphic to $\bigoplus_{\al \in H} M_{H,p} \otimes \Omega_{F,H}^\al$ as an $M_{H,p}$-module.
In particular,
$F$ is generated by the subspace $\Omega_{F,H}$
as a module of the Heisenberg Lie algebra $\hat{H}$.

For $\al \in H$,
define $z^{(p\al)(0)}\z^{(\p\al)(0)} \in \End \Om_{F,H}[z^\R,\z^\R]$ by
$$z^{(p\al)(0)}\z^{(\p\al)(0)}v=z^{(p\al,p\be)_l}\z^{(\p\al,\p\be)_r}v$$ for $v \in \Omega_{F,H}^\be$.
For $\al \in H$, 
set \begin{align*}
E^-(\al,\uz)&= \exp\Bigl(\sum_{n \geq 1} \frac{p\al(-n)}{n}z^n+\frac{\p\al(-n)}{n}\z^n\Bigr)\\
E^+(\al,\uz)&= \exp\Bigl(\sum_{n \geq 1} \frac{p\al(n)}{-n} z^{-n}+\frac{\p\al(n)}{-n}\z^{-n} \Bigr).
\end{align*}
Then, for any $h_l \in H_l$ and $n>0$,
\begin{align}
[h_l(n),E^-(\al,\uz)]&=(h_l,\al)_l z^nE^-(\al,\uz) \nonumber \\
[h_l(-n),E^+(\al,\uz)]&=(h_l,\al)_l z^{-n}E^-(\al,\uz) \nonumber \\
[h_l(n),E^+(\al,\uz)]&=0 \nonumber \\
[h_l(-n),E^-(\al,\uz)]&=0 \nonumber \\
[h_l(0),E^\pm(\al,\uz)]&=0 \nonumber
\end{align}
hold (Similar results hold for $h_r \in H_r$.).

Let $v \in \Omega_{F,H}^\al$.
Set $$\hY(v,\uz)= 
E^-(-\al,\uz)Y(v,\uz)E^+(-\al,\uz)z^{(-p\al)(0)}\z^{(-\p\al)(0)}.$$

By Lemma \ref{hol_commutator},
\begin{align*}
[h_l(n), Y(v,\uz)]&=(h_l,\al)_l z^nY(v,\uz)\\
[h_r(n),Y(v,\uz)]&=(h_r,\al)_r \z^nY(v,\uz)
\end{align*}
 for any $h_l \in H_l$ and $h_r \in H_r$ and $n \in \Z$.
Hence, we have $[h(n), \hY(v,\uz)]=0$ and $[\h(n), \hY(v,\uz)]=0$ for any $0 \neq n \in \Z$ and $v \in \Om_{H,H}$, $h_l \in H_l$, $h_r \in H_r$.
Thus, $\hY(v,\uz)$ preserves $\Omega_{F,H}$,
that is $\hY(v,\uz) \in \End \Omega_{F,H}[[z^\R,\z^\R]]$,
which defines a product on $\Omega_{F,H}$.

Set \begin{align*}
\omega_{H_l}&=\frac{1}{2} \sum_{i}h_l^i(-1,-1)h_l^i \in F_{2,0} \\
\om_{H_r}&=\frac{1}{2} \sum_{i}h_r^i(-1,-1)h_r^i \in F_{0,2},
\end{align*}
where $\{h_l^i\}_i$ is an orthonormal basis of $H_l\otimes_\R \C$
and $\{h_r^i\}_i$ is an orthonormal basis of $H_r\otimes_\R \C$,
and
$$D_\Om=D-\omega_{H_l}(0,-1), \D_\Om=\D-\omega_{H_r}(-1,0)
$$
and
$$L_\Om(0)=L_F(0)-\om_{H_l}(1,-1),
\Ld_\Om(0)=\Ld_F(0)-\om_{H_r}(-1,1),
$$
where $L_F(0), \Ld_F(0) \in \End F$ are defined by
$L_F(0)|_{F_{t,\td}}=t$ and $\Ld_F(0)|_{F_{t,\td}}=\td$
for $t,\td\in \R$.
Then, we have:
\begin{lem}\label{conformal_FHp}
For any $v\in \Om^\al \cap F_{t,\td}$,
\begin{align*}
[D_\Om,\hY(v,\uz)]&=\dz \hY(v,\uz), \\ 
[\D_\Om,\hY(v,\uz)]&=\ddz \hY(v,\uz),\\
[L_\Om(0),\hY(v,\uz)]&=(z \dz+t-\frac{(p\al,p\al)_l}{2}) \hY(v,\uz),\\
[\Ld_\Om(0),\hY(v,\uz)]&=(\z \ddz+\td-\frac{(\p\al,\p\al)_r}{2}) \hY(v,\uz).
\end{align*}
\end{lem}
\begin{proof}
It is easy to show that
$D_\Om, \D_\Om, L_\Om(0), L_\Om(0)$ commute with the action of
the Heisenberg Lie algebra $\hat{H}$.
Since
$[\om_{H_l}(0),Y(v,\uz)]=Y(\om_{H_l}(0)v,\uz)$
and $\om_{H_l}(0)=\sum_{i} \sum_{k \geq 0}h_i(-k-1)h_i(k)$,
we have $[\om_{H_l}(0),Y(v,\uz)]=Y((p\al)(-1,-1)v,\uz)$.
Since, by Lemma \ref{hol_commutator}, $$Y((p\al)(-1,-1)v,\uz)
=(p\al)^+(z)Y(v,\uz)+Y(v,\uz)(p\al)^-(z),$$
we have
\begin{align*}
[D_\Om,\hY(v,\uz)]
&=E^-(-\al,\uz)[D_\Om,Y(v,\uz)] E^+(-\al,\uz)z^{(-p\al)(0)}\z^{(-\p\al)(0)}\\
&=E^-(-\al,\uz)(\dz Y(v,\uz)-Y((p\al)(-1,-1)v,\uz)E^+(-\al,\uz)z^{(-p\al)(0)}\z^{(-\p\al)(0)}\\
&=\dz \hY(v,\uz)
\end{align*}
Since $\om_{H_l}(1,-1)=\sum_{i}\Bigl( 1/2 h_i(0)h_i(0) +\sum_{k \geq 1}h_i(-k)h_i(k)\Bigr)$,
we have
\begin{align*}
[\om_{H_l}(1,-1),Y(v,\uz)]&=Y(\om_{H_l}(0,-1)v,\uz)z+Y(\om_{H_l}(1,-1)v,\uz)\\
&=zY((p\al)(-1,-1)v,\uz)+\frac{(p\al,p\al)_l}{2} Y(v,\uz).
\end{align*}
Thus, similarly to the above,
\begin{align*}
[L_\Om(0),\hY(v,\uz)]
&=E^-(-\al,\uz)[L_\Om(0),Y(v,\uz)] E^+(-\al,\uz)z^{(-p\al)(0)}\z^{(-\p\al)(0)}\\
&=E^-(-\al,\uz) \Bigl(
\bigl(z \dz+h-\frac{(p\al,p\al)_l}{2}\bigr)Y(v,\uz)
-zY((p\al)(-1,-1)v,\uz)\Bigr)\\
&E^+(-\al,\uz)z^{(-p\al)(0)}\z^{(-\p\al)(0)}\\
&=(z\dz+(h-\frac{(p\al,p\al)_l}{2})) \hY(v,\uz).
\end{align*}
\end{proof}

Define a new bilinear form $(-,-)_\lat$ on $H$ by 
$$(\al,\be)_\lat=(p\al,p\be)_l - (\p\al,\p\be)_r$$
for $\al,\be \in H$.
The main result of this section is the following theorem:
\begin{thm}\label{vacuum_space}
For a full $\mathcal{H}$-vertex algebra $(F,H,p)$,
$(\Om_{F,H},\hY,\va,H,(-,-)_\lat)$ is a generalized full vertex algebra.
\end{thm}
\begin{proof}
We will show the assertion by using Proposition \ref{general_locality}.
(GFV2)-(GFV4) and (GFV7) are obvious.
For $\al,\be \in M_\Om$ and $a\in \Om^\al$ and $b \in \Om^\be$,
$\hY(a,\uz)b
=E^-(-\al,\uz)Y(a,\uz)z^{-(p\al,p\be)_l}\z^{-(\p\al,\p\be)_r}.$
Since $z^{-(p\al,p\be)_l}\z^{-(\p\al,\p\be)_r}
=z^{-(\al,\be)_\lat} |z|^{-(\p\al,\p\be)_r}$,
(GFV1) holds.
By Lemma \ref{conformal_FHp}, (GFV6) holds.
(GFL1) and  (GFL2) follow from Lemma \ref{conformal_FHp}.
It suffices to show that
(GFL3).
Let $a_i \in \Omega^{\al_i}$ for $i=1,2,3$
and $u \in \Om^\vee$.
We remark that $M_{H,p}$ is graded by $\om_{H_l}(1,-1)$ and $\om_{H_r}(-1,1)$.
Then, $(M_{H,p})_{0,0}=\C \va$ and $M_{H,p}=\bigoplus_{n,m \geq 0} 
(M_{H,p})_{n,m}$.
Set $M_{H,p}^+=\bigoplus_{(n,m)\neq (0,0)}(M_{H,p})_{n,m}$.
Denote by $\pi$ the projection of $F = M_{H,p}\otimes \Om=
\C\va\otimes \Om \oplus M_{H,p}^+ \otimes \Om$
to $\C\va \otimes \Om$.
Set $u'=u\circ \pi \in F^\vee$.
By the construction,
$u' (h(-n)-)=u'(\h(-n)-)=0$ for any $n\in \Z_{\geq 1}$.
Since
$$Y(a_i,\uz_i)=E^-(\al_i,\uz_i)\hY(a_i,\uz_i)  E^+(\al_i,\uz_i)
z_i^{(-p\al_i)(0)}\z_i^{(-\p\al_i)(0)}$$ for $i=1,2$,
we have
\begin{align*}
u'(&Y(a_1,\uz_1)Y(a_2,\uz_2) a_3)\\
&=
u'(\hY(a_1,\uz_1)E^+(\al_1,\uz_1)z_1^{(p\al_1)(0)}\z_1^{(\p\al_1)(0)}
E^-(\al_2,\uz_2)\hY(a_2,\uz_2)z_2^{(p\al_2)(0)}\z_2^{(\p\al_2)(0)}a_3)\\
&=z_1^{(p\al_1,p\al_2+p\al_3)_l}z_2^{(p\al_2,p\al_3)_l}
\z_1^{(\p\al_1,\p\al_2+\p\al_3)_r}\z_2^{(\p\al_2,\p\al_3)_r}
u(\hY(a_1,\uz_1)E^+(\al_1,\uz_1)E^-(\al_2,\uz_2)\hY(a_2,\uz_2)a_3).
\end{align*}
By Lemma \ref{lattice_commutator}
$$E^+(\al_1,\uz_1)E^-(\al_2,\uz_2)
=(1-z_2/z_1)^{(p\al_1,p\al_2)_l}(1-\z_2/\z_1)^{(\p\al_1,\p\al_2)_r}
E^-(\al_2,\uz_2)E^+(\al_1,\uz_1).
$$
Since $\{h(n),\h(n)\}_{n\neq 0, h\in H_l,\h\in H_r}$ commute with
$\hY(a_i,\uz_i)$,
we have
\begin{align*}
u'(&Y(a_1,\uz_1)Y(a_2,\uz_2) a_3)\\
&z_1^{(p\al_1,p\al_3)_l}z_2^{(p\al_2,p\al_3)_l}
\z_1^{(\p\al_1,\p\al_3)_r}\z_2^{(\p\al_2,\p\al_3)_r}
(z_1-z_2)^{(p\al_1,p\al_2)_l}(\z_1-\z_2)^{(\p\al_1,\p\al_2)_r}
u(\hY(a_1,\uz_1)\hY(a_2,\uz_2)a_3).
\end{align*}
Since
\begin{align*}
z_1^{(p\al_1,p\al_3)_l}z_2^{(p\al_2,p\al_3)_l}
\z_1^{(\p\al_1,\p\al_3)_r}\z_2^{(\p\al_2,\p\al_3)_r}
(z_1-z_2)^{(p\al_1,p\al_2)_l}(\z_1-\z_2)^{(\p\al_1,\p\al_2)_r}\\
=z_1^{(\al_1,\al_3)_\lat}z_2^{(\al_2,\al_3)_\lat}
(z_1-z_2)^{(\al_1,\al_2)_\lat}
|z_1|^{(\p\al_1,\p\al_3)_r}|z_2|^{(\p\al_2,\p\al_3)_r}
|(z_1-z_2)|^{(\p\al_1,\p\al_2)_r}
\end{align*}
and
$|z_1|^{(\p\al_1,\p\al_3)_r}|z_2|^{(\p\al_2,\p\al_3)_r}
|(z_1-z_2)|^{(\p\al_1,\p\al_2)_r} \in \GCor_2$,
(GFL3) follows from (FV5).
\end{proof}
\begin{rem}
\label{rem_DL_chiral}
Let $L_{\mathfrak{g},k}$ be the simple affine vertex operator algebra at level $k\in \Z_{\geq 0}$
and $H_{\mathfrak{g}} \subset \mathfrak{g}=(L_{\mathfrak{g},k})_1$ a Cartan subalgebra of the weight one subspace.
Then, $(L_{\mathfrak{g},k},H_{\mathfrak{g}})$ is a chiral $\mathcal{H}$-vertex algebra (see for the chiral setting Section \ref{sec_Remark_vertex}).
Dong and Lepowsky show that the vacuum space $\Om_{L_{\mathfrak{g},k},H_{\mathfrak{g}}}$ inherits a generalized vertex algebra structure \cite[Theorem 14.16]{DL}.
Theorem \ref{vacuum_space} generalizes this result to general $\mathcal{H}$-vertex algebras and full $\mathcal{H}$-vertex algebras (see also Proposition \ref{vacuum_vertex}).
\end{rem}

A full $\mathcal{H}$-conformal vertex algebra
is a pair of a full $\mathcal{H}$-vertex algebra and an energy-momentum tensor $(\om,\omb)$
such that $\om(n+2,-1)H_l=0$ and $\omb(-1,n+2)H_r=0$ for $n \in \Z_{\geq 0}$.
By Lemma \ref{conformal_FHp} and Theorem \ref{vacuum_space}, we have:
\begin{cor}\label{vac_omega}
Let $(F,H,p,\om,\omb)$ be a full $\mathcal{H}$-conformal vertex algebra.
Then, $(\om-\om_{H_l},\omb-\om_{H_r})$ is an energy-momentum tensor of the generalized full vertex algebra
$\Om_{F,H}$.
\end{cor}

\subsection{Equivalence between categories}\label{sec_equiv}
In this section, we show that
Theorem \ref{construction} and
Theorem \ref{vacuum_space} give an equivalence 
between a category of 
full $\mathcal{H}$-vertex algebras
and a category of
generalized full vertex algebras with an additional structure $p$.

We first define these categories.
A morphism from a full $\mathcal{H}$-vertex algebra
$(F_1,H_1,p_1)$ to a full $\mathcal{H}$-vertex algebra $(F_2,H_2,p_2)$
is a full vertex algebra homomorphism $\phi :F_1\rightarrow F_2$
such that $\phi(H_1) = H_2$.
We denote the category of full $\mathcal{H}$-vertex algebras
by $\FHp$.

Let $\OHp$ denote the following category:
The objects are pairs of a generalized full vertex algebra $(\Om,H)$ and $p\in P(H)$.
A morphism from $(\Om_1,H_1,p_1)$ to $(\Om_2,H_2,p_2)$
is a generalized full vertex algebra homomorphism $(\psi,\psi'):(\Om_1,H_1) \rightarrow
(\Om_2,H_2)$ satisfying $\psi' \circ p_1 = p_2 \circ \psi'$.
We call $p\in P(H)$ {\it a charge structure} of a generalized full vertex algebra.

Let $(F,H,p)$ be a full $\mathcal{H}$-vertex algebra.
Then, $p \in P(H)$. Thus, $(\Om_{F,H},H,p)$
is an object in $\OHp$.
\begin{lem}
The assignment $\Om:\FHp \rightarrow \OHp,\; (F,H,p)\mapsto (\Om_{F,H},H,p)$ is a functor.
\end{lem}
\begin{proof}
Let $\phi$ be a morphism from
a full $\mathcal{H}$-vertex algebra $(F_1,H_1,p_1)$ to
a full $\mathcal{H}$-vertex algebra $(F_2,H_2,p_2)$.
Since $\phi$ preserves the vacuum vector,
$\phi(\ker p_1)=\phi (H_1 \cap \ker D)=H_2 \cap \ker D = \ker p_2$.
Since $\phi(h_l(1,-1)h_l')=\phi(h_l)(1,-1)\phi(h_l')$
for any $h_l,h_l'\in (H_1)_l$,
$\phi$ is an isometric isomorphism between $H_1$ and $H_2$
and $\phi \circ p_1 = p_2 \circ \phi$.
Since the restriction of $\phi$ on the vacuum spaces gives a linear map
$\phi|_{\Om_{F_1,H_1}}: \Om_{F_1,H_1}\rightarrow \Om_{F_2,H_2}$,
the pair $(\phi|_{\Om_{F_1,H_1}},\phi|_{H_1})$ is a morphism 
of $\OHp$.
\end{proof}

Let $(\Om,H,p)$ be an object in $\OHp$.
Then, $F_{\Om,H,p}$ is a full vertex algebra.
Since $M_{H,p} = G_{H,p}^0 \otimes \C\va \subset 
G_{H,p}^0 \otimes\Om^0 \subset F_{\Om,H,p}$,
$F_{\Om,H,p}$ is naturally a full $\mathcal{H}$-vertex algebra.
\begin{lem}
The assignment $F:\OHp \rightarrow \FHp,\; (\Om,H,p)\mapsto (F_{\Om,H,p},H,p)$ is a functor.
\end{lem}
\begin{proof}
Let $(\Om_1,H_1,p_1)$ and $(\Om_2,H_2,p_2)$ be objects in $\OHp$
and $(\psi,\psi')$ be a morphism from $(\Om_1,H_1,p_1)$ to $(\Om_2,H_2,p_2)$.
Since $\psi'$ is an isometric isomorphism,
by Lemma \ref{translation_p},
we have an isomorphism of generalized full vertex algebras
$$
\tilde{\psi'}:G_{H_1,p_1} \rightarrow G_{H_2,p_2},
$$
where we used $\psi' \circ p_1= p_2 \circ \psi'$.
Then, we have a generalized full vertex algebra homomorphism
$$
\tilde{\psi'}\otimes \psi: G_{H_1,p_1}\otimes \Om_1 \rightarrow
G_{H_2,p_2}\otimes \Om_2.$$
The restriction of the homomorphism on
$G_{H_1,p_1}\otimes_{\Delta H_1} \Om_1\subset G_{H_1,p_1}\otimes \Om_1$ gives us a full $\mathcal{H}$-vertex algebra homomorphism
as desired.
\end{proof}

It is clear that the above functors are mutually inverse
equivalences.
Thus, we have:
\begin{thm}
\label{equivalence}
$\Om: \FHp \rightarrow \OHp$ and $F:\OHp \rightarrow \FHp$ gives an equivalence of categories.
\end{thm}

\begin{cor}
\label{decomposition}
Let $(F,H,p)$ be a full $\mathcal{H}$-vertex algebra.
Then, $F$ is isomorphic to $F_{\Om_{F,H},H,p}=G_{H,p}\otimes_{\Delta H} \Om_{F,H}$
as a full $\mathcal{H}$-vertex algebra.
\end{cor}

\subsection{Adjoint functor I -- generalized full vertex algebra and associative algebra}
In this section, we construct an adjoint functor 
from the category of generalized full vertex algebras to
some category of associative algebras.

We first recall that for a vertex algebra $V$,
$\ker D_V=\{v\in V\;|\; v(-2)\va=0\}$ is a commutative $\C$-algebra
(see for example \cite[Section 3.11]{LL}).
Conversely,
any commutative $\C$-algebra $A$ is a vertex algebra,
where the vertex operator is defined by
$Y(a,z)b=ab$ (consisting of only the constant term).
In fact, this correspondence gives an adjoint functor
between the category of vertex algebras
and the category of commutative $\C$-algebras.


We generalize this result to a
generalized full vertex algebra based on the discussion in \cite{LL}.
Since a generalized full vertex algebra has the monodromies,
the corresponding $\C$-algebra is no longer commutative, which we call an {\it AH-pair}. We first recall the notion of AH pairs introduced in \cite{M1}, which is a commutative algebra object in some braided tensor category (see \cite{M2}, Section 5.3).

Let $H$ be a finite-dimensional vector space over $\mathbb{R}$ equipped with a non-degenerate symmetric bilinear form $(-,-)$ and $A$ a unital associative algebra over $\mathbb{C}$ with the unity $1$. 
Assume that $A$ is graded by $H$ as $A=\bigoplus_{\alpha\in H}A^\alpha$. 

We will say that such a pair $(A,H)$ is an {\it even AH pair} if the following conditions are satisfied:
{\leftmargini2.8em
\begin{enumerate}
\itemsep1ex\parskip0ex
\item[AH1)]
$1 \in A^0$ and $A^\alpha A^\beta\subset A^{\alpha+\beta}$ for any $\alpha,\beta\in H$;
\item[AH2)]
If $A^\alpha \neq 0$, then $(\alpha,\alpha)\in 2\mathbb{Z}$;
\item[AH3)]
For $v\in A^\alpha,w\in A^\beta$, $vw=(-1)^{(\alpha,\beta)}wv$;
\end{enumerate}}
\begin{rem}
Suppose that $A^\alpha A^\be \neq 0$ for $\al,\be\in H$.
Then, by (AH1) and (AH2), $(\al,\al), (\be,\be), (\al+\be,\al+\be)\in 2\Z$
and thus $(\al,\be) \in \Z$.
Hence, $(-1)^{(\al,\be)}$ is well-defined.
\end{rem}

Define an $\R^2\times H$-grading on $A$ by
$$
\begin{cases}
A_{t,\td}^\al=0 \text{ if }(t,\td) \neq (0,0),\cr
A_{0,0}^\al = A^\al
\end{cases}
$$
for any $\al \in H$
and set $\hY(a,\uz)=l_a \in \End A$ for $a\in A$,
where $l_a$ is the left multiplication by $a$
and $\va=1$.
\begin{prop}\label{AH_general}
For an even AH pair $(A,H)$,
$(A,\hY,\va,H)$ is a generalized full vertex algebra.
Furthermore, $(A,H)$ is a generalized full conformal vertex algebra
with the energy-momentum tensor $(0,0)$.
\end{prop}
\begin{proof}
Since $(\al,\al) \in 2\Z$ for any $\al \in M_{A,H}$,
(GFV2) holds. Let $a_i \in A^{\al_i}$ and $u\in A^\vee$.
Then,
$u(a_1(a_2a_3))=(-1)^{(\al_1,\al_2)}u(a_2(a_1a_3))=u((a_1a_2)a_3)$ by (AH3),
which implies that (GFV5) holds. The rest is obvious.
\end{proof}

For even AH pairs $(A,H_A)$ and $(B,H_B)$,
a homomorphism of even AH pairs is a pair $(f,f')$ of maps $f:A \longrightarrow B$ and $f':H_A\longrightarrow H_B$ such that $f$ is an algebra homomorphism and $f'$ an isometry such that $f(A^\alpha)\subset B^{f'(\alpha)}$ for all $\alpha\in H_A$.
We denote by $\AH$ the category of even AH pairs.
Then, Proposition \ref{AH_general} gives a functor from the category of even AH pairs to the category of
generalized full vertex algebras,
denoted by $i: \AH \rightarrow \GFA$.
In the rest of this section, we construct an adjoint functor following \cite{M1}.

Let $(\Om,H)$ be a generalized full vertex algebra.
Set $A_\Om=\ker D \cap \ker \D \cap \Om_{0,0}$.
By Proposition \ref{skew_symmetry},
$D$ and $\D$ act as derivations of the algebra.
Thus, $\ker D \cap \ker \D$ is a subalgebra of $\Om$.
If $a\in \ker D \cap \ker \D$, by Proposition \ref{skew_symmetry} again,
$\hY(a,\uz)=a(-1,-1) \in \End \Om$,
that is, the vertex operator is independent of the position.
By (GFV6),  $A_\Om$ is a subalgebra of $\ker D \cap \ker \D$ and $\Om$.
%
Set $A_\Om^\al=A_\Om \cap \Om_{0,0}^\al$ for $\al \in H$.
Define a product on $A_\Om$ by
$$a\cdot b=a(-1,-1)b,$$
for $a,b \in A_\Om$.
%
%
%
Then, we have:
\begin{prop}\label{omega_AH}
For a generalized full vertex algebra $(\Om,H)$,
$(A_\Om,H)$ is an even AH-pair.
\end{prop}
\begin{proof}
By (GFV3) and (GFV4), $\va$ is unity and (AH1) holds.
If $\Om_{0,0}^{\al}\neq 0$, then by (GFV2) $(\al,\al) \in 2\Z$, which implies (AH2).
Assume that $a\cdot b\neq 0$.
Since $a\cdot b=\hY(a,\uz)b \in z^{(\al,\be)}\Om((z,\z,|z|))$,
$(\al,\be) \in \Z$.
By (GFV5), $a(bc)=(-1)^{(\al,\be)} b(ac)=(ab)c$ for any $c \in A_\Om$.
Thus, $A_\Om$ is an even AH pair.
\end{proof}

This correspondence
$$A: \GFA \rightarrow \AH,\; (\Om,H) \mapsto (A_\Om,H)
$$
is a functor since a morphism of generalized full vertex algebras preserves the vacuum vector $\va$,
thus, commutes with $D,\D$.
\begin{prop}
The above functor $A: \GFA \rightarrow \AH$ is right adjoint to the inclusion functor
$i:\AH \rightarrow \GFA$.
\end{prop}
\begin{proof}
Let $(A,H)$ be an even AH pair
and $(\Om,H')$ a generalized full vertex algebra
and $(f,f'):(A,H) \rightarrow (\Om,H')$ a generalized full vertex algebra homomorphism.
Since $DA=\D A=0$ and $f(1)=\va$, thus, $f$ commutes with $D,\D$,
the image of $f$ is in $\ker D \cap \ker \D$.
Since $f$ preserves the $\R^2$-grading of the generalized full vertex algebras,
$f(A) \subset \ker D \cap \ker \D \cap \Om_{0,0}$.
Thus, the restriction gives a generalized full vertex algebra homomorphism $(f,f'): A \rightarrow A_\Om$, which is an even AH pair homomorphism.
Since the rest of the argument is completely similar
to the proof of \cite[Theorem 3.1]{M1},
the details are left to the reader.
\end{proof}

\subsection{Adjoint functor II -- Lattice full vertex algebra revisited}\label{sec_adjoint2}
A structure of AH pairs is studied in \cite{M1}. We briefly recall it.
Let $H$ be a finite-dimensional real vector space equipped with a non-degenerate bilinear form $(-,-): H\times H\rightarrow \R$.

A {\it good AH pair} is an even AH pair $(A,H)$ such that:
\begin{enumerate}
\item[GAH1)]
$A^0=\C \va$;
\item[GAH2)]
$ab \neq 0$ for any $\al,\be\in H$ and $a\in A^\al\setminus \{0\}$, $b\in A^\be \setminus \{0\}$.
\end{enumerate}
A {\it lattice pair} is a good AH pair such that
\begin{enumerate}
\item[LP)]
$A^{-\al} \neq 0$ if $A^{\al}\neq 0$ for $\al \in H$.
\end{enumerate}

For a good AH pair $(A,H)$,
set $M_{A,H}=\{\al \in H\;|\; A^\al \neq 0 \}$.
Then, by (GAH1) and (GAH2), $0 \in M_{A,H}$
and $\al+\be \in M_{A,H}$ for any $\al,\be\in M_{A,H}$.
Thus, $M_{A,H}$ is a submonoid of $H$.
A good AH pair $(A,H)$ is a lattice pair
if and only if $M_{A,H}$ is a subgroup of $H$.

We also introduce the notion of an even $H$-lattice (see section 2.2 in \cite{M1}).
An even $H$-lattice is a subgroup $L \subset H$
such that $(\al,\al)\in 2\Z$ for any $\al \in L$.
The subgroup $M_{A,H} \subset H$ for a lattice pair $(A,H)$
is an example of an even $H$-lattice by (AH4).

Conversely, Let $L\subset H$ be an even $H$-lattice
and $Z^2(L,\C^\times)$
the $\C^\times$-coefficient two-cocycles of 
the abelian group $L$.
It is not hard to show that there exists $\epsilon \in Z^2(L,\C^\times)$
such that $\epsilon(\al,0)=\epsilon(0,\al)=1$ and $\epsilon(\al,\be)\epsilon(\be,\al)=(-1)^{(\al,\be)}$ for any $\al,\be\in L$
(see \cite{FLM,M1}).
Then, define a new product on the group algebra $\C[L]=\bigoplus_{\al \in L} \C e_\al$ by
$$
e_\al e_\be =\epsilon(\al,\be)e_{\al+\be}.
$$
Since $\epsilon$ is a two-cocycle, the product is associative.
Denote by $\C[\hat{L}]$ the associative algebra.
By construction, $\C[\hat{L}]$ is a lattice pair,
which is a generalization of the twisted group algebra constructed in \cite{FLM}.
In fact, any lattice pair $(A,H)$ is isomorphic to $\C[\hat{M_{A,H}}]$.
More precisely, we have (see section 2.1 and 2.2 in \cite{M1}):
\begin{prop}\label{isom_lattice}
Let $(A,H)$ be a lattice pair and
$M_{A,H}$ be an even $H$-lattice associated with the lattice pair.
Then, $(A,H)$ is isomorphic to $(\C[\hat{M_{A,H}}],H)$ as even AH pairs.
\end{prop}

A category of good AH pairs (resp. a category of lattice pairs) is a full subcategory of $\AH$ whose objects are good AH pairs
(reps. lattice pairs), which is denoted by $\gAH$ (resp. $\LP$).
Let $i: \LP\rightarrow \gAH$ be the inclusion functor.
We will construct an adjoint functor $\lat: \gAH \rightarrow \LP$.
Let $(A,H)$ be a good AH pair
and set $L_{A,H} = \{\al \in M_{A,H}\;|\; -\al \in M_{A,H}  \} = M_{A,H} \cap (-M_{A,H})$.
Then, $L_{A,H}$ is a subgroup of $H$, thus, an even $H$-lattice.
Set $A^\lat =\bigoplus_{\al \in L_{A,H}} A^\al$.
Since $L_{A,H}$ is a subgroup, $A^\lat$ is a subalgebra of $A$ as an AH pair.
In fact, the correspondence $(A,H) \mapsto (A^\lat,H)$ define the functor
$\lat:\gAH\rightarrow \LP$.
Hence we have:
\begin{prop}[Proposition 2.5 in \cite{M1}]\label{lattice_functor}
The functor $\lat:\gAH\rightarrow \LP$ is right adjoint to the inclusion functor
$i: \LP\rightarrow \gAH$.
\end{prop}

As an application, we give examples of full vertex algebras constructed from even lattices.
An {\it integral lattice} of rank $n \in \mathbb{N}$ is a rank $n$ free abelian group $L$
equipped with a $\mathbb{Z}$-valued symmetric bilinear form
$$( \,\;,\; ):L \times L \rightarrow \mathbb{Z}.$$
An integral lattice $L$ is said to be {\it even} if
$$(\al,\al) \in 2\Z \fora \al \in L,$$
and {\it non-degenerate} if
the induced bilinear form on the real vector space $L \otimes_\Z \R$ is non-degenerate,
and {\it positive-definite} if
$$(\al,\al)>0 \fora \al \in L\setminus \{0\}.$$
Let $L$ be an even non-degenerate lattice.
Since $L$ is an even $L\otimes_\Z \R$-lattice,
a lattice pair $\C[\hat{L}]$ can be constructed as above.
Since $\C[\hat{L}]$ is an even AH pair, it is a generalized full conformal vertex algebra by Proposition \ref{AH_general}.
Thus, by Theorem \ref{construction}, for any $p \in P(L\otimes_\Z \R)$,
$F_{\C[\hat{L}],L\otimes_\Z \R,p}$ is a full conformal vertex algebra.
We denote it by $F_{L,p}$ and call {\it a lattice full vertex algebra}.
\begin{rem}
It is natural in physics to choose the projection $p \in P(L\otimes_\Z \R)$ such that $\ker p$ is a negative-definite
subspace in $L\otimes_\Z \R$.
If the signature of $L\otimes_\Z \R$ is $(n,m)$,
then such projections are parametrized by
the orthogonal Grassmannian
$$
\mathrm{O}(n,m)/ \mathrm{O}(n) \times \mathrm{O}(m),
$$
where $\mathrm{O}(n,m)$ is an orthogonal group with
the signature $(n,m)$.
It is noteworthy that, in this case, the spectrum of the lattice full vertex algebra is compact.
Thus, we constructed a continuous family of compact full vertex algebras.
We study these algebras in more detail in section \ref{sec_toroidal}.
\end{rem}
To summarize, we constructed two adjoint functors and one equivalence of categories in this section:

\begin{align*}
\xymatrix{
\LP \ar@<1.2ex>[r]^{i}
& \gAH \ar@<1.2ex>[l]^{-^\lat}_{\scriptscriptstyle\boldsymbol{\top}}
},\\
\xymatrix{
\AH \ar@<1.2ex>[r]^{\;\;\;i}
& \GFA \ar@<1.2ex>[l]^{\;\;\;\;\;\;A_{-}}_{\;\;\;\;\;\;\;\;\scriptscriptstyle\boldsymbol{\top}}
},\\
\xymatrix{
\OHp \ar@<1.2ex>[r]^{F_{-}}
& \FHp \ar@<1.2ex>[l]^{\; \Om}_{\scriptscriptstyle\boldsymbol{\cong}}
}.
\end{align*}

\subsection{Remark on vertex algebras}\label{sec_Remark_vertex}
In this section, we discuss the equivalence of categories
in the case that a full $\mathcal{H}$-vertex algebra consists of only chiral vectors.

Let $V$ be a $\Z$-graded vertex algebra
and $H$ a real subspace of $V_1$ such that:
\begin{enumerate}
\item[HS1)]
$h(1)h'\in \mathbb{\R}\mathbf{1}$ for any $h,h'\in H$;
\item[HS2)]
For any $h,h'\in H$,
$h(n)h'=0$ if $n=0$ or $n\geq 2$;
\item[HS3)]
The bilinear form $(-,-)$ on $H$ defined by $h(1)h'=(h,h')\mathbf{1}$ for $h,h'\in H$ is non-degenerate.
\end{enumerate}
Then, as in Section \ref{sec_FHp}, $H$ generates a representation of the Heisenberg Lie algebra.
Set $$
(\Om_{V,H})_{t}^\al=\{v \in V_{t+\frac{(\al,\al)}{2}}^\al\;|\; h(0)v=(\al,h)v, h(n)v=0 \text{ for any }h \in H \text{ and } n \geq 1 \}
$$
for $\al \in H$ and $t \in \R$.
The above pair $(V,H)$ is said to be an $\mathcal{H}$-vertex algebra
if the following conditions hold:
\begin{enumerate}
\item[VH1)]
$h(0)$ is semisimple on $V$ with real eigenvalues for any $h\in H$;
\item[VH2)]
For any $\al \in H$,
there exists $N \in \Z$ such that 
$V_{t}^\al=0$ for any $t \leq N$.
\end{enumerate}
By Proposition \ref{graded_vertex},
an $\mathcal{H}$-vertex algebra is a full $\mathcal{H}$-vertex algebra.
A category of $\mathcal{H}$-vertex algebras
is a full subcategory of $\FHp$ whose objects are $\mathcal{H}$-vertex algebras.
We denoted  the category of $\mathcal{H}$-vertex algebras by $\VHp$.
We also denote the category of generalized vertex algebras by $\GVA$.

Let $(V,H)$ be an $\mathcal{H}$-vertex algebra.
Then, by the proof of Theorem \ref{vacuum_space},
$(\Om_{V,H},H)$ is a generalized vertex algebra.
Furthermore, the charge structure of $\Om_{V,H}$ is the identical projection $\mathrm{id}_H \in \End\;H$
since all vectors in $V$ are chiral.
Thus, $(V,H)$ can be recovered from $\Om_{V,H}$.
Let $V: \GVA \rightarrow \VHp$ be the functor defined by $V_{\Om,H}=F_{\Om,H,\mathrm{id}_H}$
for a generalized vertex algebra $(\Om,H)$.
Then, we have:
\begin{prop}\label{vacuum_vertex}
The restriction of the functor $\Om: \VHp \rightarrow \GVA$ gives an equivalence of the categories
and the inverse functor is given by $V: \GVA \rightarrow \VHp$.
\end{prop}

\section{Current-current deformation}
In this section, we define and study current-current deformations of full $\mathcal{H}$-vertex algebras.

Let $(F,H,p_0)$ be a full $\mathcal{H}$-vertex algebra.
For $p\in P(H)$, 
set $F_p=G_{H,p} \otimes_{\Delta H} \Om_{F,H}$.
Then, by Theorem \ref{construction}, $F_p$ is a full $\mathcal{H}$-vertex algebra.
Thus, we have a family of full $\mathcal{H}$-vertex algebras parametrized by $P(H)$.
By Corollary \ref{decomposition}, $F_{p_0}$ is isomorphic to $F$ as a full $\mathcal{H}$-vertex algebra.

Let $O(H;\R)$ be the orthogonal group of the real vector space $(H,(-,-)_\lat)$.
Then, $O(H;\R)$ acts on $P(H)$ by
$\si \cdot p=\si p \si^{-1}$ for $\si \in O(H;\R)$ and $p \in P(H)$.
From the elementary linear algebra, the following lemma follows:
\begin{lem}\label{transitive_proj}
For projections $p,p' \in P(H)$, the following conditions are equivalent:
\begin{enumerate}
\item
There exits $\si \in O(H;\R)$ such that $\si \cdot p=p'$.
\item
The signature of the real spaces $\ker p$ and $\ker p'$ are the same.
\end{enumerate}
\end{lem}
Thus, the $O(H;\R)$ orbit of $p_0 \in P(H)$
is equal to the orthogonal Grassmannian
$$
O(H;\R)/O(H_l;\R)\times O(H_r;\R),
$$
which is the connected component of $P(H)$ containing $p_0$.

We call the family of full $\mathcal{H}$-vertex algebras
$\{F_{\si \cdot p_0}  \}_{\si \in O(H;\R)}$
the {\it current-current deformation} of
the full $\mathcal{H}$-vertex algebra $(F,H,p_0)$.

By Corollary \ref{vac_omega} and Theorem \ref{construction},
we have:
\begin{prop}
If $F$ is a full $\mathcal{H}$-conformal vertex algebra,
then a current-current deformation of $F$ also has an energy-momentum tensor.
\end{prop}

A full $\mathcal{H}$-vertex algebra is called positive
if both $(H_l, (-,-)_l)$ and $(H_r, (-,-)_r)$ are positive-definite.
The following proposition says that the
compactness of conformal field theory is preserved by the current-current deformation under some mild assumption.
\begin{prop}\label{positive_compact}
Let $(F,H,p_0)$ be a full $\mathcal{H}$-vertex algebra such that $(\Om_{F,H})_{t,\td}^\al =0$
for any $t \leq 0$ or $\td \leq 0$ and any $\al \in H$.
If $F$ is positive and compact,
then a current-current deformation of $F$
is also positive and compact.
\end{prop}
\begin{proof}
Let $\si \in O(H;\R)$.
By Lemma \ref{transitive_proj}, $F_{\si \cdot p_0}$ is a positive full $\mathcal{H}$-vertex algebra.
Since for any $\al,\be \in H$, $(\si p_0 \si^{-1} \al, \si p_0 \si^{-1} \be)_\lat = (p_0 \si^{-1} \al, p_0 \si^{-1} \be)_\lat$,
\begin{align*}
F_{\si \cdot p_0}&= G_{H,\si \cdot p_0}\otimes_{\Delta H} \Om_{F,H}\\
&= \bigoplus_{\al \in H} 
M_{H, \si \cdot p_0}(\al) \otimes \Om_{F,H}^\al\\
&= \bigoplus_{\al \in H} 
M_{H, p_0}(\si^{-1} \cdot \al) \otimes \Om_{F,H}^\al.
\end{align*}
Thus, by the positivity and the assumption,
$(F_{\si \cdot p_0})_{h,\h}=0$ unless $h,\h \geq 0$,
thus the spectrum of $F_{\si \cdot p_0}$ is bounded below.

Let $N \in \R$.
It is easy to show that
$\sum_{h+\h < N}\dim (F_{\si \cdot p_0})_{h,\h}
< \infty$ if and only if
$\sum_{t,\td,\al} \dim (\Om_{F,H})_{t,\td}^\al < \infty$,
where in the sum $t,\td \in \R$ and $\al \in H$ satisfy $t+\td+\frac{1}{2}(\si^{-1}\al,\si^{-1}\al)_l
+\frac{1}{2}(\si^{-1}\al,\si^{-1}\al)_r < N$.
Set $||\al||=\frac{1}{2}(\al,\al)_l+\frac{1}{2}(\al,\al)_r $ for $\al \in H$.
Since $\si \in \mathrm{GL}(H)$,
by an elementary linear algebra,
there exists $k_\si \in \R_{>0}$ such that
$k_\si ||\al||<||\si^{-1}\al||$ for any $\al \in H$.
We may assume that $0<k_\si <1$.
Then, for any $\al \in H$ and $t,\td \geq 0$,
$$||\si^{-1}\al||+t+\td
> k_\si \Bigl(||\al||+ \frac{1}{k_\si}(t+\td) )\Bigr)
> k_\si \Bigl(||\al||+ t+\td)\Bigr).$$
Thus, the spectrum of $F_{\si\cdot p_0}$
is discrete since that of $F_{p_0}$ is discrete.
Hence, $F_{\si\cdot p_0}$ is compact.
\end{proof}
\begin{rem}
It seems that for any unitary compact conformal field theory
the assumption in the above proposition is satisfied.
We conjecture that the unitary compact conformal field theory
is stable under exactly marginal deformations.
\end{rem}

\subsection{Physical meaning of deformation}\label{sec_physics}
In this section, we discuss a relation between the
current-current deformation of a full $\mathcal{H}$-vertex algebra
and an exactly marginal deformation in physics.
Let $(F,H,p)$ be a full $\mathcal{H}$-vertex algebra and $h_l \in \ker \p$
and $h_r \in \ker p$ satisfy $(h_l,h_l)_\lat=1$ and $(h_r,h_r)_\lat=-1$.
Set $H^\perp = \{h\in H\;|\;(h,h_l)_\lat=0, (h,h_r)_\lat=0 \}$
and define a group homomorphism $\si:\R \rightarrow O(H;\R)\; g \mapsto \si(g)$ by
\begin{align*}
\begin{cases}
\si(g) |_{H^\perp} &=\mathrm{id}, \cr
\si(g)(h_l) &=\mathrm{cosh}(g)h_l + \mathrm{sinh}(g)h_r, \cr
\si(g)(h_r) &= \mathrm{cosh}(g)h_r + \mathrm{sinh}(g)h_l.
\end{cases}
\end{align*}
It is believed that a quantum field theory can be deformed by adding a new field to the Lagrangian (see Introduction).
We can show that the deformation family $\{F_{\si(g)\cdot p}\}_{g\in \R}$
corresponds to
the deformation by the $(1,1)$-field $Y(h_l(-1,-1)h_r, \uz)=h_l(z)h_r(\z)$
by using the path-integral. This is why we call this deformation the current-current deformation.

\subsection{Double coset description}
In this section, we gives a double coset description of the parameter space of a current-current deformation.
Let $(F,H,p)$ be a full $\mathcal{H}$-vertex algebra and let $(\psi,\psi')$ an automorphism of a generalized full vertex algebra $(\Om_{F,H},H)$.
Then, $\psi' \in O(H;\R)$. Thus, we have a group homomorphism
$\Aut (\Om_{F,H},H) \rightarrow O(H;\R)$
from the group of generalized full vertex algebra automorphisms
to the orthogonal group.
Denote the image of this map by $D_{F,H} \subset O(H)$,
which we call a {\it duality group}.
We note that $(\psi,\psi') \in \Aut (\Om_{F,H},H)$ lifts to a full vertex algebra automorphism
if and only if it preserves the charge structure, that is,  $\psi' \cdot p=p$.
The following theorem follows from Theorem \ref{equivalence}:
\begin{thm}\label{moduli}
For $p, p'\in P(H)$,
$F_p$ and $F_{p'}$ are isomorphic as full $\mathcal{H}$-vertex algebras
if and only if there exists $\si \in D_{F,H}$ such that $\si \cdot p=p'$.
In particular,
there is a bijection between the isomorphism classes of the current-current deformations
of $(F,H,p)$ and the double coset
$$D_{F,H}\backslash O(H;\R)/O(H_l;\R)\times O(H_r;\R).$$
\end{thm}

\subsection{Example: Toroidal Compactification}\label{sec_toroidal}
Let $L$ be an even non-degenerate lattice of signature $(n,m)$ and $H=L\otimes_\Z \R$.
Then, we have a lattice full vertex algebra $F_{L,H,p}$ for any $p \in P_>(H)$.
Since $D_{F_{L,H,p},H}$ is isomorphic to the lattice automorphism group, $\Aut\,L$,
the isomorphism classes is 
$$\Aut L \backslash O(n,m;\R)/O(n;\R)\times O(m;\R).
$$

Let $\tw= \Z z\oplus \Z w$ be the rank two even lattice defined by $(z,z)=(w,w)=0$ and $(z,w)=-1$.
Then, $\tw$ is a unique even unimodular lattice of signature $(1,1)$.
Set $\twN=\tw^{\oplus k}$ for $k\geq 1$.
The lattice full vertex algebras $\{F_{\twN,\twN\otimes_\Z \R,p}\}_{p\in P_{>}(\twN\otimes_\Z \R)}$ appear in the toroidal compactification of string theory (see for example \cite{P}),
which is parametrized by
\begin{align}
O(k,k;\Z) \backslash O(k,k;\R)/O(k;\R)\times O(k;\R).
\label{eq_Narain_k}
\end{align}

In the rest of this section, we explicitly describe the action of the duality group $O(k,k;\Z)$ in detail in the case of $k=1$.
Set $H_\tw=\tw\otimes_\Z \R$.
Let $p \in P_>(H)$.
Since $\ker \p$ is positive-definite, there is a unique (up to
the multiplication by $\pm 1=O(1;\R)$) vector $v \in \ker \p$
such that $(v,v)=1$. It is clear that $p$ is uniquely determined by this vector.
Let $v=az+bw \in H_\tw$ be a norm $1$ vector ($a,b\in \R$).
Then, by $(v,v)=-2ab$,
we may assume that $v=\frac{1}{\sqrt{2}}(R z-R^{-1} w)$ for $R \in \R_{>0}$.
Denote by $p_R$ the corresponding projection in $P_>(H_\tw)$.
Thus, we have an isomorphism
$\R_{>0} \rightarrow O(1,1;\R)/O(1;\R)\times O(1;\R),\; R \mapsto p_R$.
The lattice automorphism group $\Aut \tw$ is $\Z/2\Z\times \Z/2\Z$,
which is generated by the involutions $\si,\tau$ such that:
\begin{align*}
\si(z)=w, \si(w)=z,\\
\tau(z)=-z, \tau(w)=-w.
\end{align*}
The action of $\si$ on $p_R$ is determined by
$$\si(\frac{1}{\sqrt{2}}(R z-R^{-1} w))=- \frac{1}{\sqrt{2}}(R^{-1} z-R w).
$$
Hence, $\si \cdot p_R =p_{R^{-1}}$.
Since $\tau \in O(1)\times O(1) \subset O(1,1)$,
$$\Aut \tw \backslash O(1,1)/O(1)\times O(1) \cong \R_{\geq 1}.$$
In the string theory, $R$ is a radius of a compactification of the target space.
Denote by $C_{R}$ the full vertex algebra $F_{\tw,H_\tw,p_R}$.
The isomorphism $\tilde{\si}:C_R \rightarrow C_{R^{-1}}$
is called a {\it T-duality} of string theory.
Let $R=e^s$ for $s \in \R$.
Then, the action of the $1$-parameter deformation $\si(g)$
associated with $h_l=\frac{1}{\sqrt{2}}(e^s z-e^{-s} w), h_r=\frac{1}{\sqrt{2}}(e^s z+e^{-s} w)$
is  
\begin{align*}
\si(g)\Bigl(\frac{1}{\sqrt{2}}(e^s z-e^{-s} w)\Bigr)
&=\frac{1}{\sqrt{2}}(\mathrm{cosh}(g)(e^s z-e^{-s} w) +\mathrm{sinh}(g)(e^s z+e^{-s} w))\\
&=\frac{1}{\sqrt{2}}(e^{s+g}z-e^{-s-g}w).
\end{align*}
Thus, $\si(g)$ changes the radius $R=e^s$ into $e^gR= e^{g+s}$.

We end this section by studying the chiral vertex algebra $\ker \D$ of a full vertex algebra $C_R$.
It is easy to show that the conformal weight of $e_{nz+mw} \in \C[\hat{\tw}]$ is $(\frac{(nR^{-1}-mR)^2}{4},\frac{(nR^{-1}+mR)^2}{4})$ for $n,m\in \Z$.
The state $e_{nz+mw}$ is in $\ker \D$ if and only if
$R^2=-\frac{n}{m}$.
Thus, if $R^2 \in \R \setminus \mathbb{Q}$,
$\ker \D \otimes \ker D$ is isomorphic to the affine Heisenberg full vertex algebras $M_{H_\tw, p_R}$.
We assume that $R^2=\frac{p}{q}$ for some coprime intergers $p,q \in \Z_{> 0}$.
In this case,
$$\ker \D = M_{\ker \p_R} \otimes \bigoplus_{k \in \Z} \C e_{k(pz-qw)}.$$
Since the conformal weight of $e_{k(pz-qw)}$ is $(pqk^2,0)$,
$\ker \D$ is isomorphic to the lattice vertex algebra $V_{\sqrt{2pq}\Z}$
associated with the rank one lattice $\sqrt{2pq}\Z$.
In particular, $C_{\sqrt{\frac{p}{q}}}$ is a finite extension of 
the lattice full vertex algebra $V_{\sqrt{2pq}\Z} \otimes \bar{V}_{\sqrt{2pq}\Z}$.
We will determine the irreducible decomposition of $C_{\sqrt{\frac{p}{q}}}$
as a $V_{\sqrt{2pq}\Z} \otimes \overline{V}_{\sqrt{2pq}\Z}$-module.
We recall that there are $2pq$ irreducible modules of $V_{\sqrt{2pq}\Z}$,
denoted by $\{V_{\sqrt{2pq}\Z+ \frac{i}{\sqrt{2pq}}}\}_{i \in \Z/2pq\Z}$, see for example \cite{LL}.
Since 
\begin{align*}
\Bigl(p_R (pz-qw),p_R( nz+mw)\Bigr)&= nq-mp, \\
-\Bigl(\p_R (pz+qw), \p_R( nz+mw)\Bigr)&= nq+mp,
\end{align*}
$e_{nz+mw}$ is contained in $V_{\sqrt{2pq}\Z+ \frac{nq-mp}{\sqrt{2pq}}} \otimes
\overline{V}_{\sqrt{2pq}\Z+ \frac{nq+mp}{\sqrt{2pq}}}$.

We will use the following elementary lemma:
\begin{lem}
Let $(a,b) \in \Z^2$ satisfy $ap -b q =1$.
Then, $n_{p,q}=ap +b q$ satisfies
$n_{p,q}^2 = 1 \in \Z/4pq\Z$,
in particular, $n_{p,q} \in (\Z/2pq\Z)^\times$.
Furthermore, the value $n_{p,q}=ap+bq \in \Z/2pq\Z$ is independent of a choice of the solution.
\end{lem}
Since $n_{p,q} \in (\Z/2pq\Z)^\times$, 
$\{kn_{p,q} \}_{k=0,1,\dots,2pq-1}$ runs through all the elements in 
$\Z/2pq\Z$.
Thus, we have:
\begin{prop}\label{twist}
If $R^2 \in \R \setminus \mathbb{Q}$,
then $\ker \D \otimes \ker D$ is isomorphic to the affine Heisenberg full vertex algebras $M_{H_\tw, p_R}$.
If $R^2=\frac{p}{q}$,
then $\ker \D \otimes \ker D$ is isomorphic to the lattice full vertex algebra $V_{\sqrt{2pq}\Z} \otimes \overline{V}_{\sqrt{2pq}\Z}$
and the irreducible decomposition of $C_{\sqrt{\frac{p}{q}}}$ is
$$
C_{\sqrt{\frac{p}{q}}}
=\bigoplus_{i \in \Z/2pq \Z} V_{\sqrt{2pq}\Z+ \frac{i}{\sqrt{2pq}}} \otimes
\bar{V}_{\sqrt{2pq}\Z+ \frac{n_{p,q} i}{\sqrt{2pq}}}.
$$
\end{prop}
We remark that the condition $n_{p,q}^2 = 1 \in \Z/4pq\Z$ corresponds the condition (FV2).
Thus, for $N \in \Z_{>0}$ and each order $2$ element in $(\Z/ 4N\Z)^\times$,
there is an extension of the lattice full vertex algebra $V_{\sqrt{2N}\Z} \otimes \bar{V}_{\sqrt{2N}\Z}$.

For example, $C_{\sqrt{6}}$ is the diagonal model
$\bigoplus_{i \in \Z/12 \Z} V_{\sqrt{12}\Z+ \frac{i}{\sqrt{12}}} \otimes
\bar{V}_{\sqrt{12}\Z+ \frac{i}{\sqrt{12}}}$,
whereas 
$C_{\sqrt{\frac{2}{3}}}$ is twisted by $7$,
$\bigoplus_{i \in \Z/12 \Z} V_{\sqrt{12}\Z+ \frac{i}{\sqrt{12}}} \otimes
\bar{V}_{\sqrt{12}\Z+ \frac{7i}{\sqrt{12}}}.$
We also remark that $C_{1}$ is isomorphic to the
$\mathrm{SU}(2)$ WZW-model of level $1$,
which is the fixed point of the duality group.

\begin{rem}
Let $q:\Z / 2N\Z \rightarrow \R/2\Z $ be a norm defined by
$a \mapsto \frac{a^2}{2}$.
An element $n \in \Aut \Z/ 2N\Z=(\Z/2N\Z)^\times$ preserves
the norm $q$ if and only if $n^2 =1$ in $\Z/4N\Z$.
Thus, an order $2$ element in $(\Z/ 4N\Z)^\times$ corresponds to
the outer automorphism of the modular tensor category.
\end{rem}

\section{Current-current deformations and genera of vertex algebras}

In \cite{M1}, we introduced an equivalence class for $\mathcal{H}$-vertex algebras, which we call a {\it genus of $\mathcal{H}$-vertex algebras} \cite{M1}. Although this notion of genera was introduced independently of current-current deformations, they are closely related to each other. In this section we explain the relationship between the genera and the current-current deformation.

The notion of a genus of $\mathcal{H}$-vertex algebras is motivated by the notion of a genus of lattices, and we first review the genus of lattices in the next section. In Section \ref{sec_genus_vertex} we review the definition of a genus of $\mathcal{H}$-vertex algebras, and we discuss the relationship between the genus and the current-current deformations (Theorem \ref{genus}).
In Section \ref{sec_genus_vertex_lattice} and \ref{sec_genus_mass}, we study the duality group based on \cite{M1}, and in Section \ref{sec_genus_hol} we obtain nontrivial duality groups from holomorphic vertex operator algebras of central charge $24$ (Theorem \ref{prop_new_grass}).

\subsection{Genus and mass of lattices}
Let $L$ be an integral lattice.
For a unital commutative ring $R$,
we can extend the bilinear form $(\;\,,\;)$ on $L$ bilinearly to $L\otimes_{\Z} R$.
The dual of $L$ is the set 
$$L^\vee=\{\al \in L\otimes_{\Z}\R \;|\;(\al,L)\subset \Z \}.$$
The lattice $L$ is said to be unimodular if
$L = L^\vee$.

Two integral lattices $L$ and $M$ are said to be equivalent or {\it in the same genus} if 
$$L\otimes_{\Z}\mathbb{R}\simeq M\otimes_{\Z}\mathbb{R},\quad L\otimes_{\Z}\Z_p\simeq M\otimes_{\Z}\Z_p,$$
for all the prime integers $p$,
where $\Z_p$ is the ring of $p$-adic integers.
Denote by $\gen(L)$ the genus of lattices which contains $L$.
If $L$ is positive-definite, then a mass of $\gen(L)$ is defined by
\begin{align}\label{eq: mass formula for lattices}
\text{mass}(L) = \sum_{L' \in \gen(L)} \frac{1}{\# \Aut L'} \in \mathbb{Q},
\end{align}
where $\Aut L'$ is the automorphism group of the lattice $L'$.

Lattices over $\mathbb{R}$ are completely determined by the signature.
Similarly, lattices over $\Z_p$ are determined by some invariant, called $p$-adic signatures
(If $p=2$, we have to consider another invariant, called an oddity).
The Smith-Minkowski-Siegel's mass formula is a formula which computes $\text{mass}(L)$ 
by using those invariants (see \cite{Si,Mi,CS,Ki}).

Consider the unique even unimodular lattice $\tw$ of signature $(1,1)$. 
The proof of the following lemma can be found in \cite{KP,Ni}:
\begin{lem}
\label{def_genus}
The lattices $L_1$ and $L_2$ are in the same genus if and only if 
$$L_1 \otimes \tw\simeq L_2 \otimes \tw$$ 
as lattices.
\end{lem}

\subsection{Genus of vertex algebra and current-current deformation}
\label{sec_genus_vertex}
In the previous section, we recall the notion of a genus of lattices,
which is an equivalence relation of lattices
and important to classify lattices.
By using Lemma \ref{def_genus},
we generalize it and define a genus of $\mathcal{H}$-vertex algebra.

Let us consider the lattice vertex algebra $V_\tw$ associated with the rank 2 lattice $\tw$ (see Section \ref{sec_toroidal})
and let $(V,H)$ be an $\mathcal{H}$-vertex algebra.
Then, by Proposition \ref{tensor},
$V\otimes C_s$ is a full $\mathcal{H}$-vertex algebra and 
$V \otimes V_\tw$ is an $\mathcal{H}$-vertex algebra.

$\mathcal{H}$-vertex algebras $(V,H)$ and $(V',H')$ are said to be equivalent (or in the same genus) if
$(V\otimes V_\tw,H\oplus H_\tw)$ and $(V' \otimes V_\tw,H'\oplus H_\tw)$
are isomorphic as $\mathcal{H}$-vertex algebras,
which defines an equivalent relation on $\mathcal{H}$-vertex algebras.
An equivalent class is called a {\it genus} of $\mathcal{H}$-vertex algebras \cite{M1}.
The equivalent classes of an $\mathcal{H}$-vertex algebra $(V,H)$
is denoted by $\gen(V,H)$ or $\gen(V)$ for short.
\begin{thm}\label{genus}
Let $(V,H)$ and $(V',H')$ be $\mathcal{H}$-vertex algebras.
Then, the following conditions are equivalent:
\begin{enumerate}
\item
$\mathcal{H}$-vertex algebras $(V,H)$ and $(V',H')$ are in the same genus;
\item
There exits a current-current deformation between the full $\mathcal{H}$-vertex algebras $V\otimes C_s$ and $V' \otimes C_s$;
\item
Generalized full vertex algebras $(\Om_{V,H}\otimes \C[\hat{\tw}],H\oplus H_\tw)$ and $(\Om_{V',H'} \otimes \C[\hat{\tw}],H\oplus H_\tw)$
are isomorphic as generalized full vertex algebras.
\end{enumerate}
\end{thm}


\begin{proof}[proof of Theorem \ref{genus}]
Since the vacuum spaces of
$V\otimes C_s$ and $V \otimes V_\tw$ 
are isomorphic to $\Om_{V,H} \otimes \C[\hat{\tw}]$,
(1) or (2) implies (3).
Assume that (3) holds.
Since all fields in $V\otimes V_\tw$ and $V' \otimes V_\tw$
are holomorphic,
they are isomorphic to $F_{\Om_{V,H} \otimes \C[\hat{\tw}],H\oplus H_\tw,\mathrm{id}}$,
where $\mathrm{id} \in P(H\oplus H_\tw)$ is the identity map.
Similarly,
by Lemma \ref{transitive_proj},
the projections which define $V\otimes C_s$ and $V'\otimes C_s$
is in the same orbit of $O(H\oplus H_\tw;\R)$
since the signature of the anti-holomorphic part $\ker p$ must be $(0,1)$.
Hence, (3) implies (1) and (2).
\end{proof}

\begin{rem}
\label{rem_genus_motivation}
The phenomenon that two vertex algebras which are not isomorphic to each other become isomorphic by tensoring $V_\tw$ corresponds to the notion of the genus of lattices from Lemma \ref{def_genus} for lattice vertex algebras.
It was found in \cite{HS} that such a phenomenon occurs even in the case of non-lattice vertex algebras.
Together with another example \cite{HS2}, we introduced the notion of the genus of $\mathcal{H}$-vertex algebras as an analogue of the notion of the genus of lattices \cite{M1}.
Theorem \ref{genus} says that such isomorphisms can be regarded as equivalence relations by the current-current deformations of conformal field theory.
\end{rem}

\subsection{From vertex algebra to lattice}
\label{sec_genus_vertex_lattice}
In this section, we construct an even $H$-lattice from an $\mathcal{H}$-vertex algebra $(V,H)$.
Let $(V,H)$ be an $\mathcal{H}$-vertex algebra and
$(\Om_{V,H},H)$ the generalized vertex algebra constructed in Proposition \ref{vacuum_vertex}
and $A_{\Om_{V,H}}$ the even AH pair constructed in Proposition \ref{omega_AH}.
The $\mathcal{H}$-vertex algebra $(V,H)$ is good
if $A_{\Om_{V,H}}$ is good.
By the following lemmas, almost all natural $\mathcal{H}$-vertex algebras are good:
\begin{lem}\label{good_criterion}
If $V$ is a simple vertex algebra and $V_0^0=\C\va$,
then $A_{\Om_{V,H}}$ is a good AH pair.
\end{lem}
\begin{proof}
(GAH1) follows from $V_0^0=\C\va$.
Let $a \in A_{\Om_{V,H}}^\al$ and $b \in A_{\Om_{V,H}}^\be$ be non-zero vectors for some $\al,\be\in H$.
Then, $ab \neq 0$ if and only if $\hY(a,z)b \neq 0$.
By the definition of $\hY(-,z)$, $ab \neq 0$ if and only if $Y(a,z)b\neq 0$.
Thus, by Lemma \ref{zero_divisor}, (GAH2) holds.
\end{proof}
\begin{lem}[Lemma 3.13 in \cite{M1}]\label{stable_good}
Let $(V,H)$ be an $\mathcal{H}$-vertex algebra.
Then, $A_{\Om_{V\otimes V_\tw, H\oplus H_\tw}}$ is
isomorphic to $A_{\Om_{V,H}}\otimes \C[\hat{\tw}]$ as an even AH pair.
In particular, $A_{\Om_{V\otimes V_\tw, H\oplus H_\tw}}$ is good if and only if
$A_{\Om_{V,H}}$ is good.
\end{lem}

Let $(V,H)$ be a good $\mathcal{H}$-vertex algebra.
Then, by Proposition \ref{lattice_functor}, we have the lattice pair $(A_{\Om_{V,H}}^\lat,H)$
and the even $H$-lattice $L_{\Om_{V,H},H}$.
Set $L_{V,H}=L_{\Om_{V,H},H}$.
By Proposition \ref{isom_lattice}, $A_{\Om_{V,H}}^\lat$ is isomorphic to
the twisted group algebra $\C[\hat{L_{V,H}}]$.
Since $\C[\hat{L_{V,H}}]$ is a subalgebra of the even AH pair $A_{\Om_{V,H}}$,
by the equivalence of categories the lattice vertex algebra $V_{L_{V,H}}$ is
a subalgebra of $V$ as an $\mathcal{H}$-vertex algebra.
This lattice subalgebra has the following universal property:
\begin{prop}
For any even $H$-lattice $M \subset H$ and an $\mathcal{H}$-vertex algebra homomorphism $\phi: V_M \rightarrow V$,
$$
\xymatrix{
V_{M} \ar[r]^{\phi} \ar@{-->}[d]_{\exists!}
 & V \\
V_{L_{V,H}} \ar[ru]^{i} & {}
}.
$$
\end{prop}
\begin{proof}
By using adjoint functors, we have
\begin{align*}
\Hom_{\VHp}(V_M,V)&\cong \Hom_{\GVA}(\C[\hat{M}],\Om_{V,H})\\
&\cong \Hom_{\AH}(\C[\hat{M}],A_{\Om_{V,H}})\\
&\cong \Hom_{\gAH}(\C[\hat{M}],A_{\Om_{V,H}})\\
&\cong \Hom_{\LP}(\C[\hat{M}],A_{\Om_{V,H}}^\lat)\\
&\cong \Hom_{\LP}(\C[\hat{M}],\C[\hat{L_{V,H}}]).
\end{align*}
\end{proof}

Let $\Aut (V,H)$ the $\mathcal{H}$-vertex algebra automorphism group of $(V,H)$, that is, 
$$
\Aut (V,H)=\{f\in \Aut (V)\;|\; f(H)=H \}.
$$
Then, similarly to Section \ref{sec_equiv},
there is a group homomorphism $\Aut (V,H) \rightarrow O(H;\R)$.
Then, by the equivalence of categories,
we have:
\begin{lem}
For an $\mathcal{H}$-vertex algebra $(V,H)$,
$\Aut (V,H)$ is isomorphic to the automorphism group of the generalized vertex algebra $(\Om_{V,H},H)$.
\end{lem}

By construction, the group $\Aut (V,H)$ acts on the lattice pair $A_{\Om_{V,H}}^\lat$.
Thus, we have a group homomorphism $\Aut (V,H) \rightarrow \Aut (L_{V,H})$,
where $\Aut (L_{V,H})$ is the lattice automorphism group.
The image of $\Aut ({V,H})$ in $\Aut(L_{V,H})$ is denoted by $G_{V,H}$.
The following lemma is clear from the definition:
\begin{lem}\label{duality_inj}
If $L_{V,H}$ is a free abelian group of rank equal to $\dim_\C H$,
then $G_{V,H}$ is equal to the duality group $D_{V,H}$ in $\mathrm{O}(H;\R)$.
\end{lem}

\subsection{Mass formula}
\label{sec_genus_mass}
In this section, we recall the mass formula \cite{M1}.
An $\mathcal{H}$-vertex algebra $(V,H)$ is called positive if
$(H,(-,-))$ is positive-definite.
We note that since $H$ is positive-definite, $L_{V,H}$ is a positive-definite lattice and
$\Aut (L_{V,H})$ and $G_{V,H}$ are finite groups.

%

Let $(V,H)$ be a good positive-definite $\mathcal{H}$-vertex algebra.
By Lemma \ref{stable_good}, all $\mathcal{H}$-vertex algebras in the genus $\text{mass}(V,H)$
are good and positive-definite.
The mass of $\gen(V,H)$ is a rational number defined by
$$
\text{mass}(V,H) = \sum_{(W,H_W) \in \gen(V,H)} \frac{1}{\# G_{W,H_W}}.
$$
In \cite[Theorem 5.19]{M1}, we prove the following result:
\begin{thm}\label{mass_formula}
Let $(V,H)$ be a simple positive-definite $\mathcal{H}$-vertex algebra with $V_0^0=\C\va$.
If the index of the groups $[\Aut(L_{V,H}\oplus \tw): G_{V\otimes V_\tw, H\oplus H_\tw}]$ is finite,
then 
$$\frac{\text{mass}(V,H)}{\text{mass}(L_{V,H})} = [\Aut(L_{V,H}\oplus \tw): G_{V\otimes V_\tw, H\oplus H_\tw}].$$
\end{thm}

Thus, all the isomorphism classes of simple positive-definite $\mathcal{H}$-vertex algebras produced by the current-current deformation can be counted by the mass formula.

\subsection{Application to holomorphic VOA of central charge 24}
\label{sec_genus_hol}
A simple regular VOA is said to be holomorphic if it has only one irreducible module.
For a positive-definite even lattice $L$,
the corresponding lattice VOA $V_L$ is holomorphic if and only if $L$ is unimodular.
The genus of lattices and the mass formula are useful tools for classifying positive-definite even unimodular lattices (see \cite{CS}).
The genus of $\mathcal{H}$-vertex algebras, an analogue of the genus of lattices, is also useful for the classification of holomorphic VOAs.
In this section we will briefly describe them and their application to the CFT moduli space (see also introduction 0.6).

%
%
%

Let $V=\bigoplus_{n\geq 0}V_n$ be a holomorphic VOA of central charge $24$ and assume that $V_1\neq 0$.
Then, $V_1$ is a reductive Lie algebra. There are $70$ possible Lie algebra structures on $V_1$, called Schelleken's list \cite{Sc}.
For a Lie algebra of the list, the existence and uniqueness of the corresponding VOA have been proved under a great deal of researches (see for example \cite{Ho,ELMS,CLM,MSc,BLS} for more detail).

In \cite[Proposition 5.36]{M1}, we prove the following result:
\begin{prop}
Let $V$ be a holomorphic VOA of central charge $24$ and $H_V \subset V_1$ be a Cartan subalgebra of the Lie algebra $V_1$.
Then, for any $(W,H_W) \in \gen(V,H_V)$, $W$ is also a holomorphic VOA of central charge $24$.
\end{prop}
\begin{rem}
The assumption in the above proposition that $V$ is a holomorphic VOA of central charge $24$ is unnecessary and in fact holds for any holomorphic $\mathcal{H}$-vertex algebra (see \cite[Theorem 4.5]{M1}).
\end{rem}

Thus, the $70$ holomorphic VOAs are divided into several genera of $\mathcal{H}$-vertex algebras.
In fact, this classification is consistent with the classification of holomorphic VOA of central charge 24 introduced by H\"ohn. H\"ohn proposed the construction of holomorphic VOAs using vertex operator algebras with group like fusion \cite{Ho}, which was proved by Lam \cite{La}.
From a vertex operator algebra with group like fusion, one can construct several holomorphic VOAs which are not isomorphic in general, however, we show that they fall into the same genus of $\mathcal{H}$-vertex algebras \cite{M7}.
%
%
We will not discuss this result in detail here, but together with \cite{M7} and Theorem \ref{genus}, we have:
\begin{prop}
\label{prop_genera_hol}
The notion of the genera of $\mathcal{H}$-vertex algebras coincides with the H\"ohn's genera for holomorphic VOAs of central charge 24.
There are $11$ genera of holomorphic VOAs of central charge $24$ whose $V_1$ is non-zero.
Moreover, if $V$ and $W$ are in the same genus, then there exists a current-current deformation between
the full VOAs $V\otimes C_s$ and $W\otimes C_s$.
\end{prop}

The parameter space of the current-current deformations of $V\otimes C_s$ in Theorem \ref{moduli} is a part of the CFT moduli space of central charge $(25,1)$.
Hence, it is important to determine {\it the duality groups} for these full VOAs.
The duality groups are generally difficult to determine because they are not finite groups, but by using \cite[Lemma 3.18]{BLS} and \cite[Theorem 5.38]{M1} for types A, B, and G in H\"ohn's notation for the genera, the duality groups can be completely determined.

In the case of type A, i.e., $V$ is a lattice VOA $V_L$ associated with an even unimodular lattice $L$ of rank $24$, the duality group is not interesting,
which is the automorphism group of the lattice ${I\hspace{-.1em}I}_{25,1}$, where ${I\hspace{-.1em}I}_{25,1}$ is the unique even unimodular lattice of signature $(25,1)$, and the parameter space is
\begin{align*}
\Aut {I\hspace{-.1em}I}_{25,1} \backslash \mathrm{O}(25,1;\R) / \mathrm{O}(25;\R)\times \mathrm{O}(1;\R).
\end{align*}
In other words, it is the same as the Narain moduli spaces \eqref{eq_Narain_k}.

Let $V_B$ (resp. $V_G$) be a holomorphic VOA of central charge $24$ of type B (resp. of type G),
and let $\genus_{17,1}(2_{\genus}^{+10})$ (resp. $\genus_{9,1}(2_\genus^{+6}3^{-6})$) be the unique even lattice of signature $(17,1)$ (resp. $(9,1)$) with the given discriminant form $2_{\genus}^{+10}$ (resp. $2_\genus^{+6}3^{-6}$) (see \cite{CS} for the notation).
Then, by Lemma \ref{duality_inj}, \cite[Lemma 3.18]{BLS} and the proof of \cite[Theorem 5.38]{M1}, we have:
\begin{thm}
\label{prop_new_grass}
The duality group $D_{V_{B}\otimes C_s}$ and $D_{V_{G}\otimes C_s}$ are isomorphic to the automorphism groups of the lattices $\genus_{17,1}(2_{\genus}^{+10})$ and $\genus_{9,1}(2_\genus^{+6}3^{-6})$, respectively.
In particular, the current-current deformations of 
$V_{B}\otimes C_s$ and $V_{G}\otimes C_s$
are respectively parametrized by
\begin{align*}
\Aut \genus_{17,1}(2_{\genus}^{+10}) \backslash O(17,1;\R) /O(17;\R) \times O(1;\R)
\end{align*}
and
\begin{align*}
\Aut \genus_{9,1}(2_\genus^{+6}3^{-6}) \backslash O(9,1;\R) /O(9;\R) \times O(1;\R).
\end{align*}
\end{thm}

\vspace*{5mm}

\begin{center}
{\large \bf Acknowledgements
}
\end{center}
 \par \bigskip
First of all, I would like to offer my gratitude to my supervisor
Professor Masahito Yamazaki
for his great instruction, support and encouragement. I also wish to express my gratitude to Professor Yuji Tachikawa for valuable discussions and his suggestion to study the toroidal compactification of string theory, which is the starting point of this work,
and to Professor Atsushi Matsuo, who is my former supervisor, for his encouragement and valuable comments.
I would also like to express my thanks to Professor Ching Hung Lam for answering my questions about holomorphic VOAs and to the referee for the valuable comments that helped to improve the paper.

This work was supported by World Premier International Research Center Initiative 
(WPI Initiative), MEXT, Japan. The author was also supported by the Program for Leading Graduate Schools, MEXT, Japan.

\end{document}